\documentclass{article}
\usepackage[dvipdfm,truedimen,left=25truemm,right=25truemm,top=30truemm,bottom=30truemm]{geometry}
\usepackage{amsmath,amssymb}
\usepackage{graphicx}
\usepackage{tikz}
\usepackage{booktabs}
\usepackage{pinlabel}
\usepackage{ytableau}
\usepackage{braket}

\usetikzlibrary{arrows.meta, angles, quotes, calc}
\usetikzlibrary{positioning}
\usepackage{amsthm}
\usepackage{thmtools}
\usepackage{comment}
\usepackage{tikz-3dplot}
\usepackage{ulem}
\usepackage[absolute,overlay]{textpos}
\usepackage[legacycolonsymbols]{mathtools}

\theoremstyle{definition}

\newtheorem{theorem}{Theorem}[section]
\newtheorem*{theorem*}{Theorem}

\newtheorem{definition}[theorem]{Definition}
\newtheorem*{definition*}{Definition}
\newtheorem{prop}[theorem]{Proposition}
\newtheorem{lemma}[theorem]{Lemma}
\newtheorem{cor}[theorem]{Corollary}
\newtheorem{example}[theorem]{Example}
\newtheorem{rem}[theorem]{Remark}
\newtheorem{notation}[theorem]{Notation}

\numberwithin{equation}{section}

\DeclareFontFamily{U}{mathx}{}
\DeclareFontShape{U}{mathx}{m}{n}{<-> mathx10}{}
\DeclareSymbolFont{mathx}{U}{mathx}{m}{n}
\DeclareMathAccent{\widehat}{0}{mathx}{"70}
\DeclareMathAccent{\widecheck}{0}{mathx}{"71}

\DeclareMathSymbol{\antishriek}{\mathord}{operators}{'74}

\newcommand{\dashededgea}[3]{
\begin{tikzpicture}[x=0.75pt,y=0.75pt,yscale=0.3,xscale=0.3, baseline=-3pt] 

\draw [dash pattern={on 4pt off 3pt}, {#1}]  (10,0)--(90,0);

\draw  (0,0) circle (10) node [anchor=north] {{#2}};
\draw  (100,0) circle (10) node [anchor=north] {{#3}};
\end{tikzpicture}}

\newcommand{\dashededgeb}[3]{ 
\begin{tikzpicture}[x=0.75pt,y=0.75pt,yscale=0.3,xscale=0.3, baseline=-3pt] 

\draw [dash pattern={on 4pt off 3pt}, {#1}, line width = 1pt]   (0,0)--(90,0);

\draw  [fill={rgb, 255:red, 0; green, 0; blue, 0 }  ,fill opacity=1 ] (0,0) circle (10) node [anchor=north] {{#2}};
\draw  [fill={rgb, 255:red, 0; green, 0; blue, 0 }  ,fill opacity=1 ] (100,0) circle (10) node [anchor=north] {{#3}};
\end{tikzpicture}}

\newcommand{\dashededgec}[3]{
\begin{tikzpicture}[x=0.75pt,y=0.75pt,yscale=0.3,xscale=0.3, baseline=-3pt] 

\draw [dash pattern={on 4pt off 3pt}, {#1}]   (10,0)--(90,0);

\draw  [fill={rgb, 255:red, 0; green, 0; blue, 0 }  ,fill opacity=1 ] (0,0) circle (10) node [anchor=north] {{#2}} ;
\draw  (100,0) circle (10) node [anchor=north] {{#3}};
\end{tikzpicture}}

\newcommand{\solidedged}[3]{
\begin{tikzpicture}[x=0.75pt,y=0.75pt,yscale=0.3,xscale=0.3, baseline=-3pt] 

\draw  [{#1}, line width =0.8 pt] (0,0)--(90,0);

\draw  [fill={rgb, 255:red, 0; green, 0; blue, 0 }  ,fill opacity=1 ]  (0,0) circle (10) node [anchor=north] {{#2}};
\draw  [fill={rgb, 255:red, 0; green, 0; blue, 0 }  ,fill opacity=1 ] (100,0) circle (10) node [anchor=north] {{#3}};
\end{tikzpicture}}

\newcommand{\solidedgee}[3]{
\begin{tikzpicture}[x=0.75pt,y=0.75pt,yscale=0.3,xscale=0.3, baseline=-3pt] 

\draw  [{#1}] (0,0)--(90,0);

\draw  [fill={rgb, 255:red, 0; green, 0; blue, 0 }  ,fill opacity=1 ]  (-10, -10) rectangle (10, 10);
\draw (0,0) node [anchor=north] {{#2}};
\draw  [fill={rgb, 255:red, 0; green, 0; blue, 0 }  ,fill opacity=1 ] (90,-10) rectangle (110,10);
\draw (100,0)node [anchor=north] {{#3}};
\end{tikzpicture}}

\newcommand{\negativedefectgraph}[0]{
\begin{tikzpicture}[x=1pt,y=1pt,yscale=-0.3,xscale=0.3,baseline=-85pt, line width=1pt]

\draw [dash pattern={on 3pt off 2pt}] [-stealth]  (90,270)..controls (170,220)..(245,265) ;
\draw [dash pattern={on 3pt off 3pt}]  [-stealth] (170,270).. controls (250, 220).. (325,265) ;

\draw [-stealth] (90, 270)--(170,335) ;
\draw [-stealth](170, 340)--(245,270) ;
\draw [stealth-](170, 270)--(170,340) ;

\draw  [fill={rgb, 255:red, 0; green, 0; blue, 0 }  ,fill opacity=1 ] (90,270) circle (5);
\draw (90, 290) node {$1$};
\draw  [fill={rgb, 255:red, 0; green, 0; blue, 0 }  ,fill opacity=1 ] (170,270) circle (5);
\draw (190, 290) node {$2$};
\draw [fill={rgb, 255:red, 0; green, 0; blue, 0 }  ,fill opacity=1 ] (250,270) circle (5);
\draw (250, 290) node {$3$};
\draw [fill={rgb, 255:red, 0; green, 0; blue,0 } ,fill opacity=1] (330,270) circle (5);
\draw (330, 290) node {$4$};
\draw [fill={rgb, 255:red, 0; green, 0; blue,0 } ,fill opacity=1] (168, 330) rectangle (178, 340);
\draw (190, 340) node {$5$};

\end{tikzpicture}}

\newcommand{\negativedefectgraphb}[0]{
\begin{tikzpicture}[x=1pt,y=1pt,yscale=-0.3,xscale=0.3,baseline=-85pt, line width=1pt]

\draw  [-stealth]  (170,265)..controls (210,220)..(245,265) ;
\draw   [-stealth] (170,275).. controls (210, 320).. (245,265) ;
\draw    [-stealth](255, 270)--(325,270) ;

\draw   [dash pattern={on 3pt off 2pt}] [-stealth] (90, 270)--(165,270) ;
\draw    [dash pattern={on 3pt off 2pt}] [-stealth] (325, 270)--(400,270) ;

\draw  [fill={rgb, 255:red, 0; green, 0; blue, 0 }  ,fill opacity=1 ] (90,270) circle (5);
\draw (90, 290) node {$1$};
\draw  [fill = black] (170,270) circle (5);
\draw  (170, 290) node {$3$};
\draw [fill = black]  (250,270) circle (5);
\draw (250, 290) node {$5$};
\draw [fill={rgb, 255:red, 0; green, 0; blue,0 } ,fill opacity=1] (330,270) circle (5);
\draw (330, 290) node {$4$};
\draw [fill={rgb, 255:red, 0; green, 0; blue,0 } ,fill opacity=1] (400,270) circle (5);
\draw (400, 290) node {$2$};

\end{tikzpicture}}

\newcommand{\graphC}[0]{
\begin{tikzpicture}[x=1pt,y=1pt,yscale=-0.3,xscale=0.3,baseline=-85pt, line width=1pt]

\draw  [-stealth]  (170,265)..controls (210,220)..(245,265) ;
\draw   [-stealth] (170,275).. controls (210, 320).. (245,265) ;

\draw   [dash pattern={on 3pt off 2pt}] [-stealth] (90, 270)--(165,270) ;
\draw   [dash pattern={on 3pt off 2pt}] [-stealth](255, 270)--(325,270) ;

\draw  [fill={rgb, 255:red, 0; green, 0; blue, 0 }  ,fill opacity=1 ] (90,270) circle (5);
\draw (90, 290) node {$1$};
\draw  [fill = black] (170,270) circle (5);
\draw  (170, 290) node {$3$};
\draw [fill = black]  (250,270) circle (5);
\draw (250, 290) node {$4$};
\draw [fill={rgb, 255:red, 0; green, 0; blue,0 } ,fill opacity=1] (330,270) circle (5);
\draw (330, 290) node {$2$};

\end{tikzpicture}}

\newcommand{\graphDb}[0]{
\begin{tikzpicture}[x=1pt,y=1pt,yscale=-0.3,xscale=0.3,baseline=-85pt, line width=1pt]

\draw [dash pattern={on 3pt off 2pt}] [-stealth]  (90,270)..controls (130,220)..(165,265) ;
\draw [dash pattern={on 3pt off 3pt}]  [-stealth] (250 ,270).. controls (290, 220).. (325,265) ;

\draw [-stealth] (90, 270)--(165,270) ;
\draw [-stealth](170, 270)--(245,270) ;

\draw  [fill={rgb, 255:red, 0; green, 0; blue, 0 }  ,fill opacity=1 ] (90,270) circle (5);
\draw (90, 290) node {$1$};
\draw  [fill={rgb, 255:red, 0; green, 0; blue, 0 }  ,fill opacity=1 ] (170,270) circle (5);
\draw (170, 290) node {$2$};
\draw [fill={rgb, 255:red, 0; green, 0; blue, 0 }  ,fill opacity=1 ] (250,270) circle (5);
\draw (250, 290) node {$3$};
\draw [fill={rgb, 255:red, 0; green, 0; blue,0 } ,fill opacity=1] (330,270) circle (5);
\draw (330, 290) node {$4$};

\end{tikzpicture}}

\newcommand{\graphDc}[0]{
\begin{tikzpicture}[x=1pt,y=1pt,yscale=-0.3,xscale=0.3,baseline=-85pt, line width=1pt]

\draw [dash pattern={on 3pt off 2pt}] [-stealth]  (90,270)..controls (170,220)..(245,265) ;
\draw [dash pattern={on 3pt off 3pt}]  [-stealth] (170,270).. controls (250, 220).. (325,265) ;

\draw [-stealth] (90, 270)--(165,270) ;
\draw [-stealth](250, 270)--(325,270) ;

\draw  [fill={rgb, 255:red, 0; green, 0; blue, 0 }  ,fill opacity=1 ] (90,270) circle (5);
\draw (90, 290) node {$1$};
\draw  [fill={rgb, 255:red, 0; green, 0; blue, 0 }  ,fill opacity=1 ] (170,270) circle (5);
\draw (170, 290) node {$2$};
\draw [fill={rgb, 255:red, 0; green, 0; blue, 0 }  ,fill opacity=1 ] (250,270) circle (5);
\draw (250, 290) node {$3$};
\draw [fill={rgb, 255:red, 0; green, 0; blue,0 } ,fill opacity=1] (330,270) circle (5);
\draw (330, 290) node {$4$};

\end{tikzpicture}}

\newcommand{\graphD}[0]{
\begin{tikzpicture}[x=1pt,y=1pt,yscale=-0.3,xscale=0.3,baseline=-85pt, line width=1pt]

\draw [dash pattern={on 3pt off 2pt}] [-stealth]  (90,270)..controls (170,220)..(245,265) ;
\draw [dash pattern={on 3pt off 3pt}]  [-stealth] (170,270).. controls (250, 220).. (325,265) ;

\draw [-stealth] (90, 270)--(165,270) ;
\draw [-stealth](170, 270)--(245,270) ;

\draw  [fill={rgb, 255:red, 0; green, 0; blue, 0 }  ,fill opacity=1 ] (90,270) circle (5);
\draw (90, 290) node {$1$};
\draw  [fill={rgb, 255:red, 0; green, 0; blue, 0 }  ,fill opacity=1 ] (170,270) circle (5);
\draw (170, 290) node {$2$};
\draw [fill={rgb, 255:red, 0; green, 0; blue, 0 }  ,fill opacity=1 ] (250,270) circle (5);
\draw (250, 290) node {$3$};
\draw [fill={rgb, 255:red, 0; green, 0; blue,0 } ,fill opacity=1] (330,270) circle (5);
\draw (330, 290) node {$4$};

\end{tikzpicture}}

\newcommand{\graphE}[0]{
\begin{tikzpicture}[x=1pt,y=1pt,yscale=-0.3,xscale=0.3,baseline=-85pt, line width=1pt]

\draw [dash pattern={on 3pt off 2pt}] [-stealth]  (210,220)-- (245,265) ;
\draw [dash pattern={on 3pt off 3pt}]  [stealth-] (210,220)-- (325,265) ;
\draw [dash pattern={on 3pt off 3pt}]  [-stealth] (165, 270)--(210,220) ;

\draw [-stealth](170, 270)--(245,270) ;

\draw  [fill={rgb, 255:red, 0; green, 0; blue, 0 }  ,fill opacity=1 ] (170,270) circle (5);
\draw (170, 290) node {$1$};
\draw [fill={rgb, 255:red, 0; green, 0; blue, 0 }  ,fill opacity=1 ] (250,270) circle (5);
\draw (250, 290) node {$2$};
\draw [fill={rgb, 255:red, 0; green, 0; blue,0 } ,fill opacity=1] (330,270) circle (5);
\draw (330, 290) node {$3$};

\draw  (210,215) circle (5);
\draw (210, 200) node {$4$};

\end{tikzpicture}}

\newcommand{\graphF}[0]{
\begin{tikzpicture}[x=1pt,y=1pt,yscale=-0.3,xscale=0.3,baseline=-85pt, line width=1pt]

\draw [dash pattern={on 3pt off 2pt}] [-stealth]  (170,265)..controls (210,220)..(245,265) ;
\draw [dash pattern={on 3pt off 3pt}]  [-stealth] (170,275).. controls (210, 320).. (245,265) ;

\draw [dash pattern={on 3pt off 2pt}] [-stealth] (90, 270)--(165,270) ;
\draw [dash pattern={on 3pt off 2pt}] [-stealth](255, 270)--(325,270) ;

\draw  [fill={rgb, 255:red, 0; green, 0; blue, 0 }  ,fill opacity=1 ] (90,270) circle (5);
\draw (90, 290) node {$1$};
\draw  (170,270) circle (5);
\draw (170, 290) node {$3$};
\draw (250,270) circle (5);
\draw (250, 290) node {$4$};
\draw [fill={rgb, 255:red, 0; green, 0; blue,0 } ,fill opacity=1] (330,270) circle (5);
\draw (330, 290) node {$2$};

\end{tikzpicture}}

 \newcommand{\ctext}[1]{\raise0ex\hbox{\textcircled{\scriptsize{#1}}}}

\begin{document} 
\title{On hidden face contributions of configuration space integrals for long embeddings}
\author{Leo Yoshioka\thanks{Graduate School of Mathematical Sciences, The University of Tokyo\newline\qquad e-mail:leo.yoshioka@icloud.com}}
\maketitle

\begin{abstract}
Configuration space integrals are powerful tools for studying the homotopy type of the space of long embeddings in terms of a combinatorial object called a graph complex. It is unknown whether these integrals give a cochain map due to potential obstructions called hidden faces.  The purpose of this paper is to address these hidden faces by modifying configuration space integrals: we incorporate the acyclic bar complex of some dg algebra into the original graph complex, without changing its cohomology. Then, we give a cochain map from the new graph complex to the de Rham complex of the space of long embeddings modulo immersions, by combining the original configuration space integrals with Chen's iterated integrals. As the original complex, we choose a modified graph complex so that it is quasi-isomorphic to both the hairy graph complex and a graph complex introduced in the context of embedding calculus.  

\end{abstract}

\tableofcontents

\section*{Introduction}
\addcontentsline{toc}{section}{Introduction}

A \textit{long embedding} is an embedding of $\mathbb{R}^j$ into $\mathbb{R}^n$ $(n\geq j\geq1)$ that coincides with the standard linear embedding outside a ball in $\mathbb{R}^j$. We write  $\mathcal{K}_{n,j} = \text{Emb} (\mathbb{R}^j, \mathbb{R}^n) $ for the space of long embeddings equipped with the usual $C^{\infty}$ topology.  As long embeddings are \textit{long immersions}, there is a map $\mathcal{K}_{n,j} \rightarrow \text{Imm} (\mathbb{R}^j, \mathbb{R}^n)$ to the space of long immersions.  Since the space $\text{Imm} (\mathbb{R}^j, \mathbb{R}^n)$ is relatively well-understood, we often consider the difference between the two spaces
\[
\overline{\mathcal{K}}_{n,j} = \overline{\text{Emb}} (\mathbb{R}^j, \mathbb{R}^n) =  \text{hofib}_{\iota} (\mathcal{K}_{n,j} \rightarrow \text{Imm}(\mathbb{R}^j, \mathbb{R}^n)),
\]
the homotopy fiber at the standard inclusion $\iota$. 
We mainly focus on the case $j \geq 2$, $n-j \geq 2$, but sometimes mention the case $j=1$.

One powerful tool to compute the homotopy type of the space of long embeddings $\overline{\mathcal{K}}_{n,j}$ is \textit{graph complexes}. In 2017, Fresse, Turchin and Willwacher \cite{FTW 1}, following Arone and Turchin \cite{AT 1, AT 2}, showed the surprising result that the homology of a graph complex $^\ast HGC_{n,j}$\footnote{In this paper, we always put $^{\ast}$, the symbol of \textit{dual}, to the left not to the right.}, called the \textit{hairy graph complex}, has the full dimensional information on the rational homotopy group of $\overline{\mathcal{K}}_{n,j}$, when $n-j $ is greater than or equal to three. This result is obtained by a deep homotopy theory called Goodwillie-Klein-Weiss \textit{embedding calculus} \cite{GKW, GW, Wei}. Embedding calculus gives the \textit{Taylor approximation} $\overline{\mathcal{K}}_{n,j} \rightarrow T_{\infty} \overline{\mathcal{K}}_{n,j}$, which is a weak equivalence when $n-j \geq 3$, $j \geq 1$. Then, the hairy graph complex is introduced by applying the rational homotopy theory to the \textit{derived mapping space models} of $T_{\infty} \overline{\mathcal{K}}_{n,j}$ \cite{AT 1, BW} and performing a fibrant replacement. As shown in \cite{AT 2}, another replacement, a cofibrant replacement, of the same origin gives another graph complex $^{\ast} HH_{n,j}$, which we mention later. 

On the other hand, from the 1990s to 2010s, Bott \cite{Bot}, Cattaneo, Rossi \cite{CR}, Sakai \cite{Sak} and Watanabe \cite{SW, Wat 1} introduced another graph complex $GC^{BCR}_{n,j} $ $(j \geq 2)$, which we call the \textit{BCR graph complex}. The complex is generated by \textit{BCR graphs}, which have two types of vertices (external and internal) and two types of edges (solid and dashed).
They gave a linear map from the BCR graph complex to the de Rham complex of $\mathcal{K}_{n,j}$ (and so to that of $\overline{\mathcal{K}}_{n,j})$  through configuration space integrals
\[
I: GC^{BCR}_{n,j}  \longrightarrow \Omega_{dR} \mathcal{K}_{n,j} \quad \Gamma \longmapsto \int_{\text{Conf}_{\Gamma}} \omega(\Gamma).
\]
Fortunately, this approach is applicable to the mysterious codimension: the case $n-j = 2$. 
However, it is unknown whether the map gives a cochain map. The exceptions are the cases where the first Betti number of graphs is at most one, as in \cite{Bot, CR, Sak, SW, Wat 1}. In fact, they gave non-trivial  cocycles of $\mathcal{K}_{n,j}$ or  $\overline{\mathcal{K}}_{n,j}$  for $n-j\geq 2$ from $0$ or $1$-loop BCR graph cocycles.

Configuration space integrals are superior in that they are applicable to the most mysterious case $n-j = 2$ and in that they give geometric cycles and cocycles. However, as we mentioned,  it is unknown whether they give a cochain map due to potential obstructions. These obstructions are integrals over specific faces of configuration spaces called \textit{hidden faces}:
\[
dI(\Gamma) -I d(\Gamma) = \int_{\partial_{hid} \text{Conf}_{\Gamma}} \omega(\Gamma). 
\]
Moreover, to show the non-triviality of the map $I$, we are required to compute the cohomology of $GC^{BCR}_{n,j}$, which is in general more difficult to compute than $HGC_{n,j}$ and $HH_{n,j}$. Arone and Turchin \cite{AT 2} mentioned that some of their graphs in $HH_{n,j}$ and  $HGC_{n,j}$ are very similar to BCR graphs. However, explicit relationships between $GC^{BCR}_{n,j}$ and $HH_{n,j}$ or $HGC_{n,j}$ have not been established yet.

\section*{}

In this paper, we introduce a new graph complex $DGC_{n,j}$, which we call the \textit{decorated graph complex}.  We define configuration space integrals using this complex and show they give a cochain map.  We also show $DGC_{n,j}$ is quasi-isomorphic to the hairy graph complex $HGC_{n,j}$ and Arone--Turchin's complex $HH_{n,j}$. 

Graphs of $DGC_{n,j}$ are analogs of BCR graphs but their external vertices are decorated by elements of the \textit{acyclic bar complex} of some dg algebra $A_{n,j}$,
\[
Z_{n,j} = B(A_{n,j}, A_{n,j}, \mathbb{R}) = A_{n,j}\otimes_{\tau} BA_{n,j}, 
\]
where $A_{n,j}$ is a dg algebra with a map to the space $\text{Inj}(\mathbb{R}^j, \mathbb{R}^n)$ of linear injective maps. (The symbol $\tau$ means that the complex is twisted by the universal morphism $\tau: BA_{n,j} \rightarrow A_{n,j}$.) 
This construction is highly motivated by Turchin's construction of the decorated graph complex associated with the space of long knots (modulo immersions) \cite{Tur 1} as well as by Sakai and Watanabe's correction term for anomalous hidden faces in \cite{SW}. 

The complex $Z_{n,j} = A_{n,j}\otimes_{\tau} BA_{n,j}$ is acyclic but plays a role in canceling contributions of hidden faces. The following is our first main result. 
\begin{theorem}[Theorem \ref{main theorem 1 on hidden faces2}]
\label{main theorem 1 on hidden faces}
When $n-j \geq 2$ and $j \geq 3$, there  exists a cochain map
\[
\overline{I}: DGC_{n,j} \rightarrow \Omega_{dR} (\overline{\mathcal{K}}_{n,j}). 
\] 
It holds even when $j = 2$, if $\overline{I}$ is restricted to the subspace of $DGC_{n,j}$ generated by decorated graphs with the first Betti number $g \leq 3$. 
\end{theorem}
This map is given by combining the original configuration integrals with \textit{Chen's iterated integrals} \cite{Che 1, Che 2}: a standard way to construct, when a space $X$ is given, a cochain map from the bar complex of $\Omega_{dR}(X)$  to the de Rham complex of the path space of $X$. We discuss the case $j =2$ separately because an argument for showing $A^0_{n,j} \cong \mathbb{R}$  for dimensional reasons does not work well when $j =2$. 

We can define another graph complex $PGC_{n,j}$ whose generators are graphs decorated by (a scalar multiple of) the unit of $A_{n,j}$. We call this complex the \textit{plain graph complex}. The complex $PGC_{n,j}$ is a modified BCR graph complex in the sense that it is obtained from $GC^{BCR}_{n,j}$ by adding some graphs and imposing the condition that unsuitable graphs vanish. In particular, there is a cochain map $GC^{BCR}_{n,j} \rightarrow PGC_{n,j}$ defined by the identity \footnote{This map is not a quasi-isomorphism. The explicit difference between $H^{\ast}(GC^{BCR}_{n,j})$ and $H^{\ast}(PGC_{n,j})$ is still unknown, though it might be related to the homotopy type of $\text{Imm}(\mathbb{R}^j, \mathbb{R}^n)$. See Remark \ref{differencebetweenBCRandplain}.}. Although graphs of $DGC_{n,j}$ are more complicated than $PGC_{n,j}$, we have the following, which is our second main result. 

\begin{theorem}[Theorem \ref{maintheoremrestate}. Suggested by Turchin]
\label{main theorem 2 on hidden faces}
The projection $DGC_{n,j} \rightarrow PGC_{n,j}$ is a quasi-isomorphism.
\end{theorem}
We show this theorem by comparing the $E_1$-terms of spectral sequences computing the graph cohomologies  $H^{\ast}(DGC_{n,j})$ and $H^{\ast}(PGC_{n,j})$: the $E_0$-differential on $DGC_{n,j}$ is the differential on the acyclic complex $Z_{n,j}$ and hence non-trivial decorations vanish at the $E_1$-page. 

Introducing the complex $PGC_{n,j}$ allows us to compare the graph complexes of embedding calculus and configuration space integrals. Let $PGC^{\prime}_{n,j}$ be the analog of $PGC_{n,j}$, which allows graphs with \textit{double edges} and \textit{loop edges}. The following result follows from the famous fact \cite{Kon 2, LV} that the cohomology  $H^{\ast}(\text{Conf}_{\bullet}(\mathbb{R}^n))$  of configuration spaces in $\mathbb{R}^n$ is quasi-isomorphic to \textit{Kontsevich's graph cooperad}, which we write $\overline{\text{Graph}}_n(\bullet)$,  and to its analog $\text{Graph}_n(\bullet)$ (allowing double edges and loop edges). 

\begin{theorem}[Theorem \ref{relationwithAroneTurchincomplex}]
\label{main theorem 3 on hidden faces}
The projections $PGC_{n,j} \rightarrow HH_{n,j}$ and  $PGC^{\prime}_{n,j} \rightarrow HH_{n,j}$ are quasi-isomorphisms. In particular, $PGC^{\prime}_{n,j} \rightarrow PGC_{n,j}$ is a quasi-isomorphism. 
\end{theorem}

On the other hand, by carefully studying  Arone and Turchin's argument \cite{AT 2}, we show the following. 
\begin{theorem} [Theorem \ref{relationwithhairygraphcomplex}]
\label{main theorem 4 on hidden faces}
The projection $PGC^{\prime}_{n,j} \rightarrow HGC_{n,j}$  is a quasi-isomorphism. 
\end{theorem}

We write the top cohomology of $HGC_{n,j}$ by $H^{top}(HGC_{n,j})$  \footnote{See Definition \ref{defoftopdegreecohomology} for what ``top" means. The top homology is a subspace of the homologies of degrees $k(n-j-2) + (j-1)(g-1)$, where $g$ and $k$ are the first Betti number and the order of graphs, respectively. See Lemma \ref{degreeanddefect}.}.
The $2$-loop part of $H^{top}(HGC_{n,j})$ is well-studied in \cite{CCTW, Nak}, and is shown to be infinite-dimensional. (We recall their argument in Section \ref{computation of top homology}.)  Hence, combining Theorem \ref{main theorem 2 on hidden faces}, \ref{main theorem 3 on hidden faces} and \ref{main theorem 4 on hidden faces}, we have
\begin{cor}
The 2-loop part of $H^{top}(DGC_{n,j})$ is infinite-dimensional.
\end{cor}

As in \cite{AT 2}, the proof of Theorem \ref{main theorem 4 on hidden faces} needs \textit{Koszul duality theory}. 
We can also directly compare the top homology of $^{\ast} HGC_{n,j}$ and the top homology of $^{\ast} PGC^{\prime}_{n,j}$ when $n-j$ is even, in a similar manner to Bar-Natan's \cite{Bar} proof of PBW (Poincar\'{e}--Birkhoff--Witt) isomorphisms between the space $\mathcal{B}$ of open Jacobi diagrams, and $\mathcal{A}(S^1)$, the space of Jacobi diagrams on the unit circle. This argument for the top homology is the content of the author's paper \cite{Yos 2}. See Theorem \ref{weakerrelationwithhairygraphcomplex}. 

We remark that the degrees of the $2$-loop part of the top cohomology of $HGC_{n,j}$ are concentrated in degree $j-1$ when $n-j=2$. As a result, we obtain well-defined  $(j-1)$-cocycles of $\overline{\mathcal{K}}_{j+2 ,j}$ through our configuration space integrals associated with $2$-loop graphs. In \cite{Yos 3}, in preparation, we will study pairing between these $(j-1)$-cocycles with $(j-1)$-cycles constructed similarly to \cite{Yos 1}. This approach is very likely to give an alternative proof of the non-finite generation of the rational homotopy group $\pi_{j-1}^{\mathbb{Q}} (\mathcal{K}_{j+2 ,j})_\iota$ ($j \geq 2$) established by Budney--Gabai and Watanabe \cite{BG, Wat 2}. See the proceedings \cite{Yos 4}.

Finally, let us mention the case $j=1$. The Taylor approximation $\overline{\mathcal{K}}_{n,1} \rightarrow T_{\infty} \overline{\mathcal{K}}_{n,1}$ is a weak equivalence when $n\geq 4$. Sinha \cite{Sin 1} gave a \textit{cosimplicial model} of this Taylor approximation. At the $E_1$-page of the Bousfield-Kan homology spectral sequence \cite{BK, Bou} associated with this cosimplicial space, we have a certain \textit{Hochschild complex}. We write this complex as $^\ast HH_{n,1}$ because it is an analog of $^\ast HH_{n,j}$. When $n\geq 4$, this spectral sequence converges and rationally collapses at the second page \cite{LTV}. In \cite{Tur 2}, 
Turchin introduced a hairy graph complex, which is the same as $^\ast HGC_{n,j}$, and also introduced another graph complex, which is an analog of $^{\ast} PGC^{\prime}_{n,j}$. Let us write them as $^\ast HGC_{n,1}$ and $^\ast PGC^{\prime}_{n,1}$. Turchin showed that these three complexes $^\ast HH_{n,1}$, $^\ast PGC^{\prime}_{n,1}$ and $^\ast HGC_{n,1}$ are weakly equivalent. See subsection \ref{Operadic descriptions of  $PGC_{n,j}$, $HH_{n,j}$ and $HGC_{n,j}$} for the definition of these graph complexes. 

 On the other hand, Cattaneo, Cotta-Ramusino and Longoni \cite{CCL},  following Bott-Taubes \cite{BT} established configuration integrals from another graph complex, which we write $GC^{CCL}_{n,1}$, to $\Omega_{dR}(\mathcal{K}_{n,1})$. This map is a cochain map when $n\geq 4$.  However, whether this map is a cochain map is unknown when $n=3$. We only know correction terms for the top degree introduced by Bott and Taubes \cite{BT}. 

Our approach might be useful to this case $n=3$, because we can define the decorated graph complex $DGC_{n,1} \simeq PGC^{\prime}_{n,1}$ and a cochain map from this complex to $\Omega_{dR} (\overline{\mathcal{K}}_{n,1})$ for $n\geq 3$ \footnote{By using the argument in \cite{Tur 1}, we can show $H^{\ast}(PGC^{\prime}_{n,1}) \cong  H^{\ast}(GC^{CCL}_{n,1}) \oplus \pi^{\mathbb{Q}}_{\ast} \Omega^2 (S^{n-1})$ while $\overline{\mathcal{K}}_{n,1} \simeq \mathcal{K}_{n,1} \times \Omega^2 S^{n-1}$.}.

\section*{}
This paper is organized as follows. 
In Section \ref{Preliminaries for decorated graph complexes}, we recall the definition of the graph cooperad $\text{Graph}_n(\bullet)$, which we use to describe several graph complexes. In Section \ref{The plain graph complex and its operadic description}, we define the plain graph complexes $PGC_{n,j}$, $PGC^{ \prime}_{n,j}$,  Arone--Turchin's complex $HH_{n,j}$ and the hairy graph complex $HGC_{n,j}$. We also give  (co)operadic descriptions of these complexes using $\text{Graph}_n(\bullet)$. We prove Theorem \ref{main theorem 3 on hidden faces} and Theorem \ref{main theorem 4 on hidden faces} in this section. In Section \ref{computation of top homology}, we recall the computation of the $2$- and $3$-loop parts of the top cohomology of $HGC_{n,j}$. In Section~\ref{The hidden face dg algebra and the decorated graph complex}, first, we define the hidden face algebra $A_{n,j}$. Then, we introduce the decorated graph complex generated by graphs that are decorated by elements of the bar complex of $A_{n,j}$. We prove Theorem \ref{main theorem 2 on hidden faces} at the end of this section. In Section \ref{The configuration space integrals with iterated integrals}, we give a map from $DGC_{n,j}$ to the de Rham complex of $\overline{\mathcal{K}}_{n,j}$ and prove Theorem \ref{main theorem 1 on hidden faces}.

\newpage

\begin{figure}[htpb]
   \centering
    \includegraphics [width =14cm] {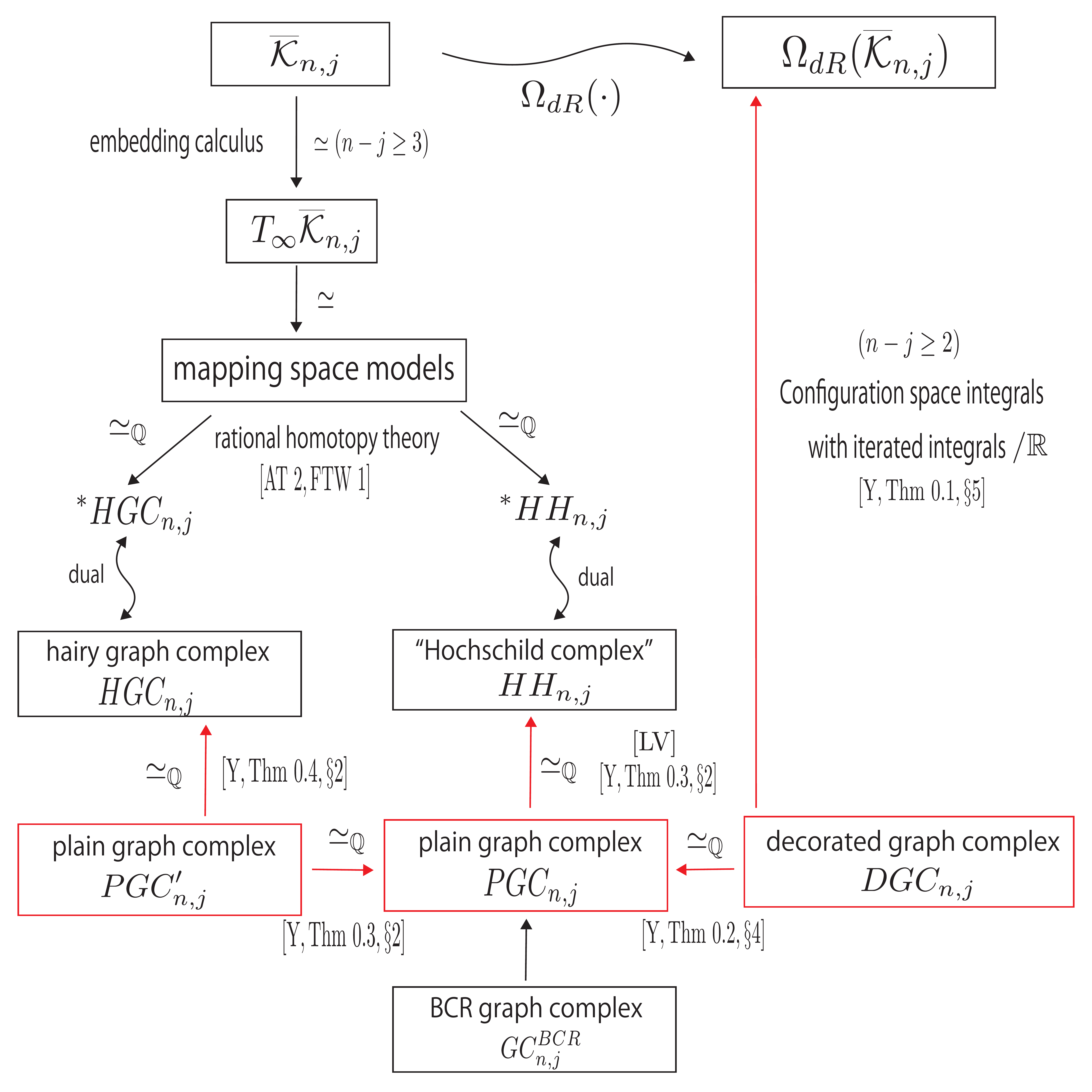}
    \caption{Overview ($j \geq 2$). The part drawn in red is the scope of this paper. }
    \label{overview2}
\end{figure}

\footnote{In \cite{AT 2}, $^{\ast}HGC_{n,j}$ and $^{\ast}HH_{n,j}$ are denoted by ${\mathcal{E}}^{j,n}_{\pi}$ and ${\mathcal{K}}^{j,n}_{\pi}$ respectively.}

\newpage

\begin{figure}[htpb]
   \centering
    \includegraphics [width =15cm] {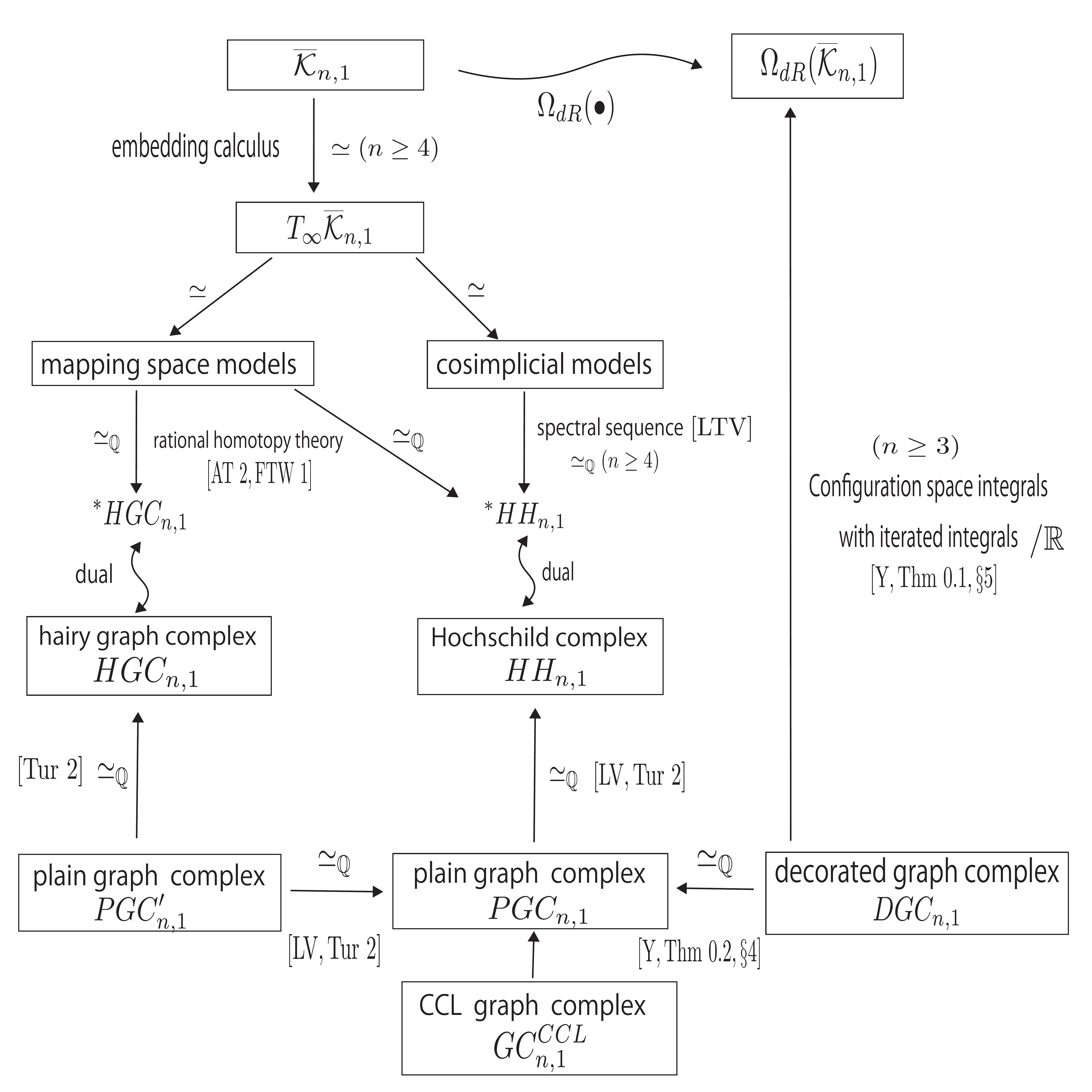}
    \caption{Overview ($j =1$)}
    \label{overview2}
\end{figure}

\footnote{In \cite{Tur 2}, $^{\ast}HGC_{n, 1}$, $^{\ast}PGC^{\prime}_{n, 1}$ and $^{\ast}HH_{n,1}$ are denoted by ${\mathcal{E}}_{1,n} = AM(P_n^{\bullet})$, $Tot (P_n^{\bullet})$ and $Tot (L_n^{\bullet})$ respectively. }

\newpage

\section*{Acknowledgement}

The author is deeply grateful to Victor Turchin for explaining to the author a lot about his work on embedding calculus and graph homologies. This paper is highly motivated by discussions with him. In particular, he suggested the proofs of Theorem \ref{main theorem 2 on hidden faces}, Theorem \ref{main theorem 3 on hidden faces} and Theorem \ref{main theorem 4 on hidden faces} to the author. He also gave the author a lot of advice for better exposition throughout this paper.  The author would like to thank Keiichi Sakai for giving the author a great opportunity to talk at Shinshu Topology Seminar and discuss details on hidden faces. The author thanks Tadayuki Watanabe for discussing hidden faces and correction terms with the author. 
The author thanks the referee for pointing out mistakes in the original proof of Proposition \ref{numberofinternal} and for suggesting a revised proof. The referee also made many valuable suggestions for improving the paper. 
This paper was partly written during the author's three-month stay at Kansas State University and two-week stay at Shinshu University. The author appreciates their hospitality. 
This research was supported by Forefront Physics and Mathematics Program to Drive Transformation (FoPM), a World-leading Innovative Graduate Study (WINGS) Program, The University of Tokyo. This work was also supported by JSPS Grant-in-Aid for JSPS Fellows Grant Number 24KJ0565.

\newpage

\section{Preliminaries for the decorated graph complex}
\label{Preliminaries for decorated graph complexes}

\subsection{The space of long embeddings}

\begin{definition}
A \textit{long embedding} is an embedding $\mathbb{R}^j \rightarrow \mathbb{R}^n$ which coincides with the standard linear embedding $\iota: \mathbb{R}^j \subset \mathbb{R}^n$ outside a disk in $\mathbb{R}^j$. 
We equip the space $\text{Emb}(\mathbb{R}^j, \mathbb{R}^n)$ of long embeddings with the induced topology from the $C^{\infty}$ topology. We define the space of \textit{long immersions} $\text{Imm}(\mathbb{R}^j, \mathbb{R}^n)$ similarly. 
\end{definition}

\begin{definition}
\label{embeddings modulo immersions}
We write $\overline{\text{Emb}}(\mathbb{R}^j, \mathbb{R}^n)$ for the homotopy fiber of the inclusion
\[
\text{Emb}(\mathbb{R}^j, \mathbb{R}^n) \hookrightarrow \text{Imm}(\mathbb{R}^j, \mathbb{R}^n)
\]
at the standard linear embedding $\iota: \mathbb{R}^j \subset \mathbb{R}^n$. That is, an element of $\overline{\text{Emb}}(\mathbb{R}^j, \mathbb{R}^n)$  is a one-parameter family $\{\overline{K}_t\}_{t \in [0,1]}$ of long immersions which satisfies $\overline{K}_1 =~\iota$, $\overline{K}_0\in \text{Emb}(\mathbb{R}^j, \mathbb{R}^n)$. Let 
\[
r:  \overline{\text{Emb}}(\mathbb{R}^j, \mathbb{R}^n)  \rightarrow  \text{Emb}(\mathbb{R}^j, \mathbb{R}^n)
\]
be the natural projection.

\end{definition}

\subsection{Bar complexes and Chen's iterated integrals}

The \textit{acyclic bar complexes} $B(A, A, \mathbb{R})$  are defined in subsection \ref{The bar complex}. 
Write  $\Omega(M, \ast, \ast)$ for the smooth loop space of a manifold $M$ and write $P(M, \bullet, \ast)$ for the path space which consists of smooth paths of $M$ with a free initial point and the fixed terminal point $\ast$. 

In \cite{Che 1,Che 2}, Chen related the cohomologies of loop spaces with the cohomologies of bar complexes.

\begin{theorem}\cite{Che 1, Che 2}
Let $M$ be a simply connected manifold. Let $A$ be a dg subalgebra of $\Omega_{dR}(M)$ which satisfies the following.
\begin{itemize}
\item The inclusion $i: A \rightarrow \Omega_{dR}(M)$ induces a quasi-isomorphism.
\item $A^0 = \mathbb{R}$.
\end{itemize}
Then, the map induced by the \textit{iterated integral}
\[
J: B(\mathbb{R}, A, \mathbb{R}) \rightarrow \Omega_{dR}(\Omega(M, \ast, \ast)), \quad [a_1|a_2|\dots|a_k] \mapsto \int i(a_1) \dots i(a_k)
\]
is a quasi-isomorphism of dg algebras. 
\end{theorem}

In this paper, we use a weaker proposition as follows.
\begin{prop}
\label{factiteratedintegral1}
Let $M$ be a connected manifold. Let $A$ be a dg subalgebra of $\Omega_{dR}(M)$ such that $A_0 = \mathbb{R}$.
Then, the map induced by  the iterated integral
\[
I: B(A, A, \mathbb{R}) \rightarrow \Omega_{dR}(P(M, \bullet, \ast))
\]
is a morphism of dg algebras. Note that the map is a quasi-isomorphism since both of the cohomologies vanish. 

\end{prop}

\subsection{Configuration spaces and their cohomology}

\begin{definition}[\textit{Configuration spaces}]
Define the \textit{$k$-point configuration space} of $\mathbb{R}^n$ by 
\[
\text{Conf}_{k}(\mathbb{R}^n) = (\mathbb{R}^n)^{\times k} \setminus \Delta
\]
where $\Delta$ is the fat diagonal $\bigcup_{1 \leq i\neq j \leq k} \{y_i = y_j\}$. We often write $C_k(\mathbb{R}^n)$ for $\text{Conf}_{k}(\mathbb{R}^n)$. 
\end{definition}

\begin{notation}
Let $n\geq 2$. Let $P_{ij}: \text{Conf}_k(\mathbb{R}^n) \rightarrow S^{n-1}$ $(1 \leq i \neq j \leq k)$ be the map assigning the direction from the $i$-th point to the $j$-th point. We write $\omega_{ij}$ for the $(n-1)$-form of $\text{Conf}_k(\mathbb{R}^n)$ obtained by pulling back the standard volume form of $S^{n-1}$ by $P_{ij}$.
\end{notation}

F. R. Cohen showed the following \cite{Coh}.

\begin{theorem}
\label{cohomologyofconfigurationspaces}
Let $n\geq 2$. The cohomology $H^{\ast}(\text{Conf}_k(\mathbb{R}^n))$ of the $k$-point configuration space is generated as an algebra by $\omega_{ij}$ $(1\leq i \neq j \leq k)$ (of degree $n-1$) with relations
\begin{itemize}
\item $\omega_{hi} \omega_{ij} + \omega_{ij}\omega_{jh} + \omega_{hi} \omega_{ij} = 0$ $(1\leq h \neq i \neq j \neq h \leq k),$ 
\item $\omega_{ij} = (-1)^{n} \omega_{ji}$ $(1\leq i \neq j \leq k)$,
\item $\omega_{ij}^2 = 0$.
\end{itemize}
The first relation is called the \textit{Arnol'd relation} \cite{Arn}. On the other hand $H^{\ast}(\text{Conf}_k(\mathbb{R}^1))$ is concentrated at  degree $0$ and $H^{0}(\text{Conf}_k(\mathbb{R}^1)) \cong  \mathbb{Z} \Sigma_k$, where $\Sigma_k$ is the $k$-th symmetric group.
\end{theorem}

Elements of $H^{\ast}(\text{Conf}_k(\mathbb{R}^n))$ $(n\geq 2)$ can be drawn as follows. We consider graphs with $k$ ordered vertices and with ordered oriented edges. An oriented edge connecting the $i$-th and $j$-th vertices stands for $\omega_{ij}$. Loop edges are not allowed. Double edges are allowed, but they vanish by relations. See Figure  \ref{figcohomologyconfigurationspace} for an example.

\begin{figure}[htpb]
   \centering
    \includegraphics [width =5cm] {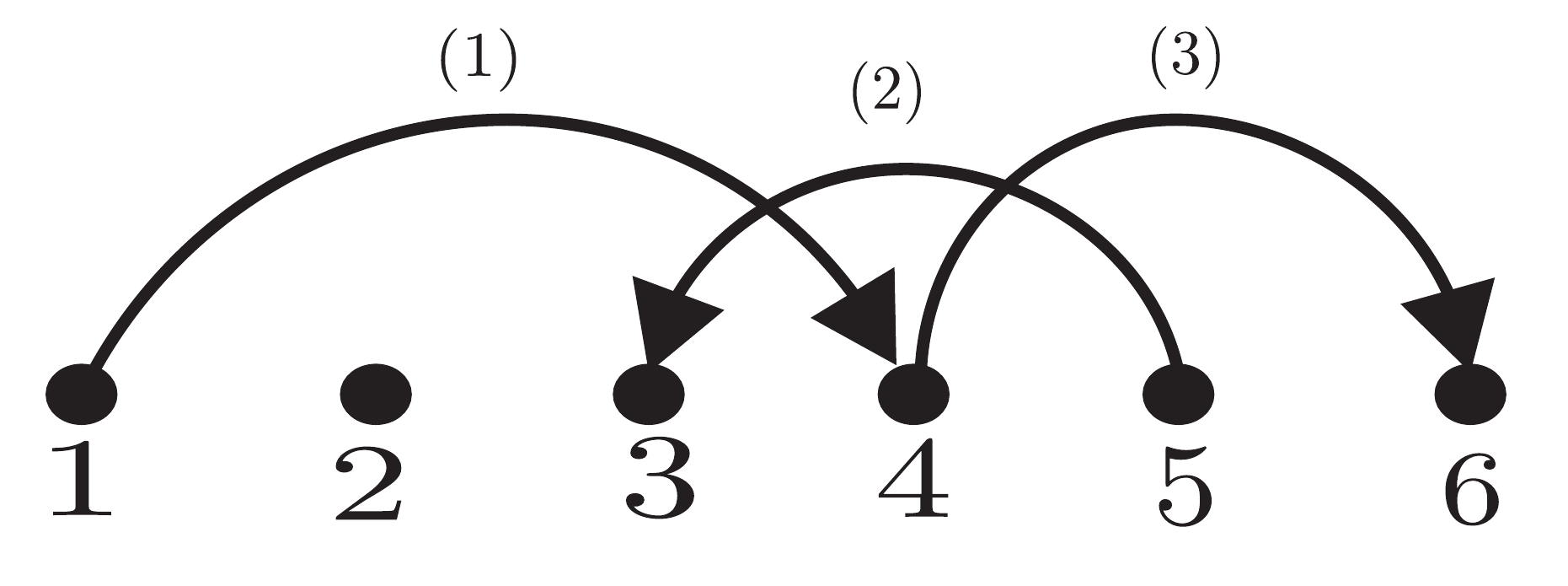}
    \caption{An element of $H^{\ast}(\text{Conf}_6(\mathbb{R}^n))$}
    \label{figcohomologyconfigurationspace}
\end{figure}

Then, an Arnold relation is drawn as follows.

\[
\begin{tikzpicture}[x=0.75pt,y=0.75pt,yscale=1,xscale=1, baseline=20pt, line width = 1pt] 
\draw [-Stealth]  (0,0)--(95,0) ;
\draw(50,0) node [anchor = north] {$(a)$} ;
\draw [-Stealth]  (100,0)--(50,48) ;
\draw (75,30) node [anchor = west]  {$(b)$} ;

\draw [fill={rgb, 255:red, 0; green, 0; blue, 0 }  ,fill opacity=1 ] (0,0.1) circle(4) node [anchor = south] {$h$}; 
\draw [fill={rgb, 255:red, 0; green, 0; blue, 0 }  ,fill opacity=1 ] (100,0) circle(4) node [anchor = south] {$i$}; ; 
\draw [fill={rgb, 255:red, 0; green, 0; blue, 0 }  ,fill opacity=1 ] (50,50) circle(4) node [anchor = south] {$j$} ; 
\end{tikzpicture}
+
\begin{tikzpicture}[x=0.75pt,y=0.75pt,yscale=1,xscale=1, baseline=20pt, line width = 1pt] 
\begin{scope}[xshift=100]
\draw [-Stealth]  (95,0)--(50,50);
\draw (75,30) node [anchor = west]  {$(a)$} ;
\draw [-Stealth]  (50,50)--(5,0);
\draw (5,30) node [anchor = west]  {$(b)$} ;
\draw [fill={rgb, 255:red, 0; green, 0; blue, 0 }  ,fill opacity=1 ] (0,0.1) circle(4) (0,0) circle(5) node [anchor = south] {$h$}; ; 
\draw [fill={rgb, 255:red, 0; green, 0; blue, 0 }  ,fill opacity=1 ] (100,0) circle(4) node [anchor = south] {$i$} ; 
\draw [fill={rgb, 255:red, 0; green, 0; blue, 0 }  ,fill opacity=1 ] (50,50) circle(4) node [anchor = south] {$j$} ; 
\end{scope}
\end{tikzpicture}
+
\begin{tikzpicture}[x=0.75pt,y=0.75pt,yscale=1,xscale=1, baseline=20pt, line width = 1pt] 
\begin{scope}[xshift=200]
\draw [-Stealth]  (0,0)--(95,0);
\draw(50,0) node [anchor = north] {$(b)$} ;
\draw [-Stealth]  (50,50)--(5,0);
\draw (5,30) node [anchor = west]  {$(a)$} ;
\draw [fill={rgb, 255:red, 0; green, 0; blue, 0 }  ,fill opacity=1 ] (0,0.1) circle(4) node [anchor = south] {$h$} ; 
\draw [fill={rgb, 255:red, 0; green, 0; blue, 0 }  ,fill opacity=1 ] (100,0) circle(4) node [anchor = south] {$i$} ; 
\draw [fill={rgb, 255:red, 0; green, 0; blue, 0 }  ,fill opacity=1 ] (50,50) circle(4) node [anchor = south] {$j$} ; 
\end{scope}
\end{tikzpicture}
= 0.
\]

Elements of $H^{\ast}(\text{Conf}_k(\mathbb{R}^1))$  can be drawn as $k$ ordered points on an oriented line:
\begin{center}
\begin{tikzpicture}[x=0.75pt,y=0.75pt,yscale=1,xscale=1, baseline=20pt, line width = 1pt] 
\draw [-Stealth, line width = 0.5 pt]  (0,0)--(300,0);
\draw [fill={rgb, 255:red, 0; green, 0; blue, 0 }  ,fill opacity=1 ] (50,0.1) circle(4) node [anchor = south] {$1$} ; 
\draw [fill={rgb, 255:red, 0; green, 0; blue, 0 }  ,fill opacity=1 ] (150,0) circle(4) node [anchor = south] {$2$} ; 
\draw [fill={rgb, 255:red, 0; green, 0; blue, 0 }  ,fill opacity=1 ] (200,0) circle(4) node [anchor = south] {$3$} ; 
\draw [fill={rgb, 255:red, 0; green, 0; blue, 0 }  ,fill opacity=1 ] (100,0.1) circle(4) node [anchor = south] {$4$} ; 
\draw [fill={rgb, 255:red, 0; green, 0; blue, 0 }  ,fill opacity=1 ] (250,0) circle(4) node [anchor = south] {$5$} ; 
\end{tikzpicture}
\end{center}

\subsection{The graph cooperad $\text{Graph}_n(\bullet)$}
\label{graphcooperad}

We recall the definition of the \textit{graph cooperad} $\text{Graph}_n(\bullet)$ and its quotient $\overline{\text{Graph}}_n(\bullet)$ to describe several graph complexes.  The quotient version $\overline{\text{Graph}}_n(\bullet)$ is introduced by Kontsevich \cite{Kon 1} and is well studied by Lambrechts and Voli\'c \cite{LV}. The version $\text{Graph}_n(\bullet)$ is used, for example, in \cite{Tur 2, AT 2, FTW 2} to form graph complexes. $\text{Graph}_n(\bullet)$ and  $\overline{\text{Graph}}_n(\bullet)$ have the structure of a cooperad in the category of commutative dg algebras, or dg Hopf cooperads. 
They are quasi-isomorphic to the cohomology $H^{\ast}(\text{Conf}_{\bullet}(\mathbb{R}^n), \mathbb{Q})$ of the configuration spaces as dg Hopf cooperads. First, let $n \geq 2$.

\begin{definition}[Graphs]
\textit{Graphs} have two types of vertices, 
\textit{external}  \begin{tikzpicture} [baseline = -3pt]  \draw (0, 0) circle (0.05) [fill={rgb, 255:red, 0; green, 0; blue, 0}, fill opacity =1.0]; \end{tikzpicture}
and \textit{internal} \begin{tikzpicture}  [baseline = -2pt]   \draw (0, 0) rectangle (0.1, 0.1) [fill={rgb, 255:red, 0; green, 0; blue, 0}, fill opacity =1.0]; \end{tikzpicture}.
and one type of edges \begin{tikzpicture} \draw [x=0.75pt,y=0.75pt,yscale=0.3,xscale=0.3, baseline=-3pt]  [line width = 1pt](0,0)--(90,0); \end{tikzpicture}.
External vertices are ordered. 
Internal vertices must have at least three (half) edges, while external vertices are allowed to have no edges. 
We assume that each component has at least one external vertex.
Double edges
\begin{tikzpicture}[x=1pt,y=1pt,yscale=1,xscale=1, baseline=-3pt, line width =1pt]
\draw  (0,0) .. controls (20,5) and (30,5)  .. (50,0);
\draw (0,0) .. controls (20,-5) and (30,-5)  .. (50,0);
\end{tikzpicture}
and loop edges
\begin{tikzpicture}[x=1pt,y=1pt,yscale=0.8,xscale=0.8, baseline=5pt, line width = 1pt]
\draw (0,0) .. controls (-20,20) and (20,20)  .. (0,0);
\end{tikzpicture}
are allowed.  
See Figure \ref{figofdiffofgraph}, \ref{figofproduct}, \ref{figofdecomposition} for examples of graphs.
\end{definition}

\begin{notation}
Write $E(\Gamma)$ for the set of edges.
Write $V(\Gamma) (= V_i(\Gamma) \cup V_e(\Gamma))$, $V_i(\Gamma)$, $V_e(\Gamma)$ for the set of vertices, internal vertices and external vertices.
\end{notation}

\begin{definition}[Degrees]
The \textit{degrees} of edges, internal vertices and external vertices are $n-1$, $-n$ and $0$ respectively. The degree $|\Gamma| = \deg(\Gamma)$ of a graph $\Gamma$ is defined by the sum of degrees $|\Gamma| = (n-1)|E(\Gamma)|-n|V_i(\Gamma)|$.
\end{definition}

\begin{definition}[Labels and orientations]
A \textit{label} of a graph $\Gamma$ consists of a choice of an ordering of the set $o(\Gamma) = E(\Gamma) \cup V_i(\Gamma)$ and a choice of orientations of edges. Each label gives an orientation of the underlying graph so that it depends only on parities of $n$. More precisely, if we change the ordering of two elements of degree $d$ and $d^{\prime}$ of $o(\Gamma)$, the orientation of a graph is changed by $(-1)^{d d^{\prime}}$. If we change the orientation of an edge, the orientation is changed by $(-1)^{n}$.
\end{definition}

\begin{definition}
Let $k \geq 1$ be an integer. $\text{Graph}_n(k)$ is a graded vector space over $\mathbb{Q}$ generated by graphs with $k$ external  vertices, equipped with the following orientation relation.
\begin{itemize}
\item If two graphs $\Gamma$ and $\Gamma^{\prime}$ with the same underlying graph have the same orientation (resp. the opposite orientations), then $\Gamma= \Gamma^{\prime}$ (resp. $\Gamma = -\Gamma^{\prime}).$
\end{itemize}
Let $\overline{\text{Graph}}_n(k)$ be the graded vector space obtained from $\text{Graph}_n(k)$ by adding the following relation. 
\begin{itemize}
\item A graph with double edges or loop edges vanishes. 
\end{itemize}
We regard $\text{Graph}_n(0)$ and $\overline{\text{Graph}}_n(0)$ as the one-dimensional vector space generated by the empty graph. 
\end{definition}

\begin{definition}[The action of $\Sigma_k$ on $\text{Graph}_n(k)$]
The symmetric group $\Sigma_k$ acts from the right on $\text{Graph}_n(k)$ as permutations of external vertices. The $i$-th external vertex of a graph $\Gamma$ is given a new label $\sigma^{-1} (i)$ as an external vertex of $\Gamma \sigma$.
\end{definition}

\begin{definition}[The differential on $\text{Graph}_n(k)$]
The differential $d: \text{Graph}_n(k) \rightarrow \text{Graph}_n(k)$ is defined by the sum of contractions of admissible edges. 
\[
d (\Gamma)  =  \sum_{\substack{e \in E^{a}(\Gamma)}} \sigma(e) \Gamma/e .
\]
See Figure \ref{figofdiffofgraph}. Here, the set $E^{a}(\Gamma) \subset E(\Gamma)$ of admissible edges consists of edges except for loop edges and edges connecting two external vertices. The sign $\sigma(e)$, which depends only on parities of $n$ and $j$, is characterized as follows
\begin{itemize}
\item Let $e= (p,q)$ $q>p$. If the edge $e$ and the (internal) vertex $q$ are the first and the second of the ordering of $o(\Gamma) = E(\Gamma) \cup V_i(\Gamma)$ (and if the remaining ordering coincides with the ordering of $o(\Gamma/e)$), the sign $\sigma(e)$ is positive. 
\end{itemize}
We easily have that $d$ vanishes on orientation relations of $\text{Graph}_n(k)$ and it preserves the number of external vertices $k$. 

\begin{figure}[htpb]
   \centering
    \includegraphics [width =9cm] {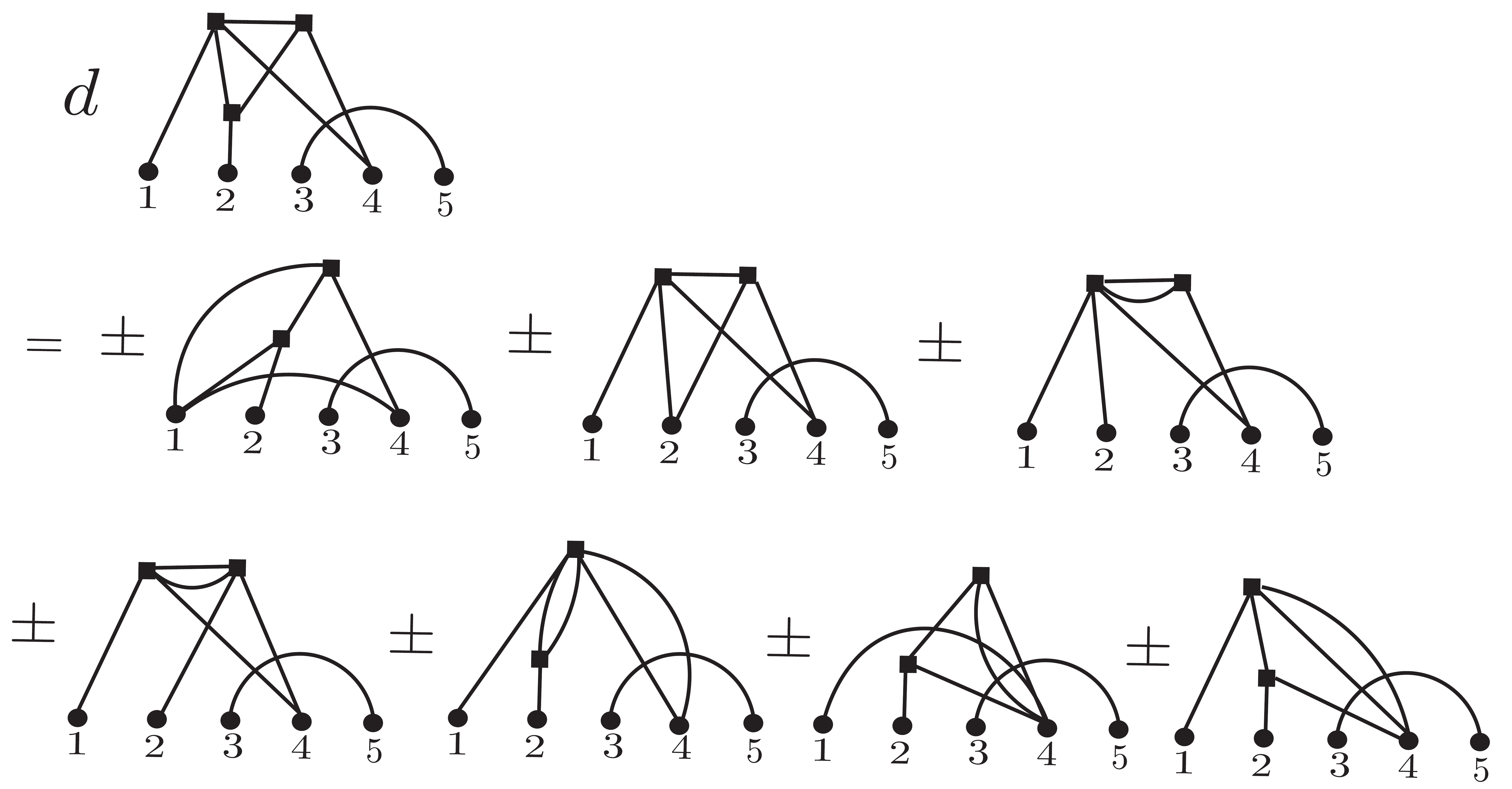}
    \caption{Example of the differential of a graph}
    \label{figofdiffofgraph}
\end{figure}

\end{definition}

\begin{definition}[The product on $\text{Graph}_n(k)$]
The product $\cdot : \text{Graph}_n(k) \otimes \text{Graph}_n(k)\rightarrow \text{Graph}_n(k)$ is defined by the \textit{superimposition} of graphs. See Figure \ref{figofproduct}. The label of the new graph is obtained as follows. The ordering of the set $o(\Gamma_1 \cdot \Gamma_2)$ is determined by the ordered set 
$o(\Gamma_1) \sqcup o(\Gamma_2)$. The edges of $\Gamma_1 \cdot \Gamma_2$ are  oriented consistently with those of $\Gamma_1$ and $\Gamma_2$.

\begin{figure}[htpb]
   \centering
    \includegraphics [width =9cm]  {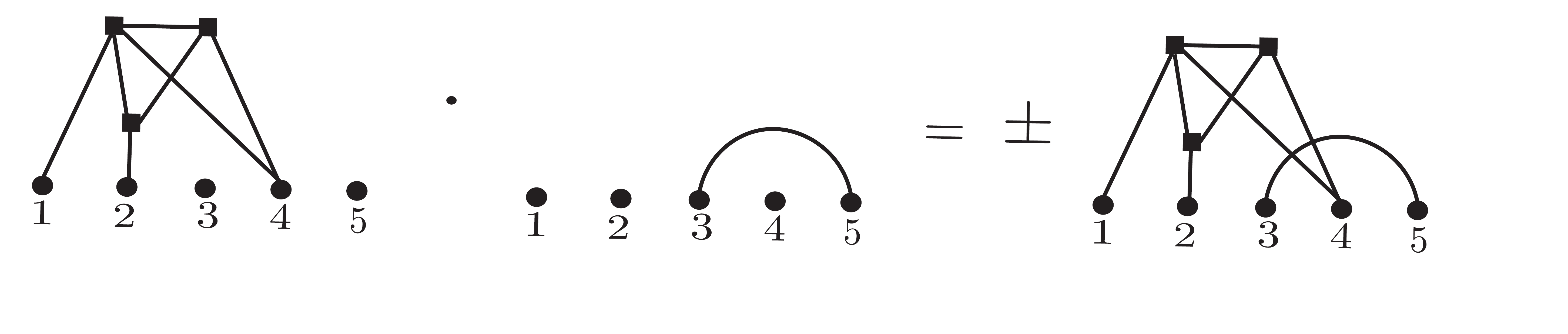}
    \caption{Example of the product of a graph}
    \label{figofproduct}
\end{figure}

\end{definition}

\begin{notation}
Let $\Gamma$ be a graph and let $S$ be a subset of $V(\Gamma)$.
We write $\Gamma_S$ for the subgraph on $S$; the set $E(\Gamma_S)$ of edges of $\Gamma_S$ consists of the edges of $\Gamma$ whose vertices are both in $S$.
\end{notation}

\begin{definition}[The decomposition operation on $\text{Graph}_n(\bullet)$]
Let $k\geq 0$.
The $i$-th $(1\leq i \leq k)$ decomposition
\[
^\ast \circ_i : \text{Graph}_n(k) \rightarrow  \bigoplus_{\substack{k=l+m-1 \\ l,m\geq 0}} \text{Graph} _n(l) \otimes \text{Graph}_n(m)
\]
is defined as follows.
Let $l, m$ satisfy $l, m \geq 0$ and $k = l+m-1$. 
We define the component
\[
^\ast \circ_i : \text{Graph}_n(k) \rightarrow  \text{Graph} _n(l) \otimes \text{Graph}_n(m)
\] by
\[
^\ast \circ_ i \Gamma = \sum_{S} \sigma(S) \sum_{\Gamma_0} \Gamma/\Gamma_0 \otimes \Gamma_0,
\]
where $S$ runs over the set of subsets of $V(\Gamma)$ which include from the $i$-th to the $(i+m-1)-$th external vertices, and $\Gamma_0$ runs over graphs such that $V(\Gamma_0) = S$ and $E(\Gamma_0) \subset E(\Gamma_S)$. The graph $\Gamma/\Gamma_0$ is the graph obtained from $\Gamma$ by collapsing all edges and vertices of $\Gamma_0$ to the $i$-th external vertex of $\Gamma$, and the collapsed vertex is considered as an external vertex of $\Gamma/\Gamma_0$.
We set $\Gamma_0 = 0$, if $\Gamma_0$ has internal vertices with less than two edges. 
See Figure \ref{figofdecomposition}.
Note that if $E(\Gamma_0) \subsetneq E(\Gamma_S)$, $\Gamma/\Gamma_0$ has a loop edge. 
The sign $\sigma(S)$ is characterized as follows. 
\begin{itemize}
\item If the ordering of $o(\Gamma)$ coincides with the ordered set $o(\Gamma/\Gamma_0) \sqcup o(\Gamma_0)$, the sign $\sigma(S)$ is positive. 
\end{itemize}

\begin{figure}[htpb]
   \centering
    \includegraphics [width =10cm] {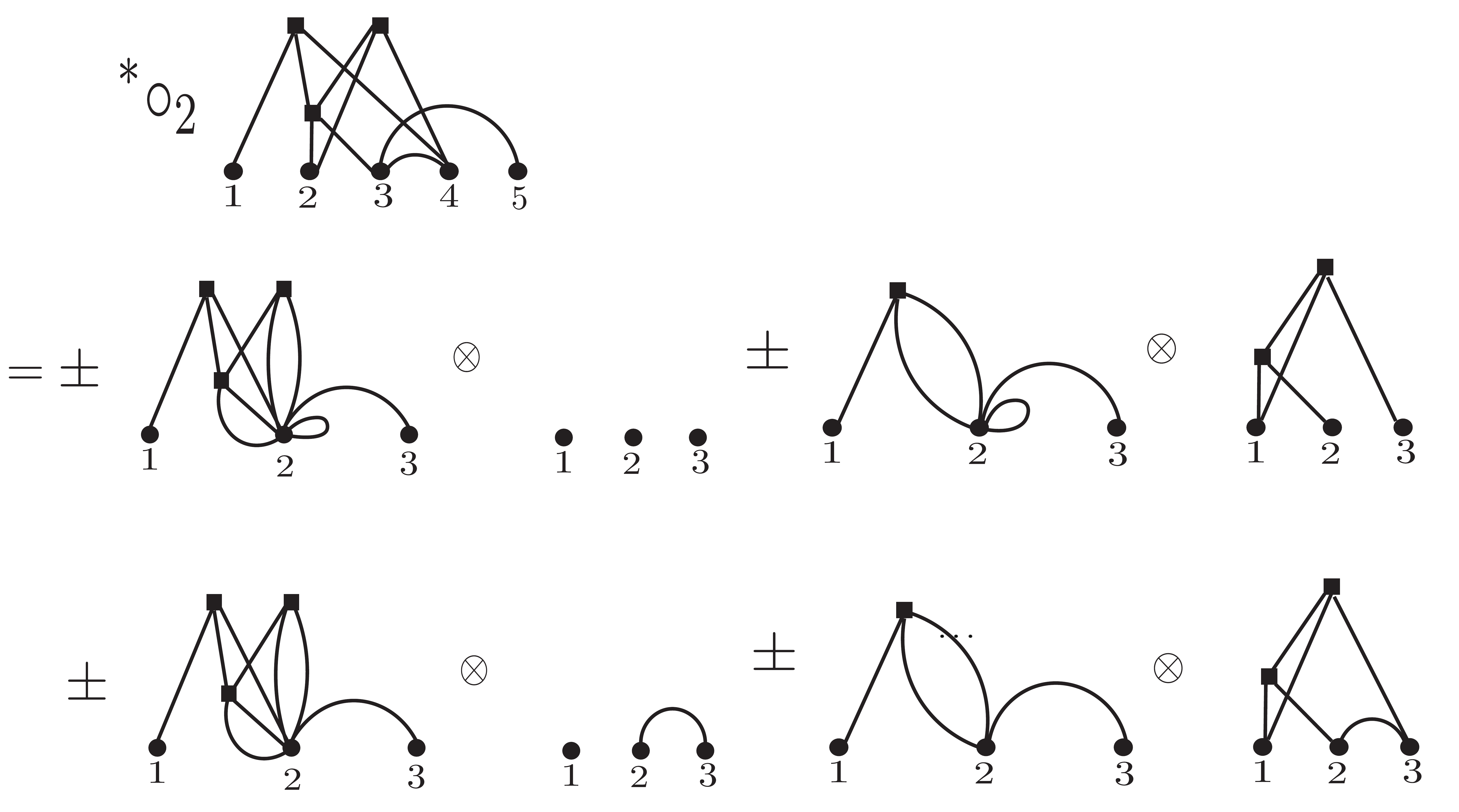}
    \caption{Example of the decomposition of a graph ($k = 5, l = 3, m=3$)}
    \label{figofdecomposition}
\end{figure}

\end{definition}

The differential, product, and decomposition on $\text{Graph}_n(k)$ induce those of $\overline{\text{Graph}}_n(k)$. Since a graph with a loop edge vanishes, the decomposition of $\overline{\text{Graph}}_n(k)$ coincides with
\[
^\ast \circ_ i \Gamma = \sum_{S} \sigma(S)  \Gamma/\Gamma_S \otimes \Gamma_S.
\]

Recall that the cohomology $H^{\ast}(\text{Conf}_{\bullet}(\mathbb{R}^n))$ $(n\geq 2)$ of configuration spaces is generated by graphs without internal vertices (Theorem \ref{cohomologyofconfigurationspaces}). We consider $H^{\ast}(\text{Conf}_{\bullet}(\mathbb{R}^n))$ as a dg Hopf cooperad with the trivial differential. 

\begin{theorem}\cite[Theorem 8.1]{LV}
The projections $\text{Graph}_n(\bullet) \rightarrow H^{\ast}(\text{Conf}_{\bullet}(\mathbb{R}^n))$ and $\overline{\text{Graph}}_n(\bullet) \rightarrow H^{\ast}(\text{Conf}_{\bullet}(\mathbb{R}^n))$ are quasi-isomorphisms as dg Hopf cooperads. In particular, the projection $\text{Graph}_n(\bullet) \rightarrow \overline{\text{Graph}}_n(\bullet)$ is a quasi-isomorphism.
\end{theorem}

When $n=1$, we define $\text{Graph}_1(\bullet) = \overline{\text{Graph}}_1(\bullet) = H^{\ast}(\text{Conf}_{\bullet}(\mathbb{R}))$. So its elements are ordered $k$ points on an oriented line.

\section{The plain graph complex and its operadic description}
\label{The plain graph complex and its operadic description}

In this section, we define the plain graph complex $PGC_{n,j}$, and recall the definition of the hairy graph complex $HGC_{n,j}$ and the Hochschild-type complex $HH_{n,j}$ which are introduced in \cite {AT 2} under different notation. After that, we describe these graph complexes in terms of the graph cooperad $\text{Graph}_n(\bullet)$ of the previous section. The contents of subsection \ref{The plain graph complex $PGC_{n,j}$}, \ref{The Hochschild-type complex $HH_{n,j}$} are for the case $j \geq 2$. The case $j =1$ is mentioned in subsection \ref{Operadic descriptions of  $PGC_{n,j}$, $HH_{n,j}$ and $HGC_{n,j}$}. Subsection \ref{The hairy graph complex $HGC_{n,j}$} is applicable to the case $j=1$.

\subsection{The plain graph complex $PGC_{n,j}$}
\label{The plain graph complex $PGC_{n,j}$}

\begin{definition}[Plain graphs]
\label{plain graphs}
\textit{Plain graphs} have three types of vertices
(white \begin{tikzpicture}[baseline = -3pt]  \draw (0, 0) circle (0.05); \end{tikzpicture},
external black \begin{tikzpicture} [baseline = -3pt]  \draw (0, 0) circle (0.05) [fill={rgb, 255:red, 0; green, 0; blue, 0}, fill opacity =1.0]; \end{tikzpicture},
internal black \begin{tikzpicture}  [baseline = -2pt]   \draw (0, 0) rectangle (0.1, 0.1) [fill={rgb, 255:red, 0; green, 0; blue, 0}, fill opacity =1.0]; \end{tikzpicture})
and two types of edges
(dashed \begin{tikzpicture}[x=0.75pt,y=0.75pt,yscale=0.3,xscale=0.3, baseline=-3pt]  \draw [dash pattern={on 4pt off 3pt}, line width = 1pt]   (0,0)--(90,0); \end{tikzpicture}, 
solid \begin{tikzpicture} \draw [x=0.75pt,y=0.75pt,yscale=0.3,xscale=0.3, baseline=-3pt]  [line width = 1pt](0,0)--(90,0); \end{tikzpicture}). 
White vertices have at least three dashed edges and no solid edge, while internal black vertices have at least three solid edges and no dashed edge. External black vertices have an arbitrary number of solid and dashed edges. We assume that each component has at least one external black vertex. 
Double edges 
\begin{tikzpicture}[x=1pt,y=1pt,yscale=1,xscale=1, baseline=-3pt, line width =1pt]
\draw  (0,0) .. controls (20,5) and (30,5)  .. (50,0);
\draw (0,0) .. controls (20,-5) and (30,-5)  .. (50,0);
\end{tikzpicture},
\begin{tikzpicture}[x=1pt,y=1pt,yscale=1,xscale=1, baseline=-3pt, line width =1pt]
\draw  [dash pattern = on 3pt off 3 pt] (0,0) .. controls (20,5) and (30,5)  .. (50,0);
\draw  [dash pattern = on 3pt off 3 pt] (0,0) .. controls (20,-5) and (30,-5)  .. (50,0);
\end{tikzpicture}
and loop edges
\begin{tikzpicture}[x=1pt,y=1pt,yscale=0.8,xscale=0.8, baseline=5pt, line width = 1pt]
\draw (0,0) .. controls (-20,20) and (20,20)  .. (0,0);
\end{tikzpicture},
\begin{tikzpicture}[x=1pt,y=1pt,yscale=0.8,xscale=0.8, baseline=5pt, line width = 1pt]
\draw [dash pattern = on 3pt off 3 pt] (0,0) .. controls (-20,20) and (20,20)  .. (0,0);
\end{tikzpicture}
are allowed. 

A plain graph is \textit{admissible} if every external black vertex has at least one dashed edge. See Figure \ref{figofplaingraph}. Note that a black vertex of an admissible plain graph is external if and only if it has at least one dashed edge. 
\begin{figure}[htpb]
   \centering
    \includegraphics [width =7cm] {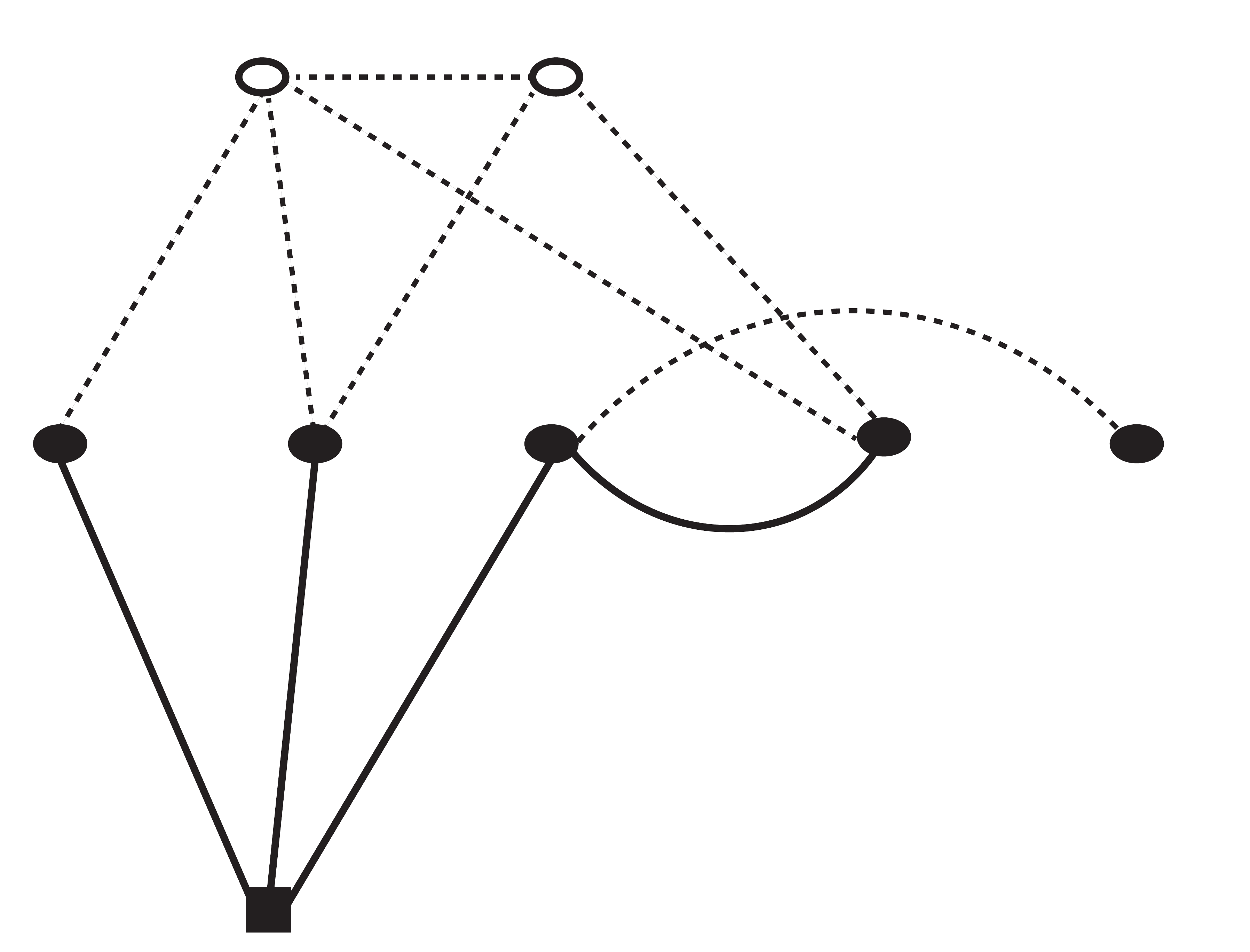}
    \caption{Example of an admissible plain graph}
    \label{figofplaingraph}
\end{figure}
\end{definition}

\begin{notation}
Write $E(\Gamma) (= E_{\theta}(\Gamma) \cup E_{\eta}(\Gamma))$, $E_{\theta}(\Gamma)$, $E_{\eta}(\Gamma)$ for the sets of edges, dashed edges and solid edges.
Write $V(\Gamma) (= W(\Gamma) \cup B(\Gamma))$, $W(\Gamma)$, $B(\Gamma) (= B_i(\Gamma) \cup B_e(\Gamma))$, $B_i(\Gamma)$ and $B_e(\Gamma)$ for the set of vertices, white vertices, black vertices, internal black vertices and external black vertices. 
\end{notation}

\begin{notation}
For a plain graph $\Gamma$, define the \textit{order} $k(\Gamma)$ by $|E_{\theta}(\Gamma)| - |W(\Gamma)|$. Write $g(\Gamma)$ for the first Betti number of $\Gamma$. If $\Gamma$ is connected, we have $g(\Gamma)-1 = |E(\Gamma)|-|V(\Gamma)|$. If $g(\Gamma) = g$, we say that $\Gamma$ is \textit{$g$-loop}. 

Arone and Turchin \cite{AT 1, AT 2} introduced \textit{Hodge degrees} $s$ and \textit{complexity} $t$ of graphs. The Hodge degree is defined by the number of solid components. The complexity is defined by the first Betti number of the graph obtained from an original graph by identifying the solid component with a single black (external) vertex. Hence we have
\begin{align*}
s(\Gamma) &= k(\Gamma) - g(\Gamma) + m(\Gamma) + g_{\eta}(\Gamma),\\
t (\Gamma) &= k(\Gamma),
\end{align*}
where $g_{\eta}(\Gamma)$ is the first Betti number of solid components and $m(\Gamma)$ is the number of components of $\Gamma$.
\end{notation}

\begin{definition}[Degrees]
The degrees of dashed edges, solid edges, white vertices and black vertices are $n-1$, $j-1$, $-n$, $-j$ respectively. 
The degree $|\Gamma| = deg(\Gamma)$ of a graph $\Gamma$ is defined by the sum of degrees
\[
|\Gamma| = (n-1) |E_{\theta}(\Gamma)| 
+ (j-1) |E_{\eta}(\Gamma)| 
-n |W(\Gamma)| 
-j |B(\Gamma)|.
\]
\end{definition}
 
 \begin{definition}[Labels and orientations]
 A label of a plain graph consists of a choice of an ordering of the set
 \[
 o(\Gamma) = E(\Gamma) \cup V(\Gamma)
 \]
 and a choice of orientations of edges. Each label gives an orientation of the underlying graph so that it depends only on parities of $n$ and $j$. More precisely, if we change the ordering of two elements of degree $d$ and $d^{\prime}$ of $o(\Gamma)$, the orientation of a graph is changed by $(-1)^{d d^{\prime}}$. If we change the orientation of an edge of degree $d$, the orientation is changed by $(-1)^{d+1}$.
 \end{definition}

\begin{definition} 
As a graded vector space, the space $PGC_{n,j}$ is defined by 
\[
PGC_{n,j}= \frac{\mathbb{Q} \{\text{connected labeled admissible plain graphs} \}}{\text{orientation relations, double edges, loop edges} }.
\]
Orientation relations identify two elements $\Gamma$ and $\Gamma^{\prime}$ (resp. $\Gamma$ and $-\Gamma^{\prime})$ if $\Gamma$ and $\Gamma^{\prime}$ are two labeled graphs with the same underlying graph and with the same (resp. opposite) orientation. The other relations cancel double edges and loop edges. 
\end{definition}

\begin{definition}[The differential on $PGC_{n,j}$]
\label{diffofPGC}
The differential $d_{PGC}$ on $PGC_{n,j}$ is defined by the sum of contractions of admissible edges. (See Figure  \ref{figofdifferential}.)
\[
d_{PGC} (\Gamma)  =  \sum_{\substack{e \in E^{a}(\Gamma)}} \sigma(e) \Gamma/e .
\]
Here, the set $E^{a}(\Gamma) \subset E(\Gamma)$ of admissible edges consists of edges except for double edges, loop edges,  \textit{dashed chords} \dashededgeb{}{}{}: dashed edges which connect two (external) black vertices, 
and \textit{multiple edges}
\begin{tikzpicture}[x=1pt,y=1pt,yscale=1,xscale=1, baseline=-3pt, line width =1pt]
\draw  [fill={rgb, 255:red, 0; green, 0; blue, 0 }  ,fill opacity=1 ] (0,0) circle (2);
\draw  [fill={rgb, 255:red, 0; green, 0; blue, 0 }  ,fill opacity=1 ] (50,0) circle (2);
\draw  [dash pattern = on 3pt off 3 pt] (0,0) .. controls (20,5) and (30,5)  .. (50,0);
\draw  (0,0) .. controls (20,-5) and (30,-5)  .. (50,0);
\end{tikzpicture}:
pairs of a dashed edge and a solid edge which connect two (external) black vertices. The sign $\sigma(e)$, which depends only on parities of $n$ and $j$, is defined by
\[
\sigma(e) = (-1)^{|\Gamma| + 1}(-1)^{deg(V(\Gamma/e))}\Phi(e)\tau(e)\eta(e). 
\]
See Remark \ref{signexplanation} below and subsection \ref{Notation on orientations and signs} for the notation $\Phi(e)$, $\tau(e)$ and $\eta(e)$. We easily have that $d_{PGC}$ vanishes on relations of $PGC_{n,j}$ and that it preserves the order and the first Betti number of graphs. \footnote{Alternatively, we can include double edges and multiple edges in $E^{a}(\Gamma)$, because contraction of these edges produce loop edges.}

\begin{figure}[htpb]
   \centering
    \includegraphics [width =9cm] {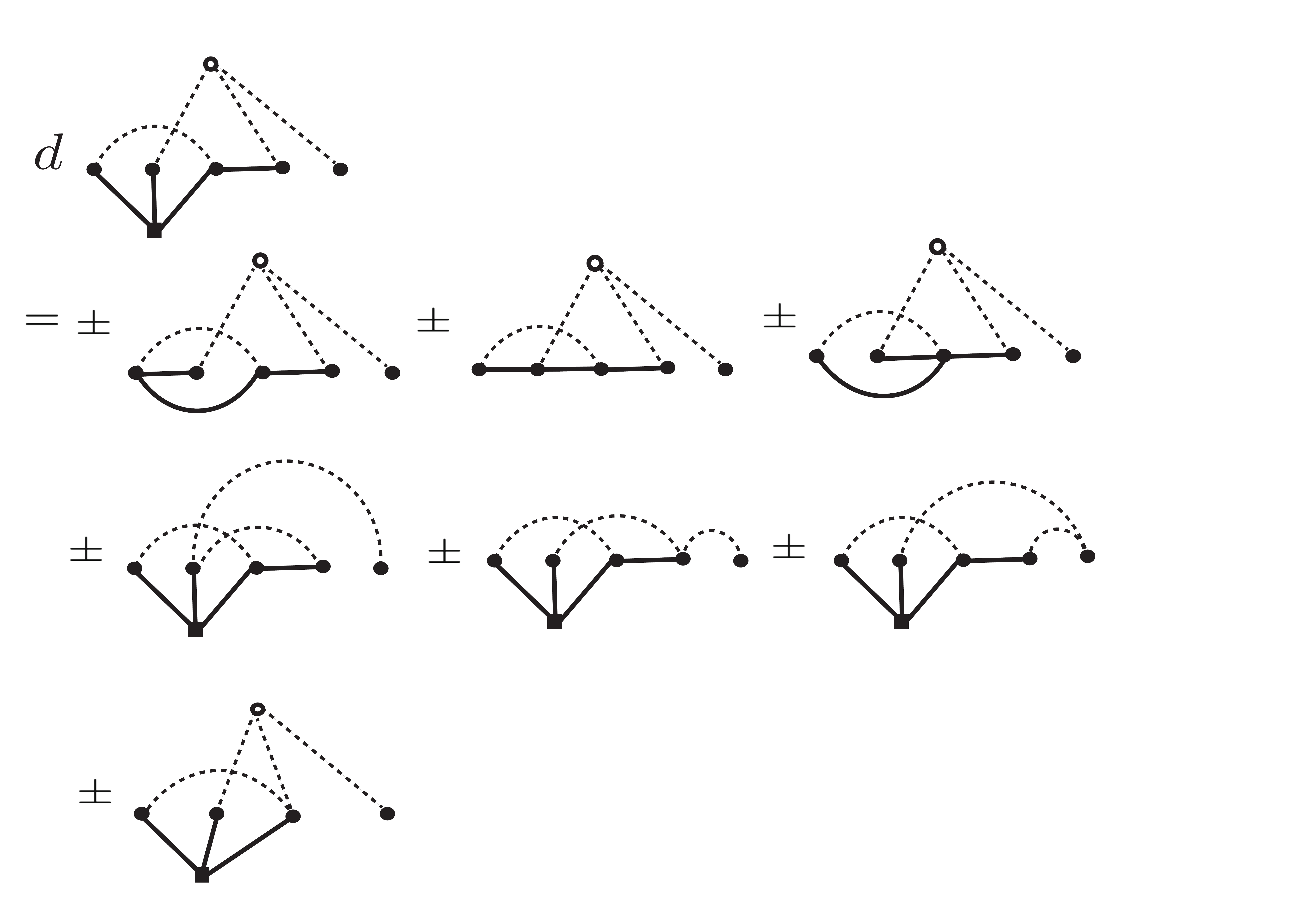}
     \caption{Example of the differential $d_{PGC}$}
      \label{figofdifferential}
\end{figure}
\end{definition}

\begin{rem}
\label{signexplanation}
Let $e = \{p, q\} (q>p)$ be an edge of a labeled plain graph $\Gamma$. Then, the sign $\sigma(e)$ is characterized as follows. Consider the ordered set
\[
\mu \sqcup o(\Gamma) = \mu \sqcup E(\Gamma)\sqcup \overline{V(\Gamma)}.
\]
Here, $\overline{V(\Gamma)}$ is the ordered set obtained by reversing the order of $V(\Gamma)$. Elements of $E(\Gamma)$ and $V(\Gamma)$ are graded. Each element of $E(\Gamma)$ is considered as an oriented edge. The variable $\mu$ at the head is a degree one element. Then,  $\sigma(e)$ is the sign which arises when permuting $\mu \sqcup  o(\Gamma)$ to
\[
E(\Gamma/e) \sqcup (\mu, (p, q), q) \sqcup \overline{V(\Gamma/e)}, 
\]
or equivalently, to
\[
E(\Gamma/e) \sqcup \overline{V(\Gamma/e)} \sqcup (\mu, (p, q), q) = (\mu, (p, q), q) \sqcup E(\Gamma/e) \sqcup \overline{V(\Gamma/e)}. 
\]
(Note that $(\mu, (p, q), q)$ has degree $0$). 
On the other hand, the sign convention in \cite{Sak} is taken so that $\sigma(e)$ is the sign which arises when permuting
\[
\mu \sqcup o^{\prime}(\Gamma) = \mu \sqcup V(\Gamma)\sqcup E(\Gamma).
\]
to
\[
V(\Gamma/e) \sqcup (\mu, q, (p, q)) \sqcup E(\Gamma/e).
\]
\end{rem}

\begin{lemma}
\label{proofofdifferential1}
$(PGC_{n,j}, d_{PGC})$ is a cochain complex: $d_{PGC}^2 = 0$.
\end{lemma}

\begin{definition}[The plain graph complex]
\label{The plain graph complex}
We call the complex $(PGC_{n,j}, d_{PGC})$ \textit{the plain graph complex}. This complex has a decomposition $PGC_{n,j} = \bigoplus_{k\geq 1} \bigoplus_{g \geq 0} PGC_{n,j} (k,g)$. This complex depends only on parities of $n$ and $j$ up to degree shifts. 
\end{definition}

\begin{proof}[proof of Lemma \ref{proofofdifferential1}]
Let $\Gamma$ be a labeled plain graph. Then
\begin{align*}
d^2 \Gamma = d\left( \sum_{e \in E^{a}(\Gamma)} \sigma(e, \Gamma) \Gamma/e \right)  = \sum_{e \in E^{a}(\Gamma)}  \sigma(e, \Gamma) d(\Gamma/e) 
\end{align*} 
Here, 
\begin{align*}
d(\Gamma/e) &=  \sum_{f \in E^{a}(\Gamma/e)} \sigma(f, \Gamma/e) ((\Gamma/e)/f) \\
& = \sum_{\substack{f \in E^{a}(\Gamma) \\ (e,f) \in P^{a}(\Gamma)}} \sigma(f, \Gamma/e) (\Gamma/(e \cup f)) \\
\end{align*}
Note that if $f \in E^{a}(\Gamma/e)$ then $f \in E^a(\Gamma)$.
The set $P^{a}(\Gamma)$ is the subset of $E^{a}(\Gamma) \times E^{a}(\Gamma)$ which consists of pairs of distinct edges $(e,f)$ such that $f \in E^{a}(\Gamma/e)$. As we see in Lemma \ref{symmetry of P}, this subset is symmetric. 
We have
\[
d^2 \Gamma  = \sum_{(e, f) \in P^{a}(\Gamma)} \sigma(e, \Gamma) \sigma(f, \Gamma/e) (\Gamma/(e \cup f)).
\]
The terms of $(e, f)$ and $(f,e)$ are canceled by each other by Lemma \ref{changing the order of the contraction}. 
\end{proof}

\begin{lemma}
\label{symmetry of P}
If $(e,f) \in P^{a}(\Gamma)$, so is $(f,e) \in P^{a}(\Gamma)$.
\end{lemma}
\begin{proof}
Assume $(e,f) \in P^{a}(\Gamma)$. Assume $e$ is not an admissible edge of $\Gamma/f$. Since  $e$ is an admissible edge of $\Gamma$, $e$ is not a loop edge of $\Gamma/f$. If $e$ is an edge of a double or multiple edge of $\Gamma/f$, then $f$ must be an edge of a double or multiple edge of $\Gamma/e$. If $e$ is a dashed chord of $\Gamma/f$, then $f$ must be a dashed chord of  $\Gamma/e$. In both cases, $f$ is not an admissible edge of $\Gamma/e$, which is a contradiction.
 \end{proof}

\begin{lemma}
\label{changing the order of the contraction}
Let $(e,f) \in P^{a}(\Gamma)$. Then $\sigma(e, \Gamma)\sigma(f, \Gamma/e) = - \sigma(f, \Gamma)\sigma(e, \Gamma/f)$.
\end{lemma}

\begin{proof}
We use the description of $\sigma(e)$ in Remark \ref{signexplanation}. Let $e = (p,q)$ $(q>p)$ and $f = (r,s)$ $(s>r)$.
The sign $\sigma(e, \Gamma)\sigma(f, \Gamma/e)$  (resp. $\sigma(f, \Gamma)\sigma(e, \Gamma/f)$) corresponds to the sign obtained by permuting the ordered set
\[
\nu \sqcup  \mu \sqcup E(\Gamma) \sqcup \overline{V(\Gamma)} \quad (\text{resp. $\mu \sqcup  \nu \sqcup E(\Gamma) \sqcup \overline{V(\Gamma)}$})
\]
to 
\begin{align*}
&E(\Gamma/ e \cup f) \sqcup \overline{V(\Gamma/ e \cup f)}  \sqcup (\nu, (r, s), s)) \sqcup (\mu, (p, q), q). \\
&\text{(resp. $E(\Gamma/  e \cup f) \sqcup \overline{V(\Gamma/  e \cup f)}  \sqcup (\mu, (p, q), q) \sqcup (\nu, (r, s), s))$}. \\
\end{align*}
Hence, the difference of $\sigma(e, \Gamma)\sigma(f, \Gamma/e)$ and $\sigma(f, \Gamma)\sigma(e, \Gamma/f)$ arises from the difference of $\nu \sqcup  \mu$ and  $\mu \sqcup  \nu$, which causes the negative sign. 
\end{proof}

We define another complex which is useful when comparing the plain graph complex with the hairy graph complex. 
\begin{definition}
Define the complex $PGC^{\prime}_{n,j}$ by 
\[
PGC^{\prime}_{n,j}= \frac{\mathbb{Q} \{\text{connected labeled admissible plain graphs} \}}{\text{orientation relations} }.
\]
That is, we allow double edges and loop edges. Note that when $n$ (resp. $j$) is odd, loop dashed edges (resp. loop solid edges) are canceled by orientation relations. On the other hand, when $n$ (resp. $j$) is even, double dashed edges (resp. double solid edges) are canceled by orientation relations.
\end{definition}

\begin{definition}
The differential on $PGC^{\prime}_{n,j}$ is defined in a way similar to $PGC_{n,j}$ by contractions of edges. We do not perform contractions on dashed chords or loop edges. However, we do perform contractions on (solid edges of) multiple edges and on any edges of double edges if they are not dashed chords. Contractions of multiple edges or double edges produce loop edges. \footnote{Note that a loop edge at an external black vertex of a plain graph corresponds to a double loop of a BCR graph in \cite{Sak, Yos 1}.}
\end{definition}

\begin{lemma}
\label{proofofdifferential2}
$(PGC^{\prime}_{n,j}, d_{PGC^{\prime}})$ is a cochain complex: $d_{PGC^{\prime}}^2 = 0$.
\end{lemma}

\begin{proof}
The proof is similar to Lemma \ref{proofofdifferential1}. It is enough to show that Lemma \ref{symmetry of P} holds when $E^a(\Gamma)$ is replaced with the set of edges except for dashed chords and loop edges. Assume $(e,f) \in P^{a}(\Gamma)$. Assume $e$ is not an admissible edge of $\Gamma/f$. If $e$ is a loop edge of $\Gamma/f$, then $f$ must be a loop edge of  $\Gamma/e$.  If $e$ is a dashed chord of $\Gamma/f$, then $f$ must be a dashed chord of  $\Gamma/e$. 
\end{proof}

Again, the complex $PGC^{\prime}_{n,j}$ only depends on parities of $n$ and $j$ up to degree shifts. It has a decomposition with respect to the order $k$ and the first Betti number $g$.

There is a cochain map $PGC^{\prime}_{n,j} \rightarrow PGC_{n,j}$ defined by the projection. We show in Theorem \ref{relationwithAroneTurchincomplex} that this map is a quasi-isomorphism.

\subsection{Arone--Turchin's Hochschild-type complex $HH_{n,j}$}
\label{The Hochschild-type complex $HH_{n,j}$}

The complex $HH_{n,j}$ is defined by Arone and Turchin in \cite{AT 2}.

\begin{definition} [Chord graphs]
A plain graph is a \textit{chord graph} if it has neither white vertices nor internal black vertices.  A chord graph is called \textit{admissible} if it is admissible as a plain graph; that is, every external black vertex has at least one dashed edge.
\end{definition}

Labels and orientations of chord graphs are defined by labels and orientations of plain graphs. 

\begin{definition} [Arnold relations of chord graphs]
An \textit{Arnold relation} of labeled chord graphs is one of the relations in the following.
\begin{itemize}

\item 

\begin{tikzpicture}[x=0.75pt,y=0.75pt,yscale=1,xscale=1, baseline=20pt, line width = 1pt] 
\draw [-Stealth]  (0,0)--(95,0) ;
\draw(50,0) node [anchor = north] {$(a)$} ;
\draw [-Stealth]  (100,0)--(50,48) ;
\draw (75,30) node [anchor = west]  {$(b)$} ;

 ;
\draw [fill={rgb, 255:red, 0; green, 0; blue, 0 }  ,fill opacity=1 ] (0,0) circle(4) node [anchor = south] {$i$}; 
\draw [fill={rgb, 255:red, 0; green, 0; blue, 0 }  ,fill opacity=1 ] (100,0) circle(4) node [anchor = south] {$j$}; ; 
\draw [fill={rgb, 255:red, 0; green, 0; blue, 0 }  ,fill opacity=1 ] (50,50) circle(4) node [anchor = south] {$k$} ; 
\end{tikzpicture}
+
\begin{tikzpicture}[x=0.75pt,y=0.75pt,yscale=1,xscale=1, baseline=20pt, line width = 1pt] 
\begin{scope}[xshift=100]
\draw [-Stealth]  (95,0)--(50,50);
\draw (75,30) node [anchor = west]  {$(a)$} ;
\draw [-Stealth]  (50,50)--(5,0);
\draw (5,30) node [anchor = west]  {$(b)$} ;

\draw [fill={rgb, 255:red, 0; green, 0; blue, 0 }  ,fill opacity=1 ] (0,0) circle(4) (0,0) circle(5) node [anchor = south] {$i$}; ; 
\draw [fill={rgb, 255:red, 0; green, 0; blue, 0 }  ,fill opacity=1 ] (100,0) circle(4) node [anchor = south] {$j$} ; 
\draw [fill={rgb, 255:red, 0; green, 0; blue, 0 }  ,fill opacity=1 ] (50,50) circle(4) node [anchor = south] {$k$} ; 
\end{scope}
\end{tikzpicture}
+
\begin{tikzpicture}[x=0.75pt,y=0.75pt,yscale=1,xscale=1, baseline=20pt, line width = 1pt] 
\begin{scope}[xshift=200]
\draw [-Stealth]  (0,0)--(95,0);
\draw(50,0) node [anchor = north] {$(b)$} ;
\draw [-Stealth]  (50,50)--(5,0);
\draw (5,30) node [anchor = west]  {$(a)$} ;
\draw [fill={rgb, 255:red, 0; green, 0; blue, 0 }  ,fill opacity=1 ] (0,0) circle(4) node [anchor = south] {$i$} ; 
\draw [fill={rgb, 255:red, 0; green, 0; blue, 0 }  ,fill opacity=1 ] (100,0) circle(4) node [anchor = south] {$j$} ; 
\draw [fill={rgb, 255:red, 0; green, 0; blue, 0 }  ,fill opacity=1 ] (50,50) circle(4) node [anchor = south] {$k$} ; 
\end{scope}
\end{tikzpicture}
= 0.

\item 
\begin{tikzpicture}[x=0.75pt,y=0.75pt,yscale=1,xscale=1, baseline=20pt, line width = 1pt] 
\draw [-Stealth, dash pattern = on 3 pt off 3 pt]  (0,0)--(95,0) ;
\draw(50,0) node [anchor = north] {$(a)$} ;
\draw [-Stealth, dash pattern = on 3 pt off 3 pt]  (100,0)--(50,48) ;
\draw (75,30) node [anchor = west]  {$(b)$} ;

 ;
\draw [fill={rgb, 255:red, 0; green, 0; blue, 0 }  ,fill opacity=1 ] (0,0) circle(4) node [anchor = south] {$i$}; 
\draw [fill={rgb, 255:red, 0; green, 0; blue, 0 }  ,fill opacity=1 ] (100,0) circle(4) node [anchor = south] {$j$}; ; 
\draw [fill={rgb, 255:red, 0; green, 0; blue, 0 }  ,fill opacity=1 ] (50,50) circle(4) node [anchor = south] {$k$} ; 
\end{tikzpicture}
+
\begin{tikzpicture}[x=0.75pt,y=0.75pt,yscale=1,xscale=1, baseline=20pt, line width = 1pt] 
\begin{scope}[xshift=100]
\draw [-Stealth, dash pattern = on 3 pt off 3 pt]  (95,0)--(50,50);
\draw (75,30) node [anchor = west]  {$(a)$} ;
\draw [-Stealth, dash pattern = on 3 pt off 3 pt]  (50,50)--(5,0);
\draw (5,30) node [anchor = west]  {$(b)$} ;

\draw [fill={rgb, 255:red, 0; green, 0; blue, 0 }  ,fill opacity=1 ] (0,0) circle(4) (0,0) circle(5) node [anchor = south] {$i$}; ; 
\draw [fill={rgb, 255:red, 0; green, 0; blue, 0 }  ,fill opacity=1 ] (100,0) circle(4) node [anchor = south] {$j$} ; 
\draw [fill={rgb, 255:red, 0; green, 0; blue, 0 }  ,fill opacity=1 ] (50,50) circle(4) node [anchor = south] {$k$} ; 
\end{scope}
\end{tikzpicture}
+
\begin{tikzpicture}[x=0.75pt,y=0.75pt,yscale=1,xscale=1, baseline=20pt, line width = 1pt] 
\begin{scope}[xshift=200]
\draw [-Stealth, dash pattern = on 3 pt off 3 pt]  (0,0)--(95,0);
\draw(50,0) node [anchor = north] {$(b)$} ;
\draw [-Stealth, dash pattern = on 3 pt off 3 pt]  (50,50)--(5,0);
\draw (5,30) node [anchor = west]  {$(a)$} ;
\draw [fill={rgb, 255:red, 0; green, 0; blue, 0 }  ,fill opacity=1 ] (0,0) circle(4) node [anchor = south] {$i$} ; 
\draw [fill={rgb, 255:red, 0; green, 0; blue, 0 }  ,fill opacity=1 ] (100,0) circle(4) node [anchor = south] {$j$} ; 
\draw [fill={rgb, 255:red, 0; green, 0; blue, 0 }  ,fill opacity=1 ] (50,50) circle(4) node [anchor = south] {$k$} ; 
\end{scope}
\end{tikzpicture}
= 0.

\end{itemize}

Here, the three graphs in each relation are labeled chord graphs and they coincide outside the drawn parts. 
\end{definition}

\begin{definition}
As a graded vector space, $HH_{n,j}$ is defined by 
\[
HH_{n,j}= \frac{\mathbb{Q} \{\text{connected labeled admissible chord graphs} \}}{\text{orientation relations, double edges, loop edges, Arnold relations} }.
\]
\end{definition}

\begin{definition}[The Hochschild-type complex]
Let $d_{HH}$ be the differential on $HH_{n,j}$ induced by $d_{PGC}$. We call the complex $(HH_{n,j}, d_{HH})$ the \textit{Hochschild-type complex}. 
\end{definition}

The reason why we call this complex ``Hochschild-type'' is explained in Remark \ref{whyHochschildcomplex}.

\subsection{The hairy graph complex $HGC_{n,j}$}
\label{The hairy graph complex $HGC_{n,j}$}

Refer to \cite {AT 2} for the dual of this complex.

\begin{definition} [Hairy graphs]
An admissible plain graph is a \textit{hairy graph}  if it has neither internal black vertices nor solid edges and if its external vertices have exactly one dashed edge
\begin{tikzpicture} [baseline = 0pt] 
 \draw [fill={rgb, 255:red, 0; green, 0; blue, 0}, fill opacity =1.0] (0, 0) circle (0.05); 
 \draw [dash pattern = on 3pt off 3 pt, line width = 1pt]  (0,0) -- (0,0.5);
 \end{tikzpicture}, called a \textit{hair}.  See Figure \ref{figofhairygraph}. Note that double edges and loop edges are allowed. (Both ends of these edges must be white.) 

\begin{figure}[htpb]
   \centering
    \includegraphics [width =7cm] {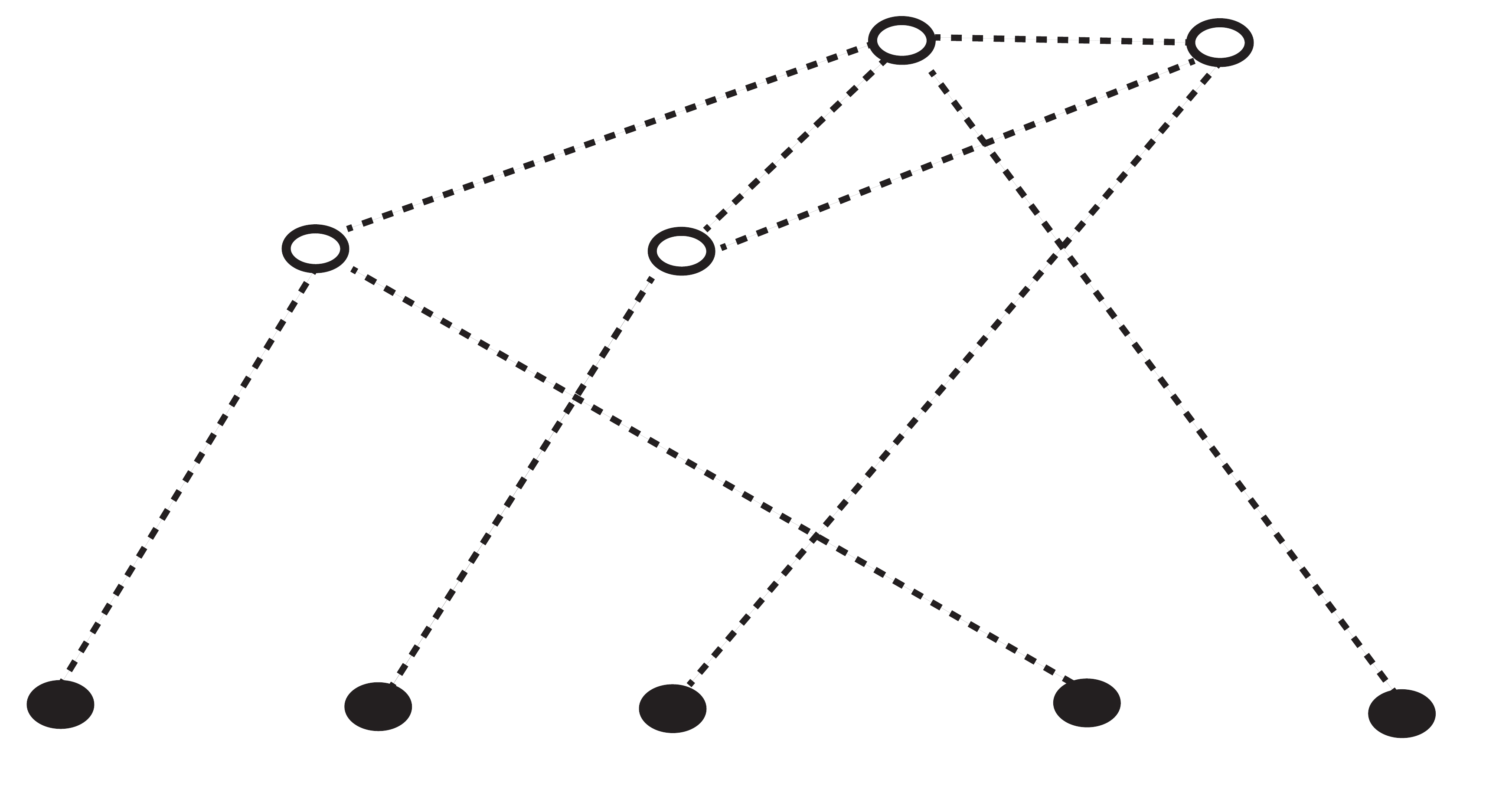}
    \caption{Example of a hairy graph}
    \label{figofhairygraph}
\end{figure}

\end{definition}

\begin{definition} 
As a graded vector space, 
\[
HGC_{n,j}= \frac{\mathbb{Q} \{\text{connected labeled hairy graphs} \}}{\text{orientation relations} }.
\]
\end{definition}

\begin{definition} [The differential on $HGC_{n,j}$]
The differential $d_{HGC}$ on $HGC_{n,j}$ is defined by the sum of contractions of edges except for loop edges and hairs. See Figure \ref{figofdifferential2}.

\begin{figure}[htpb]
   \centering
    \includegraphics [width =8cm] {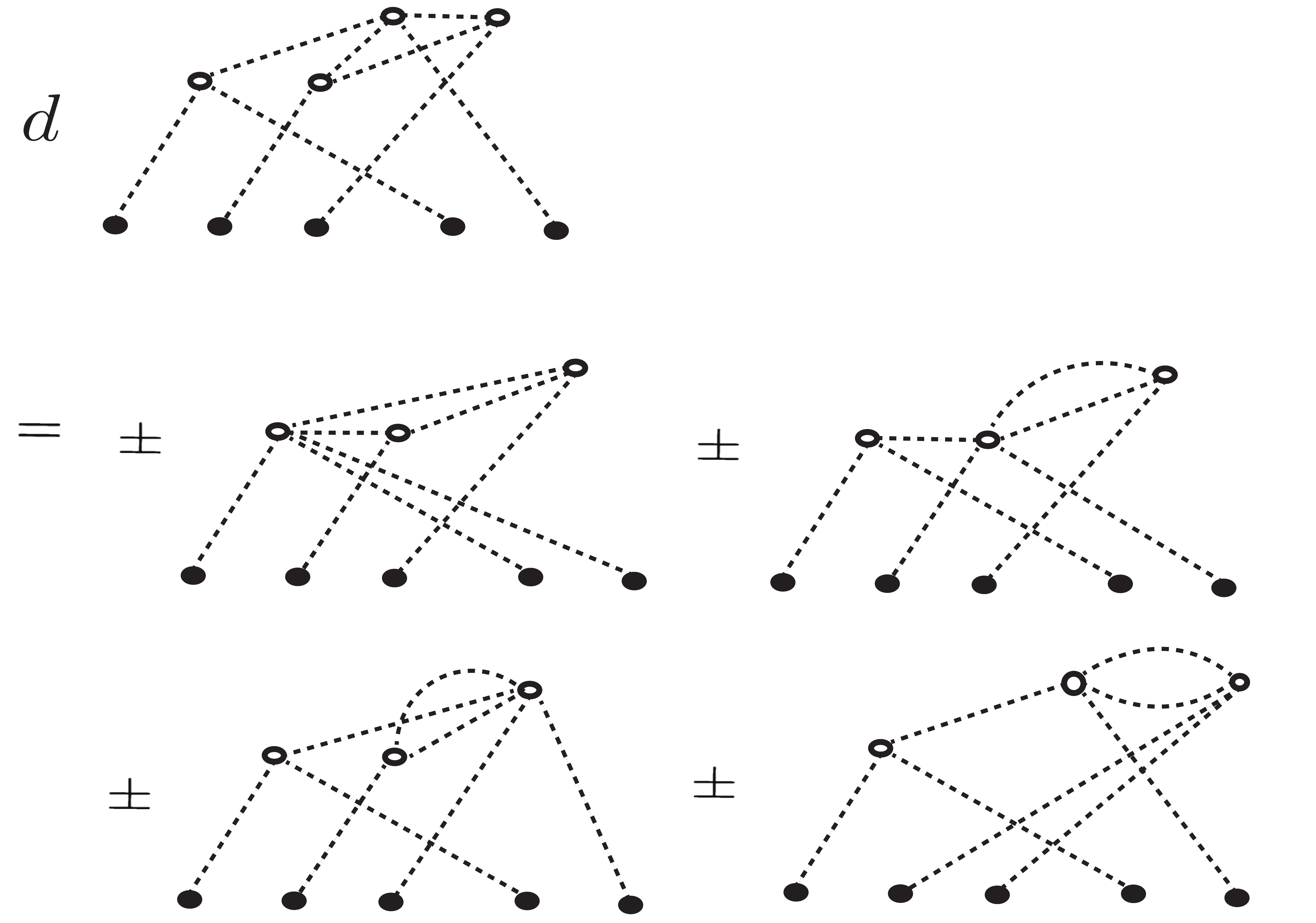}
     \caption{Example of the differential $d_{HGC}$}
      \label{figofdifferential2}
\end{figure}
\end{definition}

\begin{lemma}
\label{proofofdifferential3}
$(HGC_{n,j}, d_{HGC})$ is a cochain complex: $d_{HGC}^2 = 0$.
\end{lemma}

\begin{proof}
The proof is similar to Lemma \ref{proofofdifferential1} and \ref{proofofdifferential2}. Note that if $e$ is a hair of $\Gamma/f$, $e$ is a hair of $\Gamma$. 
\end{proof}

\begin{definition}[The hairy graph complex]
We call the complex $(HGC_{n,j}, d_{HGC})$ the \textit{hairy graph complex}. It only depends on parities of $n$ and $j$ up to degree shifts. It has a decomposition with respect to the order $k$ and the first Betti number $g$.
\end{definition}

There is a cochain map $PGC^{\prime}_{n,j} \rightarrow HGC_{n,j}$ defined by the projection. 
\subsection{Notation on orientations and signs}
\label{Notation on orientations and signs}

Recall that a labeled plain graph has white vertices of degree $(-n)$, (external/internal) black vertices of degree $(-j)$, dashed edges of degree $(n-1)$ and solid edges of degree $(j-1)$.
We introduce some notation on orientations and signs that are used in the definition of the differential of graph complexes; Definition \ref{diffofPGC}, for example. This notation is analogous to that in \cite{LV}.

Without loss of generality, we assume that the ordering of $E(\Gamma)$ is chosen so that solid edges are the first. Similarly,  we assume that the ordering of $V(\Gamma)$ is chosen so that black vertices are the first.

\begin{notation}
Let $\Gamma$ be a labeled plain graph. Let $e$ be an edge of a graph. Define the sign $\tau(e, \Gamma) = \tau(e)$ by the sign that arises when $e$ jumps to the end of the ordered set $E(\Gamma)$. More specifically, if $e$ is the $i$-th edge and $e$ is solid, 
\begin{align*}
\tau(e) 
& = (-1)^{(j-1)(i-1) + (j-1)deg(E(\Gamma/e))}\\
&=  (-1)^{(j-1)(i-1) + (j-1)(|E_{\eta}(\Gamma)|-1) + (j-1)(n-1)|E_{\theta}(\Gamma)|}.\\
\end{align*}
If $e$ is the $i$-th edge and $e$ is dashed, 
\begin{align*}
\tau(e) 
& = (-1)^{(n-1)(i-|E_{\eta}(\Gamma)|-1) + (j-1)(n-1)|E_{\eta}(\Gamma)| + (n-1)deg(E(\Gamma/e))} \\
&=  (-1)^{(n-1)(i-|E_{\eta}(\Gamma)|-1) + (j-1)(n-1)|E_{\eta}(\Gamma)| + (j-1)(n-1)|E_{\eta}(\Gamma)| + (n-1)(|E_{\theta}(\Gamma)|-1)}
\end{align*}
\end{notation}

\begin{notation}
Define the sign $\Phi(e, \Gamma) = \Phi(e)$  by the sign which arises when the vertex of $e$ with a larger label jumps to the end of the set $V(\Gamma)$. More precisely, if $e = \{p, q\}$, $q>p$  and $q$ is black,
\begin{align*}
\Phi(e) &= (-1)^ {j(q-1) + jdeg(V(\Gamma/e))} \\
& =  (-1)^ {j(q-1) + j(s-1) + njt} 
\end{align*}
If $e = (p, q)$, $q>p$  and $q$ is white,
\begin{align*}
\Phi(e) 
& = (-1)^ {n(q-s-1) + jns +  ndeg(V(\Gamma/e))} \\
&= (-1)^{n(q-s-1) + jns + jns + n(t-1)}
\end{align*}
 \end{notation}

\begin{rem}
Hence, if $n$ and $j$ are even, we have
\begin{equation*}
\tau(e) \Phi(e)= (-1)^{(i-1)+|\Gamma|+1} = (-1)^{i+|\Gamma|} 
\end{equation*}
If $n$ and $j$ are odd, 
\begin{equation*}
 \tau(e) \Phi(e) = (-1)^{(q-1)+|\Gamma|+1} = (-1)^{q + |\Gamma|} 
 \end{equation*}
 \end{rem}

\begin{notation}
Finally, define the sign $\eta(e)$ of $e=\{p,q\}$, $q>p$ by
\begin{equation*}
\eta(e) = 
\begin{cases}
1 & \text{if}\ e = (p,q), \\
(-1)^{j} & \text{if $e = (q,p)$ and $e$ is solid}, \\
(-1)^{n} & \text{if $e = (q,p)$ and $e$ is dashed}.\\
\end{cases}
\end{equation*}
\end{notation}

\begin{notation}
Let $\Gamma$ be a labeled plain graph. Let $S = \{s_1, \dots, s_{|S|}\}$ be an ordered subset of $V(\Gamma)$. Define the sign $\tau(S)$ by the sign which arises when edges of $\Gamma_S$ jump to the end of the set $E(\Gamma)$. Define the sign $\Phi(S)$  by the sign which arises when the vertices of $S$ except for the first $s_1$ jump to the end of the set $V(\Gamma)$.
\end{notation}

\subsection{Operadic descriptions of  $PGC_{n,j}$, $HH_{n,j}$ and $HGC_{n,j}$}
\label{Operadic descriptions of  $PGC_{n,j}$, $HH_{n,j}$ and $HGC_{n,j}$}

In this subsection, we describe the graph complexes in the previous subsections in terms of the graph cooperads $\text{Graph}_n(\bullet)$ or $H^{\ast}(\text{Conf}_{\bullet}(\mathbb{R}^n))$ as in \cite{AT 1, AT 2}. We can define $PGC_{n,1}$,  $PGC^{\prime}_{n,1}$, $HH_{n,1}$ and $HGC_{n,j}$ using the same operadic descriptions as in the case $j\geq 2$. (Recall that the definition of $\text{Graph}_n(\bullet)$ is different when $n=1$). The dual of $PGC^{\prime}_{n,1}$ is already introduced in \cite{Tur 2}, for example. First, we introduce several notation.

\begin{notation}[Normalizations]
Let $s_i: \text{Graph}_n(l-1) \rightarrow \text{Graph}_n(l)\ (1\leq i \leq l)$ be the map inserting an external black vertex as the $i$-th external vertex. 
We define the \textit{normalization} of $\text{Graph}_n(l)$ by
\[
N\text{Graph}_n(l) =  \text{Graph}_n(l)/ \sum_{1\leq i \leq l} Im (s_i).
\]
The normalization $N\overline{\text{Graph}}_n(l)$ and  $NH^{\ast}(\text{Conf}_l(\mathbb{R}^n))$ are defined similarly. 
\end{notation}

\begin{notation}
Write $\text{MGraph}_n(l)$ for the subspace of $\text{Graph}_n(l)$ generated by graphs which have $l$ external black vertices and their valence is exactly one. 
\end{notation}

\begin{rem}
Both $N\text{Graph}_n(\bullet)$ and $M\text{Graph}_n(\bullet)$ are cooperads. The decomposition on $N\text{Graph}_n(\bullet)$ is induced from that of $\text{Graph}_n(\bullet)$, namely, by decompositions of graphs. The decomposition on $M\text{Graph}_n(\bullet)$ is defined by decompositions of graphs along cut edges.
\end{rem}

\begin{notation}
The sign representation $(\text{sign}_l)$ is the $1$-dimensional representation of  $\Sigma_l$ which maps odd permutations to $\times (-1)$ and maps even permutations to $ \times 1$. We assume the element of $(\text{sign}_l)$ has degree $(-l)$.
\end{notation} 

By inspection, we have the following. 

\begin{prop}[Operadic descriptions]
The plain graph complex $PGC_{n,j}$ is described as the subcomplex of 
\[
fPGC_{n,j} =  \bigoplus_{l\geq1} \left((\text{sign}_l)^{\otimes j} \otimes \overline{\text{Graph}}_j(l) \otimes N\overline{\text{Graph}}_n(l) \right) /\Sigma_l, 
\]
generated by connected graphs. The symmetric group $\Sigma_l$ acts diagonally. 
The differential $d_{PGC}$ is decomposed into three pairwise (anti-)commutative differentials
\[
d_{PGC} = d_{\overline{\text{Graph}}_j} + d_{\overline{\text{Graph}}_n} + d_\tau.
\]
The term $d_\tau$ corresponds to contractions of solid chords \solidedged{}{}{} (that is, edges between external vertices) of $\text{Graph}_j$.

Similarly, the other complex $PGC_{n,j}^{\prime}$ is described as the subcomplex of
\[
fPGC_{n,j}^{\prime} = \bigoplus_{l\geq1} \left((\text{sign}_k)^{\otimes j} \otimes \text{Graph}_j(l) \otimes N\text{Graph}_n(l) \right) /\Sigma_l, 
\]
generated by connected graphs. 
\end{prop}

\begin{prop}
The complex $HH_{n,j}$ is the subcomplex of 
\[
fHH_{n,j} = \bigoplus_{l\geq1} \left((\text{sign}_l)^{\otimes j} \otimes H^{\ast}(\text{Conf}_l(\mathbb{R}^j)) \otimes NH^{\ast}(\text{Conf}_l(\mathbb{R}^n))\right) /\Sigma_l,
\]
generated by connected graphs. By applying Poincar\'e duality, $fHH_{n,j}$ has also the expression 
\[
fHH_{n,j} = \bigoplus_{l \geq 1} \overline{H}_{-\ast}(\text{Conf}_l(\mathbb{R}^j)) \otimes N\left(H^{\ast}(\text{Conf}_l(\mathbb{R}^n))\right)/\Sigma_l, 
\]
where $\overline{H}_{-\ast}(\text{Conf}_k(\mathbb{R}^j))$ is the compactly supported homology.
\end{prop}

\begin{rem}
 When $j=1$,  we say graphs of $fPGC_{n,j}$, $fPGC^{\prime}_{n,j}$ or $fHH_{n,j}$ are connected if they are not decomposed into two graphs when an arbitrary point on the oriented line is cut. 
\end{rem}

\begin{rem}
\label{whyHochschildcomplex}
Note that if $j = 1$, the complex $fHH_{n,j}$ is the dual of the \textit{Hochschild complex} of the Poisson operad $Poiss_{n-1}$, which completely describes the rational homology of $\overline{\mathcal{K}}_{n,j}$ when $n \geq 4$. See \cite {Sin 1, LTV}.
\end{rem}

\begin{prop} 
The hairy graph complex $HGC_{n,j}$ is described as the subcomplex of
\[
fHGC_{n,j} = \bigoplus_{l\geq1} \left((\text{sign}_l)^{\otimes j}  \otimes M\text{Graph}_n(l) \right) /\Sigma_l, 
\]
generated by connected graphs. 
\end{prop}

The notation $fHGC_{n,j}$ (\textit{full hairy graph complex}) is taken from \cite{FTW 1}, for example.

\subsection{Relationship between $PGC_{n,j}$, $HH_{n,j}$  and $HGC_{n,j}$}

Using the operadic description of graph complexes in the previous section, we prove Theorem \ref{main theorem 3 on hidden faces} and Theorem \ref{main theorem 4 on hidden faces}. 

\begin{theorem}[Theorem \ref{main theorem 3 on hidden faces}]
\label{relationwithAroneTurchincomplex}
The projections $PGC^{\prime}_{n,j} \rightarrow PGC_{n,j}$ and  $PGC_{n,j} \rightarrow HH_{n,j}$ are quasi-isomorphisms.
\end{theorem}
\begin{proof}
Recall that there are quasi-isomorphisms $\text{Graph}_n(l) \rightarrow H^{\ast}(\text{Conf}_l(\mathbb{R}^n))$ and $\overline{\text{Graph}}_l(\mathbb{R}^n) \rightarrow H^{\ast}(\text{Conf}_l(\mathbb{R}^n))$.
These maps induce projections $PGC^{\prime}_{n,j} \rightarrow HH_{n,j}$ and $PGC_{n,j} \rightarrow HH_{n,j}$. 
Consider the spectral sequence computing $H^{\ast}(PGC^{\prime}_{n,j})$, filtered by the number of external black vertices.
Then, the differential $d_0$ on the $E_0$-page coincides with the sum of the differentials of $\text{Graph}_n$ and  $\text{Graph}_j$.
Hence, the $E_1$-page coincides with $HH _{n,j}$. 
We can show that this spectral sequence is bounded and hence convergent, by using the decomposition with respect to the first Betti number $g$ and the order $k$.
Therefore, the projection $H^{\ast}(PGC^{\prime}_{n,j}) \rightarrow HH_{n,j}$  is a quasi-isomorphism. Similarly, the projection $H^{\ast}(PGC_{n,j}) \rightarrow HH_{n,j}$ is a quasi-isomorphism. 
In particular, the projection $PGC^{\prime}_{n,j} \rightarrow PGC_{n,j}$ is a quasi-isomorphism. 
This proof is similar to Theorem \ref{maintheoremrestate}. 
\end{proof}

\begin{rem}
Fixing the order $k$ and the first Betti number $g$, we see that the spectral sequence in the proof of Theorem \ref{relationwithAroneTurchincomplex} collapses at the $E_2$-page.
Let $p$ be the negative of the number of external black vertices.
Then, for a graph of $HH_{n,j}$, we have
\begin{align*}
k &= |E_{\theta}|-|W| = |E_{\theta}|\\
g &= |E_{\theta}|+|E_{\eta}|-|B|-|W|+1 \\
&=  |E_{\theta}|+|E_{\eta}|-|B_e|+1\\
-p &=  |B_e| \\
p+q &= (n-1)|E_{\theta}| +(j-1)|E_{\eta}|-n|W|-j|B| \\
&= (n-1)|E_{\theta}| +(j-1)|E_{\eta}|-j|B_e|.
\end{align*}
Hence,
\begin{align*}
p+q &= p+k(n-j-2)+(g-1)(j-1).
\end{align*}
Therefore, fixing $k$ and $g$, the diagrams on the $E_1$-page are supported on a single horizontal line. Hence, the same is true for the $E_2$-page. 
Since every differential $d_r (r \geq 2)$ changes the horizontal coordinate, all higher differentials vanish.
A similar argument appears in the proof of Lemma \ref{collapseofPGC}.
\end{rem}

Observe that the complexes $PGC^{\prime}_{n,j}$ and  $PGC_{n,j}$ have a decomposition with respect to the homotopy type of the solid part. Since the solid parts of graphs of $HH_{n,j}$ must be forests, we have the following. 

\begin{cor}
Only graphs whose solid parts are forests contribute to the cohomology of $PGC^{\prime}_{n,j}$ and  $PGC_{n,j}$. 
\end{cor}

As for $PGC^{\prime}_{n,j}$ and $HGC_{n,j}$, we first compare them on the \textit{top degree} with respect to a new grading called \textit{defect}. 

\begin{definition}
Let $\Gamma$ be a plain graph. Define the \textit{defect} of a vertex $v$ of $\Gamma$ by
\begin{equation*}
l(v) = \begin{cases}
|E_{\theta}(v)| -3 & (v\in W(\Gamma)) \\
|E_{\theta}(v)|-1 & (v\in B(\Gamma)).\\
\end{cases}
\end{equation*}
Here $|E_{\theta}(v)|$ is the number of dashed (half) edges that $v$ has. Note that an admissible plain graph is allowed to have an internal black vertex of negative defect, which is $-1$. 
The \textit{defect} $l(\Gamma)$ of $\Gamma$ is defined as $l(\Gamma) = \sum_{v\in V(\Gamma)}l(v)$ and is equal to $2 |E_{\theta}(\Gamma)| - 3|W(\Gamma)| - |B(\Gamma)|$. 
\end{definition}

The original degree of $\Gamma$ is determined by the order, the first Betti number and the defect. 

\begin{lemma}
\label{degreeanddefect}
Let $\Gamma$ be a connected $g$-loop plain graph of order $k$ and defect $l$.
Then the degree $|\Gamma|$ of a plain graph is equal to 
\[
|\Gamma| = k(n-j-2) + (g-1) (j-1) + l.
\]
\end{lemma}

\begin{proof}
\begin{align*}
|\Gamma| &= (n-1)|E_{\theta}| + (j-1)|E_{\eta}| - n |W| - j |B| \\
& = (|E_{\theta}|-|W|)(n-j-2) + (|E_{\theta}| + |E_{\eta}| -|W| - |B|)(j-1) + (2|E_{\theta}|-3|W|-|B|) \\
& =   k(n-j-2) + (g-1)(j-1) + l.
\end{align*}
\end{proof}

\begin{definition}[The top degree cohomology]
\label{defoftopdegreecohomology}
Let $PGC^l_{n,j}$ be the subspace of $PGC_{n,j}$ generated by graphs of defect $l$. Then $PGC_{n,j}$ is a cochain complex graded by the defect $l$. Define the top degree cohomology $H^{top} (PGC_{n,j})$ by the $0$th cohomology with respect to this grading. A graph cocycle of top defect is often called a top graph cocycle.  Define $H^{top} (PGC^{\prime}_{n,j})$, $H^{top} (HH_{n,j})$ and $H^{top} (HGC_{n,j})$ similarly. Note that if $l \leq -1$,  then $HGC^l_{n,j}$ and $HH^l_{n,j}$ are zero, because their graphs have no internal black vertices. 
\end{definition}

The following propositions are the aim of Section \ref{section top cohomology}. 

\begin{restatable}{theorem}{weakerrelationwithhairygraphcomplex}
\label{weakerrelationwithhairygraphcomplex}
Assume that $n-j$ is even and $j \geq 2$. Then, the projection $PGC_{n,j}^{\prime} \rightarrow HGC_{n,j}$ induces the surjective map $H^{top} (PGC_{n,j}^{\prime}) \rightarrow H^{top} (HGC_{n,j})$. 
\end{restatable}

\begin{restatable}{prop}{weakerrelationwithhairygraphcomplexb}
\label{weakerrelationwithhairygraphcomplexb}
Assume that $n-j$ is even and $j \geq 2$. Let $H$ be a top graph cocycle of $HGC_{n,j}$. Then, there is a lift $H^{\prime}$ to $PGC_{n,j}^{\prime}$ which consists of graphs whose solid part is a disjoint union of broken lines. 
\end{restatable}

Theorem \ref{weakerrelationwithhairygraphcomplex} is much weaker than Theorem \ref{relationwithhairygraphcomplex} below. However, Theorem \ref{weakerrelationwithhairygraphcomplex} is shown more directly while Theorem \ref{relationwithhairygraphcomplex} is shown by using \textit{Koszul duality theory}.
This direct proof in Section \ref{section top cohomology} allows us to provide a more explicit result, Proposition \ref{weakerrelationwithhairygraphcomplexb}.

\begin{rem}
In \cite{Tur 2}, Turchin showed that $^{\ast}HGC_{n,1} \rightarrow\,^{\ast}PGC^{\prime}_{n,1}$ is a quasi-isomorphism for any $n$.  Hence we have that $PGC^{\prime}_{n,1} \rightarrow HGC_{n,1}$ is a quasi-isomorphism. 
Note that the isomorphism $H_{top}(^{\ast}HGC_{n,1}) \rightarrow H_{top}(^{\ast}PGC^{\prime}_{n,1})$ for odd $n$ is first established by Bar-Natan \cite{Bar}. 
\end{rem}

The complex $fHH_{n,j}$ is introduced by Arone and Turchin  \cite{AT 2} as a complex with the same origin as $fHGC_{n,j}$, that is, the derived mapping space between $\Omega$-modules:
\[
\underset{\Omega}{\text{Rmod}}^h (\widetilde{H}_{\ast}(S^{j\bullet}), \overline{H}_{\ast}(C_{\bullet}(\mathbb{R}^n)).
\]
The complexes $fHGC_{n,j}$ and $fHH_{n,j}$ are obtained from a fibrant replacement and a cofibrant replacement, respectively. 
So we have that $HH_{n,j}$ and $HGC_{n,j}$ are quasi-isomorphic. As a result, we have that $PGC_{n,j}$, $PGC^{\prime}_{n,j}, HGC_{n,j}$ and  $HH_{n,j}$ are quasi-isomorphic. 
Moreover, by carefully reading  \cite{AT 2}, we have the following.

\begin{theorem}[Theorem \ref{main theorem 4 on hidden faces}]
\label{relationwithhairygraphcomplex}
The projection $PGC^{\prime}_{n,j} \rightarrow HGC_{n,j}$ is a quasi-isomorphism. 
\end{theorem}

\begin{proof}
We introduce another graph complex $MGC_{n,j}$, defined by the subcomplex of 
\[
fMGC_{n,j} =  \bigoplus_{k\geq1} \left((\text{sign}_k)^{\otimes j} \otimes  H^{\ast}(\text{Conf}_k(\mathbb{R}^j))  \otimes N\text{Graph}_k(\mathbb{ R}^n) \right) /\Sigma_k, 
\]
generated by connected graphs. This is a complex obtained by taking both a cofibrant and a fibrant replacement of the derived mapping space. 
The projection $PGC^{\prime}_{n,j}  \rightarrow MGC_{n,j}$ is a quasi-isomorphism by a similar argument to Theorem \ref{relationwithAroneTurchincomplex}.

We show that the projection $MGC_{n,j} \rightarrow HGC_{n,j}$ $(j\geq2)$ is a quasi-isomorphism by \textit{Koszul duality theory}\footnote{Loday and Vallette's book \cite{LoVa} is a good reference for Koszul duality theory.}. 
Let $Com(\bullet)$, $Lie(\bullet)$, $Pois(\bullet)$ be the operads of (non-unital) commutative algebras, Lie algebras and Poisson algebras, respectively.
Let $Com^{c}(\bullet)$,  $Lie^{c}(\bullet)$, $Pois^{c}(\bullet)$ be the linear dual cooperads. 
Recall that elements of  $Lie^{c}(\bullet)$ are represented by trees (modulo Arnold relations) while elements of $Pois^{c}(\bullet)$ are represented by forests. 
It is known that the operad $Com$ is \textit{Koszul} and its \textit{Koszul dual} is $Lie$. So we have
\[
Lie^{c}[1] \circ Com \simeq \mathbb{R}, 
\]
as symmetric sequences over the category of cochain complexes. 
Here, $Lie^{c}[1]$ is the operadic suspension of $Lie^{c}$. Elements of $Lie^{c}[1]$ are represented by trees whose edges have degree $-1$. The differential of the left-hand side is defined using the (co)multiplication induced by the (co)operads and the natural degree $+1$ morphism $Lie^{c}[1] \rightarrow Com$. This differential is also described as the sum of edge contractions.
The right-hand side is seen as the symmetric sequence
\[
\mathbb{R}(\bullet) = 
\begin{cases*}
0 & ($\bullet \neq 1$) \\
\mathbb{R} & ($\bullet = 1$)
\end{cases*}
\]
with the trivial differential.

Since 
\[
H^{\ast}(\text{Conf}_{\bullet}(\mathbb{R}^j)) \cong Pois^{c}(\bullet) \cong (Com^{c} \circ Lie^{c})(\bullet), 
\]
we have
\[
fMGC_{n,j} =  \bigoplus_{k\geq1} \left((\text{sign}_k)^{\otimes j} \otimes  (Com^{c} \circ Lie^{c})(k)  \otimes N\text{Graph}_n(k) \right) /\Sigma_k.
\]
Moreover, $N\text{Graph}_n(\bullet)$ is decomposed as
\[
N\text{Graph}_n(\bullet) \cong (Com \circ M\text{Graph}_n) (\bullet). 
\]
See Figure \ref{figofKoszulduality}. So we have
\begin{align*}
fMGC_{n,j}  &=  \bigoplus_{k\geq1} \left((\text{sign}_k)^{\otimes j} \otimes  (Com^{c} \circ Lie^{c})(k)  \otimes (Com \circ M\text{Graph}_n)(k) \right) /\Sigma_k. \\
& =  \bigoplus_{k\geq1} (\text{sign}_k)^{\otimes j} 
\otimes 
\left(
\bigoplus_{l^{\prime} \geq 1} Com^{c} (l^{\prime}) \otimes_{\Sigma_{l^{\prime}}}
       (
         \oplus Ind^{\Sigma_k}_{\Sigma_{j_1} \cdots \Sigma_{j_{l^{\prime}}}} Lie^{c}_{j-1}(j_1) \otimes \cdots \otimes Lie^{c}_{j-1}(j_{l^{\prime}})
       )
\right)\\
&\otimes
\left(
\bigoplus_{l \geq 1} 
       (
         \oplus Ind^{\Sigma_l}_{\Sigma_{i_1} \cdots \Sigma_{i_{k}}} Com(i_1) \otimes \cdots \otimes Com(i_k)
       )
       \otimes_{\Sigma_{l}} M\text{Graph}_n(l)
\right) / \Sigma_k \\
&= \bigoplus_{l\geq1} M\text{Graph}_n(l) \otimes \bigg(\bigoplus_{l^{\prime} \geq 1} Com^{c}(l^{\prime}) \otimes \Big( \bigoplus_{k\geq1} sign^{\otimes j}_k \otimes   \\
                         &  (\oplus Ind ^{\Sigma_k}_{\Sigma_{j_1}, \cdots, \Sigma_{j_{l^{\prime}}}} Lie^{c}_{j-1}(j_1) \otimes \cdots \otimes Lie^{c}_{j-1}(j_{l^{\prime}}))
                          \otimes (\oplus Ind ^{\Sigma_l}_{\Sigma_{i_1}, \cdots, \Sigma_{i_k}} Com(i_1) \otimes \cdots \otimes Com(i_k))
          \Big)
\bigg)\\
& / \Sigma_l, \Sigma_{l^{\prime}}, \Sigma_k.
\end{align*}
($Ind$ stands for the induced representation. $(j_1,\ldots,j_{l'})$ and $(i_1,\ldots,i_k)$ run over all tuples satisfying
$j_1+\cdots+j_{l'}=k$
and
$i_1+\cdots+i_k=l$,
respectively.) By changing the labeling convention in an equivalent way, we have
\begin{align*}
fMGC_{n,j}  &=   \bigoplus_{l\geq1} M\text{Graph}_n(l) \otimes \bigg(\bigoplus_{l^{\prime} \geq 1} sign^{\otimes j}_{l^{\prime}} \otimes Com^{c}(l^{\prime}) \otimes   \\
                         & \Big( \bigoplus_{k\geq1}   (\oplus Ind ^{\Sigma_k}_{\Sigma_{j_1}, \cdots, \Sigma_{j_{l^{\prime}}}} Lie^{c}[1](j_1) \otimes \cdots \otimes Lie^{c}[1](j_{l^{\prime}}))
                          \otimes (\oplus Ind ^{\Sigma_l}_{\Sigma_{i_1}, \cdots, \Sigma_{i_k}} Com(i_1) \otimes \cdots \otimes Com(i_k))
          \Big)
\bigg)\\
& / \Sigma_l, \Sigma_{l^{\prime}}, \Sigma_k\\
& =  \bigoplus_{l\geq1} M\text{Graph}_n(l) \otimes \bigg(\bigoplus_{l^{\prime} \geq 1} sign^{\otimes j}_{l^{\prime}} \otimes Com^{c}(l^{\prime}) \otimes   \\
                         & \Big( \oplus Ind ^{\Sigma_l}_{\Sigma_{j_1}, \cdots, \Sigma_{j_{l^{\prime}}}} (Lie^{c}[1] \circ Com)(j_1) \otimes \cdots \otimes (Lie^{c}[1] \circ Com) (j_{l^{\prime}})
                                   \Big)
\bigg)\\
& / \Sigma_l, \Sigma_{l^{\prime}}. \\
\end{align*}
See Figure \ref{figofchangeofsign}.

\begin{figure}[htpb]
   \centering
    \includegraphics [width =8cm] {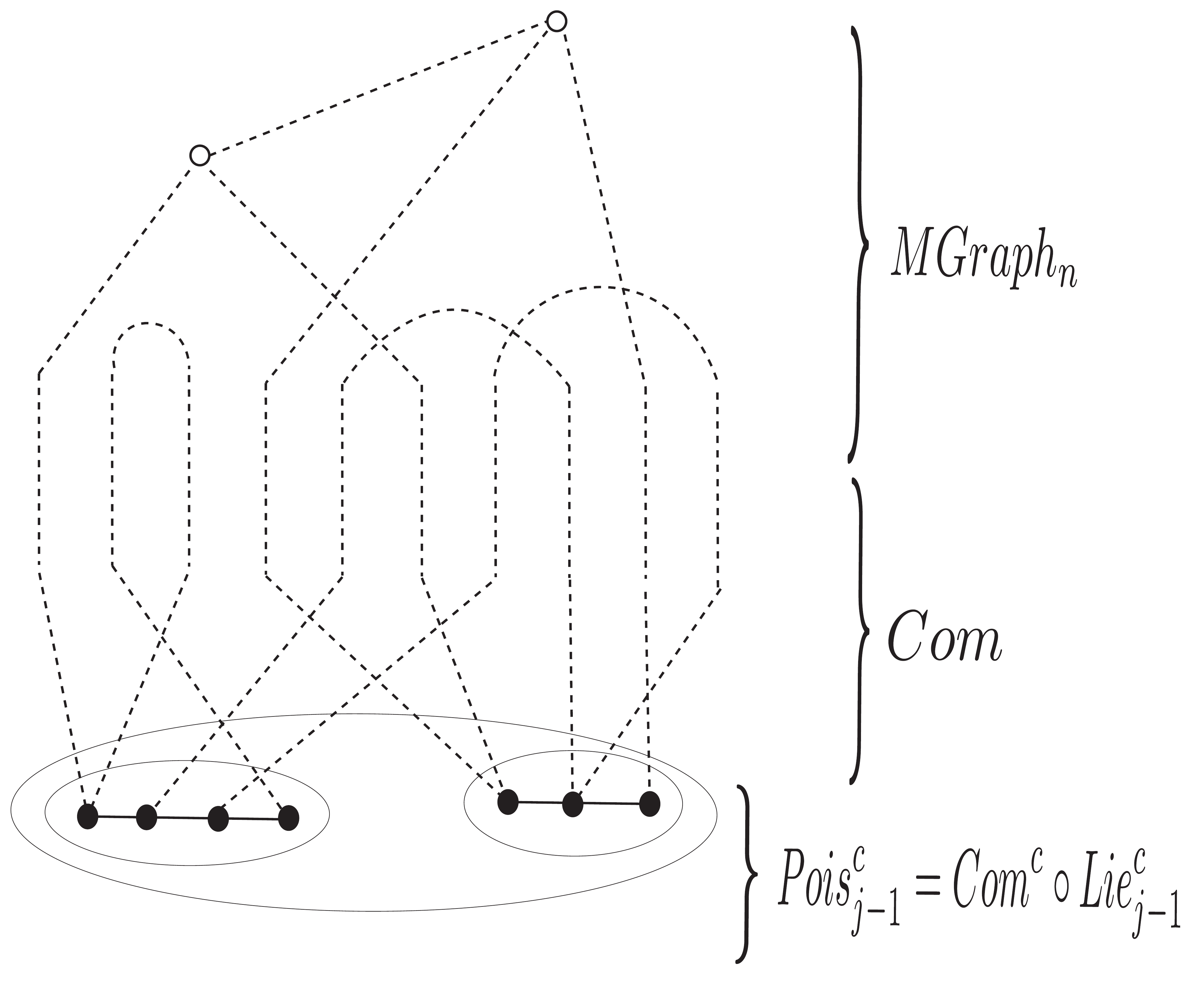}
     \caption{Decomposition of $N\text{Graph}_n(\bullet)$ ($l=10, k=7, l^{\prime} = 2$)}
      \label{figofKoszulduality}
\end{figure}

\begin{figure}[htpb]
\centering
 \includegraphics [width =16cm] {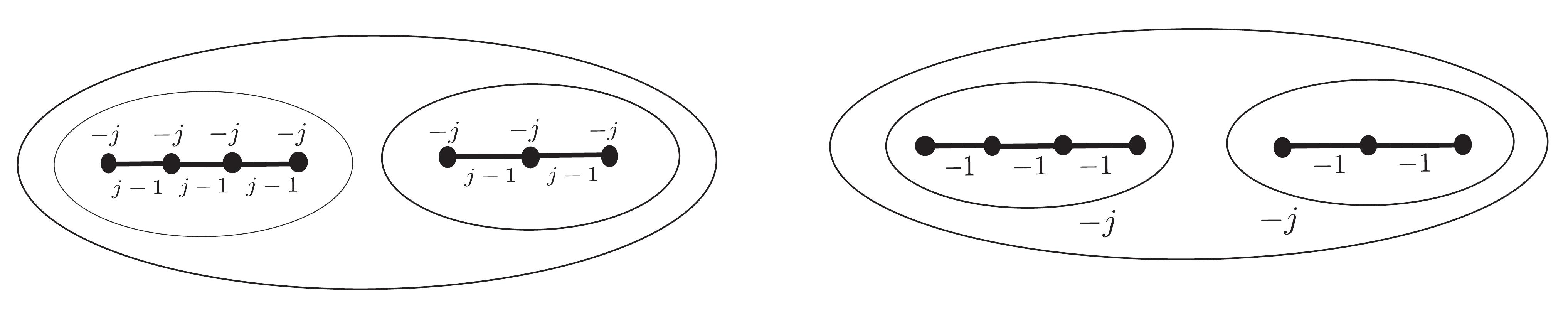}
  \caption{Change of sign convention}
   \label{figofchangeofsign}
\end{figure}

We consider the spectral sequence computing $H^{\ast}(fMGC_{n,j})$ filtered by the number of white vertices (of $M\text{Graph}_n$).
Then, the $d_0$ differential  coincides with $d_{Lie^{c}[1]\circ Com}$. Hence, the $E_1$-term coincides with 
\begin{align*}
&\bigoplus_{l\geq1} M\text{Graph}_n(l) \otimes \bigg(\bigoplus_{l^{\prime} \geq 1} sign^{\otimes j}_{l^{\prime}} \otimes Com^{c}(l^{\prime}) \otimes   
                          \Big( \oplus Ind ^{\Sigma_l}_{\Sigma_{j_1}, \cdots, \Sigma_{j_{l^{\prime}}}} \mathbb{R}(j_1) \otimes \cdots \otimes \mathbb{R} (j_{l^{\prime}})
                                   \Big)
\bigg)/ \Sigma_l, \Sigma_{l^{\prime}} \\
= & \bigoplus_{l\geq1} M\text{Graph}_n(l) \otimes \bigg( sign^{\otimes j}_{l} \otimes Com^{c}(l) \otimes  
                          \Big( \mathbb{R} \otimes \cdots \otimes \mathbb{R}
                                   \Big)
\bigg) / \Sigma_l \\
= & fHGC_{n,j}
\end{align*}
By using the decomposition with respect to the order and the first Betti number of graphs, we can show that this spectral sequence is bounded and hence convergent. 
\end{proof}
 
 \begin{rem}
Note that the $E_1$ term of the spectral sequence computing $H^{\ast}(MGC_{n,j})$ is $HGC_{n,j}$.
Fixing the order $k$ and the first Betti number $g$, we see this spectral sequence collapses at the $E_2$-page.
Let $-p$ be the number of white vertices.
Then, for a graph of $HGC_{n,j}$, we have
\begin{align*}
k &= |E_{\theta}|-|W|\\
g &= |E_{\theta}|+|E_{\eta}|-|B|-|W|+1 \\
& = |E_{\theta}|-|B|-|W|+1\\
-p &=  |W| \\
p+q &= (n-1)|E_{\theta}| +(j-1)|E_{\eta}|-n|W|-j|B| \\
&= (n-1)|E_{\theta}|-n|W|-j|B|.
\end{align*}
Hence, we have 
\begin{align*}
p+q &= p+k+g-1+k(n-j-2)+(g-1)(j-1).
\end{align*}
Therefore, fixing $k$ and $g$, the diagrams on the $E_1$-page are supported on a single horizontal line. Hence, the same is true for the $E_2$-page. 
Since every differential $d_r (r \geq 2)$ changes the horizontal coordinate, all higher differentials vanish.
A similar argument appears in the proof of Lemma \ref{collapseofPGC}.
\end{rem}

\begin{example}
When $n$ and $j$ are even, we have
\begin{align*}
H^{top}(HH_{n,j})(k=1,g=1) & = \mathbb{Q}
(
\begin{tikzpicture}[x=1pt,y=1pt,yscale=1,xscale=1, baseline=-3pt, line width =1pt]
\draw  [fill={rgb, 255:red, 0; green, 0; blue, 0 }  ,fill opacity=1 ] (0,0) circle (2);
\draw  [fill={rgb, 255:red, 0; green, 0; blue, 0 }  ,fill opacity=1 ] (50,0) circle (2);
\draw  [dash pattern = on 3pt off 3 pt] (0,0) .. controls (20,5) and (30,5)  .. (50,0);
\draw  (0,0) .. controls (20,-5) and (30,-5)  .. (50,0);
\end{tikzpicture}
),\\
H^{top}(PGC_{n,j})(k=1,g=1) & = \mathbb{Q}
(
\begin{tikzpicture}[x=1pt,y=1pt,yscale=1,xscale=1, baseline=-3pt, line width =1pt]
\draw  [fill={rgb, 255:red, 0; green, 0; blue, 0 }  ,fill opacity=1 ] (0,0) circle (2);
\draw  [fill={rgb, 255:red, 0; green, 0; blue, 0 }  ,fill opacity=1 ] (50,0) circle (2);
\draw  [dash pattern = on 3pt off 3 pt] (0,0) .. controls (20,5) and (30,5)  .. (50,0);
\draw  (0,0) .. controls (20,-5) and (30,-5)  .. (50,0);
\end{tikzpicture}
),\\
H^{top}(PGC^{\prime}_{n,j})(k=1,g=1) & = \mathbb{Q}
(
\begin{tikzpicture}[x=1pt,y=1pt,yscale=1,xscale=1, baseline=0pt, line width = 1pt]
\draw [dash pattern = on 3pt off 3 pt] (0,2) .. controls (-20,20) and (20,20)  .. (0,2);
\draw (0,0) circle (2);
\draw [dash pattern = on 2 pt off 2 pt] (0,-10) -- (0,-1);
\draw [fill = black] (0,-10) circle (2);
\end{tikzpicture}
- \begin{tikzpicture}[x=1pt,y=1pt,yscale=1,xscale=1, baseline=-3pt, line width =1pt]
\draw  [fill={rgb, 255:red, 0; green, 0; blue, 0 }  ,fill opacity=1 ] (0,0) circle (2);
\draw  [fill={rgb, 255:red, 0; green, 0; blue, 0 }  ,fill opacity=1 ] (50,0) circle (2);
\draw  [dash pattern = on 3pt off 3 pt] (0,0) .. controls (20,5) and (30,5)  .. (50,0);
\draw  (0,0) .. controls (20,-5) and (30,-5)  .. (50,0);
\end{tikzpicture}
),\\
H^{top}(HGC_{n,j})(k=1,g=1) & = \mathbb{Q}
(
\begin{tikzpicture}[x=1pt,y=1pt,yscale=1,xscale=1, baseline=0pt, line width = 1pt]
\draw [dash pattern = on 3pt off 3 pt] (0,2) .. controls (-20,20) and (20,20)  .. (0,2);
\draw (0,0) circle (2);
\draw [dash pattern = on 2 pt off 2 pt] (0,-10) -- (0,-1);
\draw [fill = black] (0,-10) circle (2);
\end{tikzpicture}
).\\
\end{align*}
\end{example}

\begin{example}
When $n$ and $j \geq 3$ are odd, we have
\begin{align*}
H^{top}(HH_{n,j})(k=2,g=1) & = \mathbb{Q}
(
\graphD{ } +\frac{1}{2}\graphDc{} 
),\\
H^{top}(PGC_{n,j})(k=2,g=1) & = \mathbb{Q}
(
 \graphE{} + \graphD{ } +\frac{1}{2}\graphDc{} 
),\\
H^{top}(PGC^{\prime}_{n,j})(k=2,g=1) & = \mathbb{Q}
(
\frac{1}{2} \graphF{} + \graphE{}  \\
&+ \graphD{ }  + \frac{1}{2} \graphDc{} 
),\\
H^{top}(HGC_{n,j})(k=2,g=1) & = \mathbb{Q}
(
\graphF{ }
).\\
\end{align*}
\end{example}

Note that 
\[
{PGC^{\prime}}^{(-1)}_{n,j}(k= 2, g=1) \ni \negativedefectgraph{}, \negativedefectgraphb{}
\]
And their images  by $d$ are  \graphD{} - 2 \graphDb{} and \graphC{} respectively. Hence, the above examples have another representative. For example, 
\begin{align*}
H^{top}(PGC_{n,j})(k=2,g=1) & = \mathbb{Q}
(
 \graphE{} + 2 \graphDb{} +\frac{1}{2}\graphDc{} 
). \\
\end{align*}

\begin{rem}
\label{differencebetweenBCRandplain}
If we follow the usual convention of the BCR complex, the cochain map $GC^{BCR}_{n,j} \rightarrow PGC_{n,j}$ defined by the identity is not a quasi-isomorphism. For example, when $n$ and $j$ are even, 
\begin{align*}
H^{top}(PGC_{n,j})(k=1,g=1) &= \mathbb{Q}
(
\begin{tikzpicture}[x=1pt,y=1pt,yscale=1,xscale=1, baseline=0pt, line width =1pt]
\draw  [fill={rgb, 255:red, 0; green, 0; blue, 0 }  ,fill opacity=1 ] (0,0) circle (2);
\draw  [fill={rgb, 255:red, 0; green, 0; blue, 0 }  ,fill opacity=1 ] (50,0) circle (2);
\draw  [dash pattern = on 3pt off 3 pt] (0,0) .. controls (20,5) and (30,5)  .. (50,0);
\draw  (0,0) .. controls (20,-5) and (30,-5)  .. (50,0);
\end{tikzpicture}
),\\
H^{top}(GC^{BCR}_{n,j})(k=1,g=1) &=0.
\end{align*}
(The differential of the graph \begin{tikzpicture}[x=1pt,y=1pt,yscale=1,xscale=1, baseline=-3pt, line width =1pt]
\draw  [fill={rgb, 255:red, 0; green, 0; blue, 0 }  ,fill opacity=1 ] (0,0) circle (2);
\draw  [fill={rgb, 255:red, 0; green, 0; blue, 0 }  ,fill opacity=1 ] (50,0) circle (2);
\draw  [dash pattern = on 3pt off 3 pt] (0,0) .. controls (20,5) and (30,5)  .. (50,0);
\draw  (0,0) .. controls (20,-5) and (30,-5)  .. (50,0);
\end{tikzpicture}
in $GC^{BCR}_{n,j}$ gives a double loop\begin{tikzpicture}[x=1pt,y=1pt,yscale=1,xscale=1, baseline=-3pt, line width =1pt]
\draw  [fill={rgb, 255:red, 0; green, 0; blue, 0 }  ,fill opacity=1 ] (0,0) circle (2);
\draw  [dash pattern = on 3pt off 3 pt] (0,0) .. controls (-40,20) and (40,20)  .. (0,0);
\draw  (0,0) .. controls (-20,10) and (20,10)  .. (0,0);
\end{tikzpicture},
but no graph in $GC^{BCR}_{n,j}$ cancels this contribution). One possible explanation is that this difference arises because the BCR complex is defined as a potential model of $\mathcal{K}_{n,j}$, not of $\overline{\mathcal{K}}_{n,j}$. 
However, in general, the difference between $H^{\ast}(GC^{BCR}_{n,j})$ and $H^{\ast}(PGC_{n,j})$ is still unknown. 
\end{rem}

\subsection{Remarks on the map at the top cohomology $H^{top}(PGC^{\prime}_{n,j}) \rightarrow H^{top}(HGC_{n,j})$}
\label{section top cohomology}

In this section, we see that the following is a consequence of \cite{Yos 2}. 

\weakerrelationwithhairygraphcomplex*

We see the corresponding statement in the language of graph homologies. 
We write $^\ast HGC_{n,j}$ , $^\ast PGC_{n,j}^{\prime}$ for the corresponding graph chain complexes
\footnote{Let $\Gamma$ be a connected labeled admissible plain graph without orientation reversing automorphisms. Let $^\ast \Gamma$ be the graph which is considered as an element of $^\ast PGC_{n,j}^{\prime}$. Then, as an element of $\text{Hom}(^\ast PGC^{\prime}_{n,j}, \mathbb{Q})$, $\Gamma$ is the dual of $^\ast \Gamma$ such that $<\Gamma, ^\ast \Gamma> = \text{Aut}(\Gamma)$.}:
\begin{align*}
HGC_{n,j} &= \text{Hom}(^\ast HGC_{n,j}, \mathbb{Q}), \\
PGC^{\prime}_{n,j} &= \text{Hom}(^\ast PGC^{\prime}_{n,j}, \mathbb{Q}).
\end{align*}
Let 
\begin{align*}
\mathcal{B} &= H_{top} (^\ast HGC_{n,j}), \\ 
\mathcal{A}_{PGC^{\prime}} &=  H_{top} (^\ast PGC^{\prime}_{n,j}). 
\end{align*}

\begin{theorem}
\label{weakerrelationwithhairygraphcomplexc}
Assume that $n-j$ is even and $j \geq 2$. Then, there is a left inverse of the map $\mathcal{B} \rightarrow \mathcal{A}_{PGC^{\prime}}$. 
\end{theorem}

\begin{rem}
Such a left inverse gives an inverse of $\mathcal{B} \rightarrow \mathcal{A}_{PGC^{\prime}}$, since $\mathcal{B}$ and $\mathcal{A}_{PGC^{\prime}}$ are isomorphic and they are finite dimensional if we fix the order and the first Betti number of graphs. 
\end{rem}

\begin{proof}[Proof of Theorem \ref{weakerrelationwithhairygraphcomplexc}]
We introduce another space $\mathcal{A}$ of \textit{good plain graphs} with a map $\mathcal{A}_{PGC^{\prime}} \rightarrow \mathcal{A}$. In \cite{Yos 2}, we gave a left inverse $\sigma_{\ast}$ of the composition $\chi_{\ast}: \mathcal{B} \rightarrow \mathcal{A}_{PGC^{\prime}} \rightarrow \mathcal{A}$.  
This construction of the left inverse is analogous to Bar-Natan's \cite{Bar} construction of the inverse of PBW (Poincar\'{e}--Birkhoff--Witt) isomorphisms from the space $\mathcal{B}$ of open Jacobi diagrams to the space $\mathcal{A}(S^1)$ of Jacobi diagrams on the unit circle. 
\end{proof}

This proof allows us to prove the following.

\weakerrelationwithhairygraphcomplexb*
\hfill$\square$

Although we focused on the case where $n-j$ is even, the results of this section can very likely be extended to the case where $n-j$ is odd. 

\section{Computation of the top homologies for $g=2$ and  $g=3$}
\label{computation of top homology}

Nakatsuru, Conant, Costello, Turchin and Weed showed the following. 

\begin{theorem}\cite{Nak}\cite[Theorem 5.7] {CCTW}
\label{2loopcomputation}
Let $\sigma_1, \sigma_2, \sigma_3$ be the elementary symmetric polynomials of three variables
\begin{align*}
\sigma_1 &= x_1 + x_2 + x_3, \\
\sigma_2 &= x_1x_2 + x_2x_3 + x_1x_3, \\
\sigma_3 & = x_1x_2x_3, 
\end{align*}
and let $\Delta$ be the polynomial
\[
\Delta = (x_1-x_2)(x_1-x_3)(x_2-x_3).
\]
Then, when $n$ and $j$ are odd, 
\[
\mathcal{B}_{n,j}(g=2) \cong  \mathcal{B}^{even}_{n,j}(g=2) \cong\mathbb{Q}[\sigma_2, (\sigma_3)^2].
\]
When $n$ and $j$ are even, 
\[
\mathcal{B}_{n,j}(g=2) \cong \mathcal{B}^{even}_{n,j}(g=2) \cong  \Delta \sigma_3\mathbb{Q}[\sigma_2, (\sigma_3)^2], 
\]
where $\mathcal{B}^{even}_{n,j}(g=2)$ is the subspace of uni-trivalent hairy graphs with an even number of hairs. 
\end{theorem}

\begin{rem}
In \cite{CCTW}, $H_{\ast}(^{\ast} HGC_{n,j}(g=2))$ is computed for any degree and any parities of $n$ and $j$. 
\end{rem}

\begin{proof}[Proof of Theorem \ref{2loopcomputation}]

First, we can identify $\mathcal{B}_{n,j}(g=2)$ with the space of labeled graphs $\Theta(p,q,r)$ in Figure \ref{2looplabeledgraph}, modulo relations we later introduce.

\begin{figure}[htpb]
\begin{center}
\includegraphics[width = 10cm]{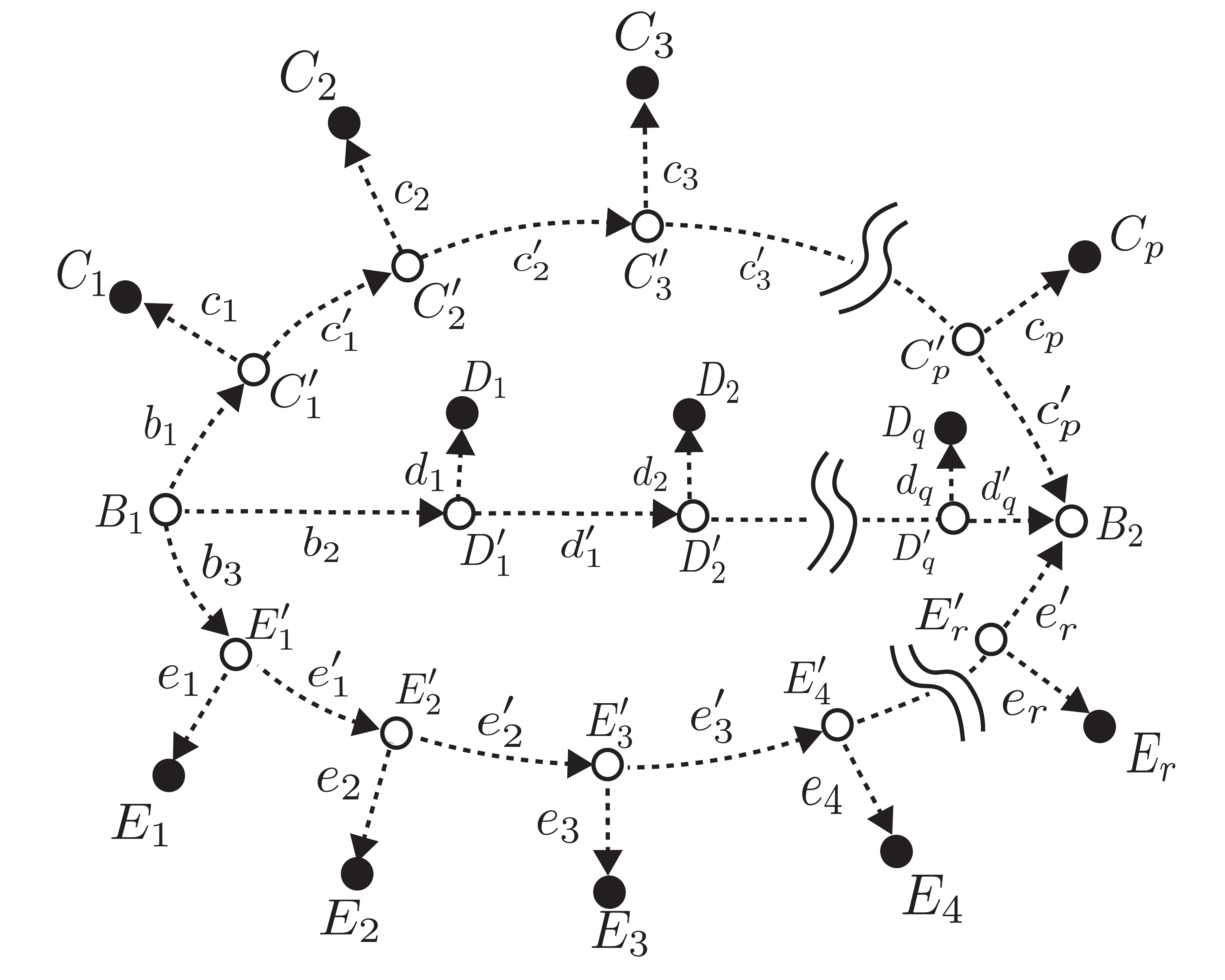}
\caption{The $2$-loop graph $\Theta(p,q,r)$}
\label{2looplabeledgraph}
\end{center}
\end{figure}

We write 
\begin{align*}
x_1 &=  \Theta(1,0,0),\\
x_2 &=  \Theta(0,1,0), \\
x_3 & = \Theta (0,0,1), \\
\end{align*}
and more generally, $x_1^p x_2^q x_3^r = \Theta(p,q,r)$. The relation corresponding to IHX relations of $\mathcal{B}_{n,j}(g=2)$ is 
\[
x_1 + x_2 + x_3 = 0.
\]
The relations corresponding to orientation relations are 
\begin{align*}
x_1^p x_2^q x_3^r &= (sign(\sigma))^{n+1} x_{\sigma(1)}^p x_{\sigma(2)}^q x_{\sigma(3)}^r \quad (\sigma \in \mathfrak{S}_3) ,\\
x_1^p x_2^q x_3^r &= (-x_1)^p (-x_2)^q (-x_3)^r,
\end{align*}
which correspond to the action of $\text{Aut}(\Theta) \cong \mathfrak{S}_3 \times  \mathbb{Z}/2\mathbb{Z}$; permutations of edges and the symmetry with respect to the vertical line. So when $n$ and $j$ are odd, we have
\begin{align*}
\mathcal{B}_{n,j}(g=2) & \cong \frac{\mathbb{Q}[x_1, x_2, x_3]}{( \mathfrak{S}_3 \times \mathbb{Z}_2),(x_1+x_2+x_3=0)} \\ 
& \cong\mathbb{Q}[x_1, x_2, x_3]^{\mathfrak{S}_3}_{even} / (x_1+x_2+x_3=0) \\
& = Sym (x_1, x_2, x_3)_{even}/ (x_1+x_2+x_3=0) \\
&  \cong \mathbb{Q}[\sigma_2, (\sigma_3)^2], 
\end{align*}
where $Sym$ means symmetric polynomials and $even$ means even degree.  On the other hand, when $n$ and $j$ are even, we have
\begin{align*}
\mathcal{B}_{n,j}(g=2) & \cong \frac{\mathbb{Q}[x_1, x_2, x_3]}{( \mathfrak{S}_3 \times \mathbb{Z}_2),(x_1+x_2+x_3 = 0)} \\ 
& \cong\mathbb{Q}[x_1, x_2, x_3]^{\mathfrak{S}_3}_{even} / (x_1+x_2+x_3 = 0) \\
& = Alt (x_1, x_2, x_3)_{even}/ (x_1+x_2+x_3 =0) \\
&  \cong \Delta \sigma_3  \mathbb{Q}[\sigma_2, (\sigma_3)^2].
\end{align*}
where $Alt$ means alternating polynomials. 
\end{proof}

Then, we proceed to study $\mathcal{B}_{n,j}(g=3)$. Moskovich and Ohtsuki showed the following.

\begin{theorem}\cite{MO} 
\label{3loopcomputation}
Let $\tau_1, \tau_2, \tau_3, \tau_4$ be the elementary symmetric polynomials of four variables: 
\begin{align*}
\tau_1 &= y_1 + y_2 + y_3 + y_4, \\
\tau_2 &= y_1y_2 + y_1y_3 + y_1y_4 + y_2y_3 + y_2y_4 + y_3y_4, \\
\tau_3 & = y_1y_2y_3 + y_1y_2y_4 +y_1y_3y_4 +  y_2y_3y_4, \\
\tau_4 & = y_1y_2y_3y_4.
\end{align*}
Then, when $n$ and $j$ are odd, we have
\[
\mathcal{B}_{n,j}(g=3) = \mathcal{B}^{even}_{n,j}(g=3) \cong\mathbb{Q}[\tau_2, (\tau_3)^2, \tau_4], 
\]
\end{theorem}

Furthermore, we show the following, which is a new result. 

\begin{theorem}
\label{3loopcomputation2}
Let $\Delta$ be the polynomial
\[
\Delta = (y_1-y_2)(y_1-y_3)(y_1-y_4)(y_2-y_3)(y_2-y_4)(y_3-y_4).
\]
When $n$ and $j$ are even, we have
\[
\mathcal{B}^{odd}_{n,j}(g=3) \cong \Delta \tau_3\mathbb{Q}[\tau_2, (\tau_3)^2, \tau_4].
\]
\end{theorem}

\begin{proof}[Proof of Theorem \ref{3loopcomputation} and \ref{3loopcomputation2}]

First, we can identify $\mathcal{B}_{n,j}(g=3)$ with the space of labeled graphs $\ctext{Y}(p_1, p_2, p_3, p_4, p_5, p_6)$ in Figure \ref{3looplabeledgraph1} and $\textcircled{=}(p_1, p_2, p_3, p_4, p_5, p_6)$  in Figure \ref{3looplabeledgraph2}, modulo relations we later introduce.

\begin{figure}[htpb]
\begin{center}
\includegraphics[width = 12cm]{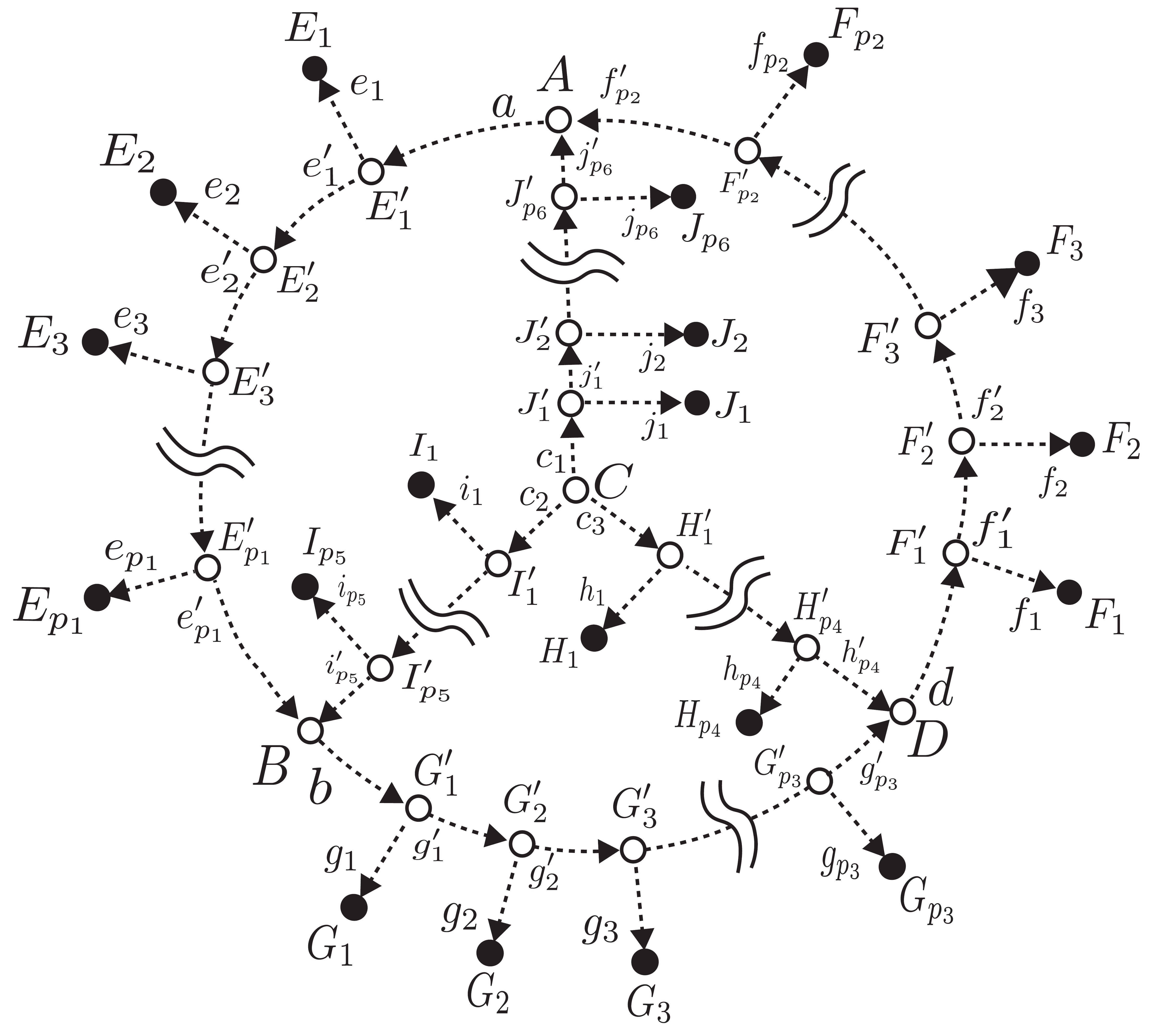}
\caption{The $3$-loop graph $\ctext{Y}(p_1, p_2, p_3, p_4, p_5, p_6)$}
\label{3looplabeledgraph1}
\end{center}
\end{figure}

\begin{figure}[htpb]
\begin{center}
\includegraphics[width = 13cm]{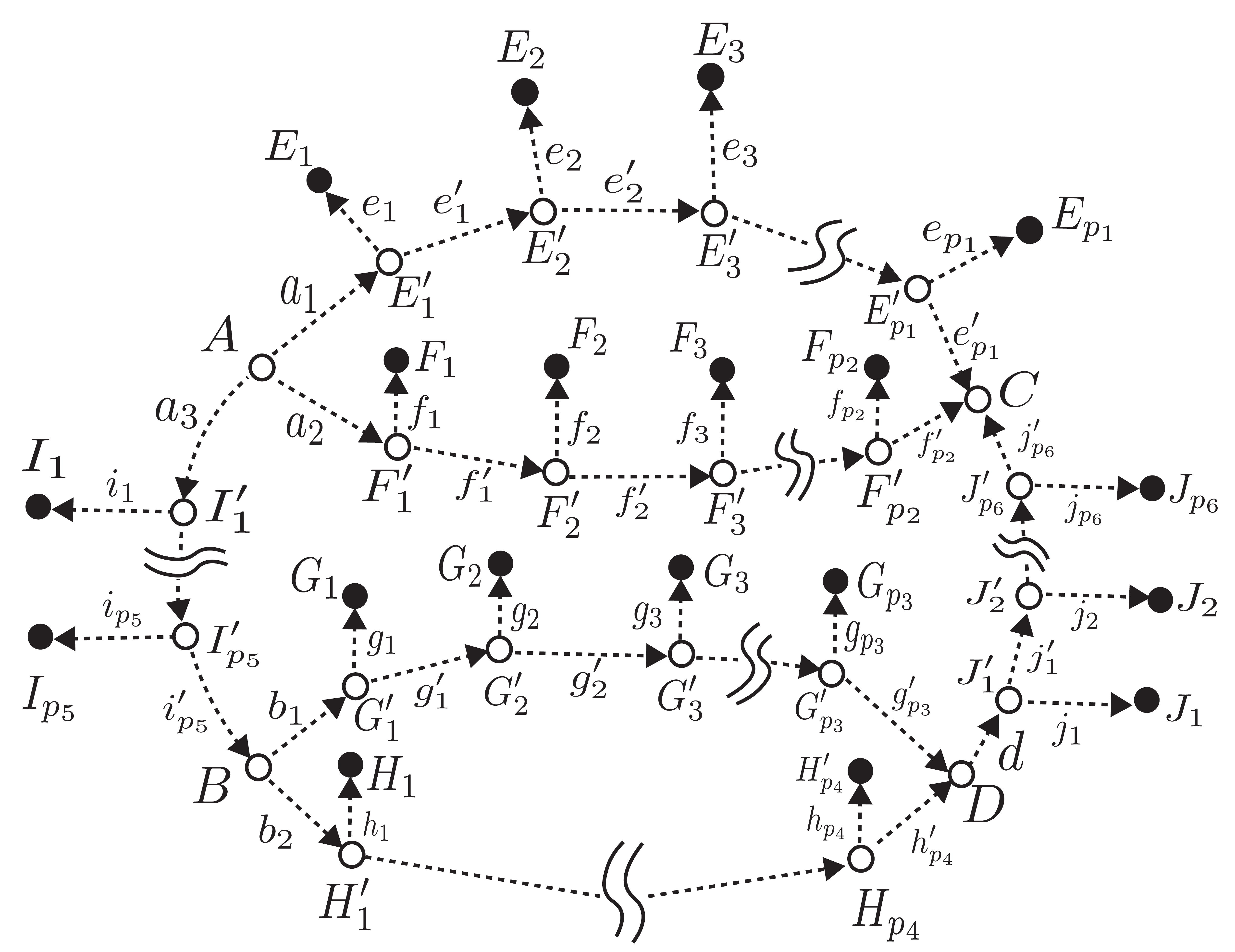}
\caption{The $3$-loop graph $\textcircled{=}(p_1, p_2, p_3, p_4, p_5, p_6)$}
\label{3looplabeledgraph2}
\end{center}
\end{figure}

We write 
\begin{align*}
x_i &=  \ctext{Y}(0, \dots,0,\overset{(i)}{1},0, \dots, 0),\\
z_i &=  \textcircled{=}(0, \dots,0,\overset{(i)}{1},0, \dots, 0).
\end{align*}
The relations corresponding to IHX relations of $\mathcal{B}_{n,j}$ are
 \[
\begin{cases}
x_1 - x_2 -x_6 = 0 \\
-x_1 + x_3 - x_5 = 0\\
x_4 + x_5 + x_6 = 0\\
-x_2 + x_3 +x_4 = 0,\\
\end{cases}
\begin{cases}
z_1 +z_2 + z_5 = 0 \\
z_3 + z_4 - z_5 = 0\\
z_1 + z_2 + z_6 = 0\\
z_3 + z_4 -z_6= 0\\
\end{cases}
\]
and the relation in Figure \ref{3looplabeledgraphIHX}, which is 
\begin{itemize}
\item [($\ast$)] $z_1^{p} z_2^{q} z_3^{r}  z_4^{s} =  x_1^{p} (-x_2)^{s} x_4^{r} x_5^{q} +(-1)^{n+1} x_1^{q} (-x_2)^{s} x_4^{r} x_5^{p}$.
\end{itemize}

\begin{figure}[htpb]
\begin{center}
\includegraphics[width = 18cm]{3looplabeledgraphIHX.pdf}
\caption{The IHX relation relating $\ctext{Y}(p_1, p_2, p_3, p_4, p_5, p_6)$ and $\textcircled{=}(p_1, p_2, p_3, p_4, p_5, p_6)$}
\label{3looplabeledgraphIHX}
\end{center}
\end{figure}

Define new variables $y_1, y_2, y_3, y_4$ by 
\[
\begin{cases}
y_1 = x_1 - x_5 + x_6  \\
y_2 = x_2 + x_4 - x_6 \\
y_3 = x_3 - x_4 + x_5 \\
y_4 = -x_1 - x_2 - x_3. \\
\end{cases}
\]
They satisfy $y_1+y_2+y_3+y_4= 0$ and $z_1+z_2+z_3+z_4= 0$. By the above IHX relations, $\mathcal{B}_{n,j}(g=3)$ is generated by elements of the form $z_1^{p_1} z_2^{p_2} z_3^{p_3}  z_4^{p_4}$ or $y_1^{q_1} y_2^{q_2} y_3^{q_3} y_4^{q_4}$.
Then, the relations corresponding to orientation relations are
\[
y_{\sigma(1)}^p y_{\sigma(2)}^q y_{\sigma(3)}^r y_{\sigma(4)}^s = {(sign(\sigma))}^{p+q+r+s} y_1^p y_2^q y_3^r y_4^s \quad (\sigma \in \mathfrak{S}_4),
\]
\[
\begin{cases}
z_1^p z_2^q z_3^r z_4^s  &= (-1)^{n+1}(-z_1)^p (-z_2)^q (-z_3)^r (-z_4)^s,  \\
z_1^p z_2^q z_3^r z_4^s &= (-1)^{n+1}  z_2^p z_1^q z_3^r z_4^s,  \\
z_1^p z_2^q z_3^r z_4^s &= (-1)^{n+1} z_1^p z_2^q z_4^r z_3^s, \\
z_1^p z_2^q z_3^r z_4^s &= z_4^p z_3^q z_2^r z_1^s, \\

\end{cases}
\]
which correspond to the action of $\text{Aut}(\ctext{Y}) \cong\mathfrak{S}_4$, permutations of faces of tetrahedron, and the action of  $\text{Aut}(\textcircled{=}) \cong D_8 \times  \mathbb{Z}/2\mathbb{Z}$. 
 Let $\mathcal{D}_{n,j}(\textcircled{=})$ be the space of $z_1^{p_1} z_2^{p_2} z_3^{p_3} z_4^{p_4}$ and let  $\mathcal{B}_{n,j}(\textcircled{=})$ be the quotient of $\mathcal{D}_{n,j}(\textcircled{=})$ by the action of  $\text{Aut}(\textcircled{=})$ and the relation $z_1+z_2+z_3+z_4=0$. Let $\mathcal{D}_{n,j}(\ctext{Y})$ be the space of $y_1^{q_1} y_2^{q_2} y_3^{q_3} y_4^{q_4}$ and let $\mathcal{B}_{n,j}(\ctext{Y})$ be the quotient of $\mathcal{D}_{n,j}(\ctext{Y})$ by the action of $\text{Aut}(\ctext{Y})$ and the relation $y_1+y_2+y_3+y_4 = 0$. 
 
 \begin{lemma}
 \label{induced map}
 Consider the map
 \[
 \psi: \mathcal{D}_{n,j}(\textcircled{=}) \rightarrow \mathcal{B}_{n,j}(\ctext{Y})
 \]
 defined from the relation ($\ast$). Then, this map $\psi$ is invariant by the action of the subgroup $D_8$ of $\text{Aut}(\textcircled{=})$. 
 Moreover, when $n$ and $j$ are odd, 
 \[
 \overline{\psi}:  \mathcal{B}^{even} _{n,j}(\textcircled{=}) \rightarrow \mathcal{B}^{even}_{n,j}(\ctext{Y})
 \]
 is induced. 
 On the other hand, when $n$ and $j$ are even
 \[
 \overline{\psi}:  \mathcal{B}^{odd} _{n,j}(\textcircled{=}) \rightarrow \mathcal{B}^{odd}_{n,j}(\ctext{Y})
 \]
 is induced. 
 \end{lemma}
 \begin{proof}
 The second and third statements follow because the action of $\mathbb{Z}/2\mathbb{Z}$ on  $\mathcal{D}^{even} _{n,j}(\textcircled{=})$ (and  $\mathcal{D}^{odd} _{n,j}(\textcircled{=})$)  is trivial when $n, j$: odd (resp. $n$, $j$: even). 
 \end{proof}
 
For both the case $n$, $j$: odd and the case $n$, $j$: even, we have
\begin{align*}
\mathcal{B}^{even}_{n,j}(\ctext{Y}) &\cong \frac{\mathbb{Q}[y_1, y_2, y_3,y_4]_{even}}{(y_1+y_2+y_3+y_4=0), \mathfrak{S}_4} \\
& \cong\mathbb{Q}[y_1, y_2, y_3, y_4]^{\mathfrak{S}_4}_{even} / (y_1+y_2+y_3+y_4=0) \\
& = Sym (y_1, y_2, y_3, y_4)_{even}/ (y_1+y_2+y_3 + y_4=0) \\
& \cong \mathbb{Q}[\sigma_2, ( \sigma_3)^2, \sigma_4], \\
\mathcal{B}^{odd}_{n,j}(\ctext{Y}) &\cong \frac{\mathbb{Q}[y_1, y_2, y_3,y_4]_{odd}}{(y_1+y_2+y_3+y_4=0), \mathfrak{S}_4} \\
& \cong\mathbb{Q}[y_1, y_2, y_3, y_4]^{\mathfrak{S}_4}_{odd} / (y_1+y_2+y_3+y_4=0) \\
& = Alt (y_1, y_2, y_3, y_4)_{odd}/ (y_1+y_2+y_3 + y_4=0) \\
& \cong \Delta \sigma_3 \mathbb{Q}[\sigma_2, ( \sigma_3)^2, \sigma_4]. \\
\end{align*}

Suppose $n$ and $j$ are odd. Then, by Lemma \ref{induced map} we have 
\[
\mathcal{B}^{even}_{n,j} (g=3)  \cong \mathcal{B}^{even}_{n,j}(\ctext{Y}) \cong \mathbb{Q}[\sigma_2, ( \sigma_3)^2, \sigma_4]. 
\]
On the other hand, 
\begin{align*}
\mathcal{B}^{odd}_{n,j}(\textcircled{=}) &\cong \frac{\mathbb{Q}[z_1, z_2, z_3,z_4]_{odd}}{(z_1+z_2+z_3+z_4=0), D_8 \times \mathbb{Z}/2\mathbb{Z}} \\
& \cong 0,
 \end{align*}
 due to the action of $\mathbb{Z}/2\mathbb{Z}$. Moskovich and Ohtsuki showed that 
 \[
 \mathcal{B}^{odd}_{n,j} (g=3) \cong 0
 \]
 by studying the relation ($\ast$).

Suppose $n$ and $j$ are even. Then, by Lemma \ref{induced map} we have 
\[
\mathcal{B}^{odd}_{n,j} (g=3) \cong \mathcal{B}^{odd}_{n,j}(\ctext{Y}) \cong \Delta \sigma_3 \mathbb{Q}[\sigma_2, ( \sigma_3)^2, \sigma_4]. 
\]
On the other hand, 
\begin{align*}
  \mathcal{B}^{even}_{n,j}(\textcircled{=}) &\cong \frac{\mathbb{Q}[z_1, z_2, z_3,z_4]_{even}}{(z_1+z_2+z_3+z_4=0), D_8 \times \mathbb{Z}/2\mathbb{Z}} \\
 & \cong 0,
 \end{align*}
 due to the action of $\mathbb{Z}/2\mathbb{Z}$.

\end{proof}

\section{The hidden face dg algebra and the decorated graph complex}
\label{The hidden face dg algebra and the decorated graph complex}

\subsection{The hidden face dg algebra $A_{n,j}$}
\label{The hidden face dg algebra $A_{n,j}$}

Let $\text{Inj}(\mathbb{R}^j, \mathbb{R}^n)$ be the space of linear injective maps from $\mathbb{R}^j$ to  $\mathbb{R}^n$.
Here, we define the hidden face dg algebra $A_{n,j} (j \geq 2)$. The algebra $A_{n,1}$ (with the trivial differential) is defined similarly. See Remark \ref{hidden face dg algebra j=1}.  Later, we give a morphism of dg algebras  $I: A_{n,j} \rightarrow \Omega_{dR}^{\ast}\text{Inj}(\mathbb{R}^j, \mathbb{R}^n)$. 

\begin{definition}
\textit{Generalized plain graphs} have three types of vertices
(white \begin{tikzpicture}[baseline = -3pt]  \draw (0, 0) circle (0.05); \end{tikzpicture},
external black \begin{tikzpicture} [baseline = -3pt]  \draw (0, 0) circle (0.05) [fill={rgb, 255:red, 0; green, 0; blue, 0}, fill opacity =1.0]; \end{tikzpicture},
internal black \begin{tikzpicture}  [baseline = -2pt]   \draw (0, 0) rectangle (0.1, 0.1) [fill={rgb, 255:red, 0; green, 0; blue, 0}, fill opacity =1.0]; \end{tikzpicture})
and two types of edges (dashed, solid). White vertices have at least one dashed edge and no solid edge, while internal black vertices have at least one solid edge and no dashed edge. External black vertices have at least one (solid or dashed) edge. Double edges and loop edges are allowed.

A generalized plain graph is \textit{admissible} if every external black vertex has at least one dashed edge. 
\end{definition}

\begin{definition}
\label{defofA}
As a vector space, the algebra $A_{n,j}$ is generated by finite sequences of connected, labeled, admissible, generalized plain graphs. 
The sequence $\emptyset$ of length $0$ is considered as an element of $A_{n,j}$.
The relations of $A_{n,j}$ are the following. 
\begin{itemize}
\item [(0)]  Relations that graphs with negative normalized degrees (see Definition \ref{normalized degree}) and sequences with normalized degrees more than $nj -1 = dim(\text{Inj}(\mathbb{R}^j, \mathbb{R}^n)) -1$ vanish. 
\item [(1)]  Orientation relations of graphs. 
\item [(2)]  Relations that graphs with double edges or loop edges vanish.
\item [(3)]  Relations that graphs without external black vertices vanish if they have at least three vertices. 
\item [(4)]  Rescaling and symmetry relations which are defined in Definition \ref{rescaling relation} and Definition \ref{symmetry relation}.
\item [(5)]  Relations that the following three graphs with two vertices, $\dashededgea{-Stealth}{1}{2}$, $\dashededgec{-Stealth}{1}{2}$, $\solidedged{-Stealth}{1}{2}$ and $\solidedgee{-Stealth}{1}{2}$ are identified with $\emptyset$. 
(On the other hand
 \dashededgeb{-Stealth}{1}{2} and 
\begin{tikzpicture}[x=0.75pt,y=0.75pt,yscale=1,xscale=1, baseline=-3pt, line width =0.8pt]
\draw  [fill={rgb, 255:red, 0; green, 0; blue, 0 }  ,fill opacity=1 ] (0,0) circle (3) node [anchor = north] {$1$};
\draw  [fill={rgb, 255:red, 0; green, 0; blue, 0 }  ,fill opacity=1 ] (50,0) circle (3) node [anchor = north] {$2$} ;
\draw  [dash pattern = on 3pt off 3 pt, -Stealth] (0,0) .. controls (20,10) and (30,10)  .. (50,0);
\draw  [-Stealth] (0,0) .. controls (20,-10) and (30,-10)  .. (50,0);
\end{tikzpicture}
may survive.  )

\item [(6)]   Permutation relations:
\[
\ytableausetup{centertableaux, mathmode, boxsize = 2em}
\begin{ytableau}
\Gamma & \Gamma^{\prime}  \\
\end{ytableau}
= (-1)^{[\Gamma][\Gamma^{\prime}]}\ 
\begin{ytableau}
\Gamma^{\prime} & \Gamma  \\
\end{ytableau}.
\]
The normalized degree $[\Gamma]$ is defined in Definition \ref{normalized degree}.
\end{itemize}

\end{definition}

\begin{rem}
By conditions (3), (4), and (5) of Definition \ref{defofA}, it follows that $A_{n,j}$ is generated by (finite sequences of) connected, admissible plain graphs in the original sense. By condition (6), the algebra $A_{n,j}$ is a commutative graded algebra generated by connected, admissible plain graphs.
\end{rem}

\begin{definition}[rescaling relations]
\label{rescaling relation}
We define rescaling relations as follows. 
Let a graph $\Gamma$ have a cut vertex $v$: a vertex such that if the vertex is removed, the graph is decomposed into at least two (non-empty) components:
\[
\Gamma = \Gamma_1 \cup_v \Gamma_2. 
\]
 If the cut vertex is white, additionally assume that  one component consists of only white vertices. Then $\Gamma \sim 0$. In particular, if a graph with at least three vertices has one of the following vertices, the graph vanishes.
\begin{itemize}
\item a white vertex with exactly one dashed edge
\item a black vertex with exactly one solid edge and no dashed edge
\item an (external) black vertex with exactly one dashed chord
\end{itemize}
\end{definition}

\begin{definition}[symmetry relations]
\label{symmetry relation}

We define symmetry relations as follows.
Let $v_1, v_2$ be (possibly the same) two vertices of $\Gamma$. Let $\Gamma_S$ be a subgraph of $\Gamma$ with $v_1, v_2$ such that $\Gamma_S \setminus \{v_1, v_2\}$ is a connected component of $\Gamma\setminus\{v_1, v_2\}$:
\[
\Gamma = \Gamma_S \cup_{\{v_1, v_2\}} \Gamma^{\prime} \quad (\Gamma^{\prime} = (\Gamma \setminus \Gamma_S) \cup \{v_1, v_2\})
\]
If at least one of $v_1, v_2$ is white, assume that $\Gamma_S \setminus \{v_1, v_2\}$ consists of only white vertices. 
Consider the new graph $\overline{\Gamma}$ obtained by reversing $\Gamma_S$, namely by attaching $v_1$ (resp. $v_2$) of $\Gamma_S$ to $v_2$ (resp. $v_1$) of $\Gamma^{\prime}$. Then 
\begin{align*}
\Gamma &\sim (-1)^{-n|W(S \setminus \{v_1, v_2\} )| - j|B(S\setminus \{v_1, v_2\})| +j|E_{\eta}(S)| + n |E_{\theta}(S)|}\ \overline{\Gamma}
\end{align*}

\begin{figure}[htpb]
   \centering
    \includegraphics [width =7cm] {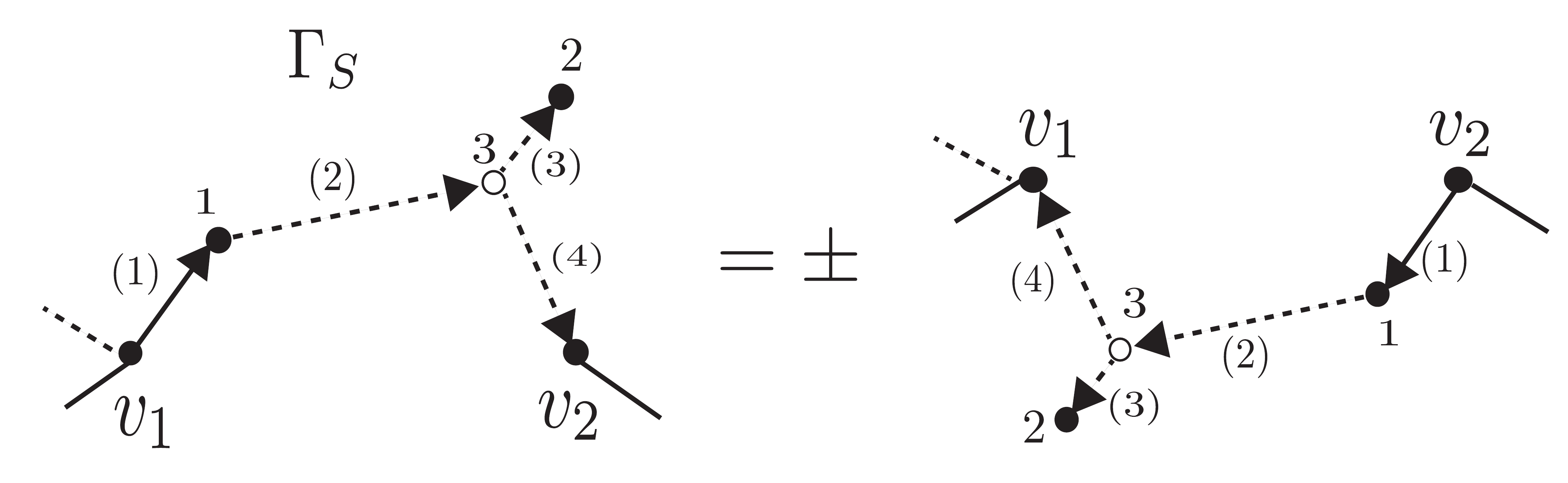}
    \caption{An example of symmetry relations}
    \label{figofsymmetryrelation}
\end{figure}

In particular,  if a graph has a white vertex with exactly two dashed edges or it has a black vertex with two solid edges and no dashed edge, the graph vanishes. This relation follows from the hidden face involution introduced by \cite{Kon 1}. 
\end{definition}

\begin{rem}
\label{allowedvertices}
We summarize what types of vertices can appear in non-trivial graphs (with at least three vertices) in $A_{n,j}$.
\begin{itemize}
\item Any white vertex must have at least three dashed edges (and no solid edges).
\item Any internal black vertex must have at least three solid edges (and no dashed edges).
\item Any external black vertex must have at least one dashed edge and an arbitrary number of solid edges.
\end{itemize}
Then we can observe that the defect of a vertex in a non-trivial graph is negative only if the vertex is internal black, in which case the defect of that vertex is exactly $-1$.
\end{rem}

\begin{definition}
\label{normalized degree}
For a connected graph $\Gamma$ of $A_{n,j}$, define the normalized degree $\overline{deg}(\Gamma) = [\Gamma]$ of $\Gamma$  by 
\[
[\Gamma] = (n-1) |E_{\theta}(\Gamma)| 
+ (j-1) |E_{\eta}(\Gamma)| 
-n |W(\Gamma)| 
-j |B(\Gamma)|
+ (j+1)
\]
if $\Gamma$ has at least one black vertex, and by 
\[
[\Gamma] = (n-1) |E_{\theta}(\Gamma)| 
+ (j-1) |E_{\eta}(\Gamma)| 
-n |W(\Gamma)| 
-j |B(\Gamma)|
+ (n+1)
\]
if $\Gamma$ has no black vertex.
The degrees of general elements of $A_{n,j}$ are defined by the sum of degrees of graphs. 
\end{definition}

\begin{definition}

\begin{itemize}
\item The product on $A_{n,j}$ is defined by
\[
\ytableausetup{centertableaux, mathmode, boxsize = 2em}
\begin{ytableau}
a \\
\end{ytableau}
\cdot
\begin{ytableau}
a^{\prime} \\
\end{ytableau}
= 
\begin{ytableau}
a & a^{\prime}  \\
\end{ytableau}\ ,
\]
which is a graded commutative product with the unit $\emptyset$.
\item The augmentation $\eta: A_{n,j} \rightarrow \mathbb{R}$ is defined by
\begin{equation*}
\eta (a) = 
\begin{cases}
x & (\text{if}\ a = x \cdot \emptyset,\ x \in \mathbb{R}) \\
0 &(\text{otherwise}).
\end{cases}
\end{equation*}.
\end{itemize}
\end{definition}

\begin{definition}
\label{differentialonA}
The differential of a connected, labeled, admissible generalized plain graph $\Gamma$ of $A_{n,j}$ is defined by 
\begin{equation*}
d \Gamma = 
\sum_{\substack{S \subsetneq V(\Gamma) \\  |S| \geq 2}} \sigma(S) \Gamma/S \cdot \Gamma_S.
\end{equation*}
Here, the graph $\Gamma_S$ is the subgraph on $S$, where black vertices without dashed edges are regarded as internal. The graph $\Gamma/S$ is the graph obtained by collapsing all vertices and edges of the subgraph $\Gamma_S$ to one vertex.
The new collapsed vertex $[S]$ of $\Gamma/S$ is considered as
\[
\begin{cases}
\text{external black} & \text{if $S$ contains an external black vertex with a dashed edge outside $\Gamma_S$},  \\
\text{internal black} & \text{if $S$ contains black vertices but every dashed edge incident to them belongs to  $\Gamma_S$},\\
\text{white} & \text{if $S$ contains no black vertex}.
\end{cases} 
\]
The sign $\sigma(S)$ is defined by 
\begin{align*}
\sigma(S) &= (-1)^{[\Gamma]+1} \tau(S) \Phi(S) (-1)^{\overline{deg}(V(\Gamma/S))(1+[\Gamma_S])}.
\end{align*}

In particular, the differential of graphs with exactly two vertices vanishes. 
Exceptionally, we define the differential of graphs with negative normalized degrees to vanish. 
The differential is extended to $A_{n,j}$ by setting
\[
d(\Gamma \cdot \Gamma^{\prime}) = d(\Gamma) \cdot \Gamma^{\prime} + (-1)^{[\Gamma]} \Gamma \cdot d(\Gamma^{\prime}).
\]
\end{definition}

Then, we have
\begin{prop}
The differential vanishes on the space $N_{n,j}$ generated by relations in Definition \ref{defofA}. 
\end{prop}
\begin{proof}
We only check it for relations (0), (2), (3), and (4). We leave the other cases to the reader.
\begin{itemize}
\item[(0)] In Definition \ref{differentialonA}, we defined the differential of graphs of negative normalized degree to vanish.
\item[(2)] If a graph has double or loop edges, the sequences after the differential also have such a graph. 
\item[(3)] Let a sequence have a connected graph without external black vertices. Then the graph must be either a graph consisting of only dashed edges and white vertices, or a graph consisting of only solid edges and external black vertices. Then the sequences after the differential also have such a graph.
\item[(4)] Let a sequence have a graph $\Gamma$ with a cut vertex $v$:
\[
\Gamma = \Gamma_1 \cup_v \Gamma_2.
\]
The contraction which breaks this cut vertex is the one that decomposes $\Gamma$ into the two components $\Gamma_1$ and $\Gamma_2$. That is, the cases $S = V(\Gamma_1)$ or  $S = V(\Gamma_2)$. 
However, these two contributions cancel each other.  
Next, consider a pair of sequences related by a symmetry relation. Let $\Gamma =  \Gamma_T \cup_{\{v_1, v_2\}} \Gamma^{\prime}$ and $\overline{\Gamma}$ be the graph to which the symmetry relation is applied. 
Say, $\Gamma \sim (-1)^n \overline{\Gamma}$.  Consider the contraction $d_S(\Gamma)$ of $\Gamma$ that breaks $\Gamma_T$. That is, the case $(S \setminus \{v_1, v_2\})  \cap T \neq \emptyset$, $T \not\subset S$. 
Then, by considering symmetry relations of both $\Gamma_S$ and  $\Gamma/S$, we see $d_S\Gamma \simeq (-1)^n d_S \overline{\Gamma}$ holds. 
\end{itemize}
\end{proof}

\begin{rem}
Let $S = (s_1, \dots, s_{|S|})$ be an ordered subset of $\Gamma$, where $s_1$ is a black vertex. Then, the sign $\sigma(S)$ is characterized as follows. Consider the ordered set
\[
-\mu \sqcup o(\Gamma) = -\mu \sqcup E(\Gamma)\sqcup \overline{V(\Gamma)} \sqcup \iota_{s_1} \sqcup \iota.
\]
Here, $\overline{V(\Gamma)}$ is the ordered set obtained by reversing the order of $V(\Gamma)$. Elements of $E(\Gamma)$ and $V(\Gamma)$ are graded. The variable $\mu$ at the head and $\iota$ at the tail are elements of degree one while $\iota_{s_1}$ is an element of degree $j$. The variables $\iota$ and $\iota_{s_1}$ correspond to the rescaling and translation of configuration spaces. Then $\sigma(S)$ is the sign which arises when permuting $- \mu \sqcup  o(\Gamma)$ to
\[
E(\Gamma/S)  \sqcup \overline{V(\Gamma/S)} \sqcup \iota_{s_1} \sqcup \iota \sqcup (E(\Gamma_S) \sqcup \overline{V(\Gamma_{S}) \setminus s_1 } \sqcup \mu ).
\]
\end{rem}

\begin{prop}
$d^2 = 0$.
\end{prop}

\begin{proof}
Let $\Gamma$ be an admissible plain graph. Then
\begin{align*}
d^2 \Gamma &= d\left( \sum_{\substack{S \subsetneq V(\Gamma) \\  |S| \geq 2}} \sigma(S, \Gamma) \Gamma/S \cdot \Gamma_S \right) \\
& = \sum_{\substack{S \subsetneq V(\Gamma) \\  |S| \geq 2}} \sigma(S, \Gamma) \left( d(\Gamma/S) \cdot \Gamma_S + (-1)^{[\Gamma/S]} \Gamma/S \cdot d \Gamma_S \right). \\
\end{align*} 
Here, 
\begin{align*}
d(\Gamma/S) &=  \sum_{\substack{T \subsetneq V(\Gamma/S) \\  |T| \geq 2}} \sigma(T, \Gamma/S) \left((\Gamma/S)/T\right) \cdot (\Gamma/S)_T \\
& = \sum_{\substack{T \subset V(\Gamma)\setminus S \\  |T| \geq 2}} \sigma(T, \Gamma/S) \left(\Gamma/(S \cup T)\right) \cdot \Gamma_T \\
& + \sum_{\substack{S \subsetneq T \subsetneq  V(\Gamma) \\ }} \sigma(T/S, \Gamma/S) \left(\Gamma/T\right) \cdot (\Gamma_T)/S, \\
d(\Gamma_S) &= \sum_{\substack{T \subsetneq S \\  |T| \geq 2}} \sigma(T, \Gamma_S) (\Gamma_S)/T \cdot \Gamma_T.
\end{align*}
So we have
\begin{align*}
d^2 \Gamma & = \sum_{\substack{S, T \subsetneq V(\Gamma) \\  |S|, |T| \geq 2,\ S\cap T = \emptyset}} \sigma(S, \Gamma) \sigma(T, \Gamma/S) (\Gamma/(S\cup T)) \cdot \Gamma_T \cdot \Gamma_S \\
& + \sum_{\substack{S \subsetneq T \subsetneq V(\Gamma) \\  |S| \geq 2}} \sigma(S, \Gamma) \sigma(T/S, \Gamma/S) (\Gamma/T) \cdot (\Gamma_T)/S \cdot \Gamma_S \\
&  + (-1)^{[\Gamma/T]}\sum_{\substack{S \subsetneq T \subsetneq V(\Gamma) \\  |S| \geq 2}} \sigma(T, \Gamma) \sigma(S, \Gamma_T) (\Gamma/T) \cdot (\Gamma_T)/S \cdot \Gamma_S
\end{align*}
These terms are canceled by the following lemma.
\end{proof}

\begin{lemma}
Let $S = (s_1, \dots, s_{|S|})$ and $T = (t_1, \dots, t_{|T|})$  be the ordered subset of $S$ and $T$.
\begin{itemize}
\item If $S \cap T = \emptyset$, we have
\[
\sigma(S, \Gamma) \sigma(T, \Gamma/S) = -(-1)^{[\Gamma_S][\Gamma_T]}\sigma(T, \Gamma)  \sigma (S, \Gamma/T). 
\]
\item If $S \subsetneq T$, we have 
\[
\sigma(S, \Gamma) \sigma(T/S, \Gamma/S) = -(-1)^{[\Gamma/T]}\sigma(T, \Gamma)  \sigma (S, \Gamma_T). 
\]
\end{itemize}
\end{lemma}
\begin{proof}

(1) The sign $\sigma(S, \Gamma)\sigma(T, \Gamma/S)$  (resp. $\sigma(T, \Gamma)\sigma(S, \Gamma/T)$) corresponds to the sign obtained by permuting the ordered set
\[
\nu \sqcup  \mu \sqcup E(\Gamma) \sqcup \overline{V(\Gamma)} \sqcup \iota_{v_1} \sqcup  \iota  \quad (\text{resp. $\mu \sqcup  \nu \sqcup E(\Gamma) \sqcup \overline{V(\Gamma)}  \sqcup \iota_{v_1} \sqcup  \iota $})
\]
to 
\[
E(\Gamma/ (S \cup T)) \sqcup \overline{V(\Gamma/ (S \cup T))} \sqcup \iota_{v_1} \sqcup  \iota  
\sqcup \left(E(\Gamma_T) \sqcup \overline{V(\Gamma_{T}) \setminus t_1 } \sqcup \nu \right)\sqcup \left(E(\Gamma_S) \sqcup \overline{V(\Gamma_{S}) \setminus s_1} \sqcup \mu \right). 
\]
(resp.
$E(\Gamma/ (S \cup T)) \sqcup \overline{V(\Gamma/ (S \cup T))} \sqcup \iota_{v_1} \sqcup  \iota 
\sqcup \left(E(\Gamma_S) \sqcup \overline{V(\Gamma_{S}) \setminus s_1 } \sqcup \mu \right)\sqcup \left(E(\Gamma_T) \sqcup \overline{V(\Gamma_{T}) \setminus t_1} \sqcup \nu \right). 
$)

(2) The sign $\sigma(S, \Gamma)\sigma(T/S, \Gamma/S)$ corresponds to the sign permuting the original orientation set 
\[
\nu \sqcup  \mu \sqcup E(\Gamma) \sqcup \overline{V(\Gamma)} \sqcup \iota_{v_1} \sqcup  \iota  \quad 
\]
to 
\[
E(\Gamma/ T) \sqcup \overline{V(\Gamma/ T)} \sqcup \iota_{v_1} \sqcup  \iota 
\sqcup \left(E(\Gamma_T/S) \sqcup \overline{V(\Gamma_{T}/S) \setminus t_1 } \sqcup \nu \right)\sqcup \left(E(\Gamma_S) \sqcup \overline{V(\Gamma_{S}) \setminus s_1} \sqcup \mu \right). 
\]
On the other hand, the sign $(-1)^{[\Gamma/T]} \sigma(T, \Gamma)  \sigma (S, \Gamma_T)$ corresponds to the sign obtained by permutatng the original orientation set
\[
\mu \sqcup  \nu \sqcup E(\Gamma) \sqcup \overline{V(\Gamma)} \sqcup \iota_{v_1} \sqcup  \iota  \quad 
\]
to 
\[
E(\Gamma/ T) \sqcup \overline{V(\Gamma/ T)} \sqcup \iota_{v_1} \sqcup  \iota  
\sqcup \left(E(\Gamma_T/S) \sqcup \overline{V(\Gamma_{T}/S) \setminus t_1 } \sqcup \nu \right)\sqcup \left(E(\Gamma_S) \sqcup \overline{V(\Gamma_{S}) \setminus s_1} \sqcup \mu \right). 
\]

\end{proof}

\begin{notation}
For a finite sequence (product) $a = \Gamma_1\cdots\Gamma_l$ of labeled plain graphs, define the order $k(a)$ of $a$ by the sum of the orders of the graphs $\Gamma_i$. Define $g(a)$ similarly. 
It is easy to see that the differential on $A_{n,j}$ preserves $k$ and $g$. 
Hence, the dg algebra $A_{n,j}$ has a decomposition as a complex
\[
A_{n,j} = \bigoplus_{k\geq 1} \bigoplus_{g \geq 0} A_{ n,j}(k, g). 
\]
\end{notation}

\begin{rem}
\label{hidden face dg algebra j=1}
Let us mention the case $j = 1$. 
Let $\Gamma$ be a connected plain graph (on an oriented line) with at least three vertices. 
In this case, $\Gamma$ can survive in $A_{n,1}$ only when each white vertex has exactly three dashed edges and each black vertex has exactly one dashed edge. 
In fact, if one of the vertices has more edges, the normalized degree of $\Gamma$ exceeds the dimension of the sphere $S^{n-1}$. As a consequence, the algebra $A_{n,j}$ is generated by finite sequences of connected uni-trivalent graphs. It has the trivial differential and the trivial multiplication. Moreover, if $n\geq 4$, only (finite sequences of) the graph 
\begin{tikzpicture}[x=0.7pt,y=0.7pt,yscale=-0.3,xscale=0.3,baseline=-55pt, line width=1pt]
\draw [dash pattern={on 3pt off 2pt}] [-stealth]  (90,270)..controls (170,220)..(245,265) ;

\draw [-stealth] (50, 270)--(300,270) ;

\draw  [fill={rgb, 255:red, 0; green, 0; blue, 0 }  ,fill opacity=1 ] (90,270) circle (5);
\draw (90, 290) node {$1$};
\draw [fill={rgb, 255:red, 0; green, 0; blue, 0 }  ,fill opacity=1 ] (250,270) circle (5);
\draw (250, 290) node {$2$};
\end{tikzpicture}
and $\emptyset$  survive. If $n = 3$, by symmetry relations with respect to one vertex, we have $A_{n,1}(2k) = 0$.
\end{rem}

\subsection{The acyclic bar complex $B(A_{n,j} ,A_{n,j}, \mathbb{R}) = A_{n,j}  \otimes_{\tau} BA_{n,j}$}
\label{The bar complex}

Let $A$ be a dg algebra with an augmentation $\eta: A \rightarrow \mathbb{R}$. For example, take $A_{n,j}$ of Section \ref{The hidden face dg algebra $A_{n,j}$} as $A$. We write $\overline{A}$ for the augmentation ideal.

\begin{definition}
As a vector space $B(A, A, \mathbb{R}) = A \otimes_{\tau} BA $ is the tensor (co)algebra of $A$. (The symbol $\tau$ means the complex is twisted by the universal morphism $\tau: BA_{n,j} \rightarrow A_{n,j}$.) So $A \otimes_{\tau} BA $ is generated by elements of the form 
\[
a_0 [s a_1 | s a_2 | \dots | s a_k ].
\]
Here, $k$ is an integer and $a_0 \in A$, $a_1, \dots, a_k \in \overline{A}$. The symbol $s$ stands for the degree $(-1)$ shift. The degree of $z = a_0 [s a_1 | s a_2 | \dots | s a_k ]$ is defined by $ \sum_{i = 0}^k deg(a_i) -k$. 
\end{definition}

\begin{definition}[The differentials on bar complexes]
The differential on $A \otimes_{\tau} BA$ is defined by $d_{A \otimes_{\tau} BA} = d_A + d_{BA} + d_{\tau}$, where
\begin{align*}
d_A (a_0 [s a_1 | s a_2 | \dots | s a_k]) = &\ d a_0 [s a_1 | s a_2 | \dots | s a_k]\\
 + & \sum_{i=1}^k (-1)^{|a_0| + (\nu_{i-1} +1)} a_0 [s a_1 | s a_2 | \dots |s (d a_i)| \dots | s a_k]. \\
d_{BA} (a_0 [s a_1 | s a_2 | \dots | s a_k]) = & \sum_{i=1}^{k-1} (-1)^{|a_0| + (\nu_{i-1} +1)} a_0 [s a_1 | s a_2 | \dots |s (a_{i-1} a_i)| \dots | s a_k] \\
d_{\tau} (a_0 [s a_1 | s a_2 | \dots | s a_k]) = &\ (-1)^{|a_0| + 1} a_0 a_1 [s a_2 | \dots | s a_k].
\end{align*}
Here, the sign $\nu_i$ is defined by 
\[
\nu_i = \sum_{l=1}^i |sa_l| = \sum_{l=1}^i |a_l| - i.
\]
\end{definition}

\begin{definition}[The products of bar complexes]

We define the product on $A \otimes_{\tau} BA$ as follows.

\begin{align*}
&a_0[s a_1| \cdots | s a_l] \cdot b_0 [s b_1 | \cdots | s b_m] \\
 = & (-1) ^{\nu_l |b_0|} \sum_{\sigma \in Sh(l, m)}  \varepsilon (\sigma, |sa_1|, \cdots, |sa_l|, |sb_1|, \cdots, |sb_m|) \ a_0 b_0 [\sigma(s a_1, \cdots, s a_l, sb_1, \cdots, sb_m)] \\
\end{align*}
Here $Sh(l, m)$ is the set of $(l,m)$-shuffles. The sign $\varepsilon(\sigma, n_1, \cdots, n_k)$ $(\sigma \in \Sigma_k)$ is the sign which arises when we permute the ordered set ($x_1, \cdots, x_k)$ $(|x_i| = n_i)$  of graded elements by $\sigma$. 
\end{definition}

\begin{definition}
Let $A_{n,j}$ be as in Definition \ref{The hidden face dg algebra $A_{n,j}$}. We write $Z_{n,j}$ for the dg algebra $B(A_{n,j}, A_{n,j}, \mathbb{R}) = A_{n,j} \otimes_{\tau} BA_{n,j}$.
\end{definition}

\begin{notation}
Using the decomposition of $A_{n,j}$, we have the decomposition of $Z_{n,j} = A_{n,j} \otimes_{\tau} BA_{n,j}$
\[
Z_{n,j} = \bigoplus_{k\geq 1} \bigoplus_{g \geq 0} Z_{n,j}(k, g). 
\]
$Z_{n,j}(k, g)$ is generated by elements of the form $a_0[sa_1|\cdots|s a_l]$ such that each $a_i$ is a finite sequence of labeled plain graphs and $\sum k(a_i) = k$, $\sum g(a_i) = g$. 
\end{notation}


\subsection{Degree-zero elements of $A_{n,j}$}
\label{Degree-zero elements}
In subsection \ref{subsection iterated integrals}, we define the iterated integral
\[
A_{n,j} \otimes_{\tau} BA _{n,j} \rightarrow \Omega_{dR}^{\ast} P(\text{Inj}(\mathbb{R}^j, \mathbb{R}^n), \bullet, \iota).
\]
This map is well-defined when the condition  $A^{0}_{n,j} = \mathbb{R}$ is satisfied. So we study what elements of degree $0$ the algebra $A_{n,j}$ has. The existence of internal black vertices (when $j \geq 2$) makes the situation a little complicated. 

Recall that if $\Gamma$ is  a $g$-loop plain graph of order $k$ and of defect $l$, then the normalized degree $[\Gamma]$ of a plain graph is equal to 
\[
[\Gamma] = |\Gamma| +(j+1) = k(n-j-2) + (g-1) (j-1) + l + (j+1).
\]

\begin{prop}
\label{numberofinternal}
Let $\Gamma$ be a plain graph of $g\geq 1$. If the number of internal black vertices is greater than $2(g-1)$, then $\Gamma$ has at least one cut (internal) black vertex. Thus $\Gamma = 0 \in A_{n,j}$. 
\end{prop}
We prove a stronger statement for more general graphs. 

\begin{lemma}
\label{lemmanumberofinternal}
For any connected graph $\Gamma$ of any first Betti number $g\geq 1$, and for any set $\{v_1, \dots v_L\}$ of vertices with valence at least  $3$, the number of cut vertices of $\{v_1, \dots v_L\}$ is greater than or equal to  $L - 2(g-1)$. 
\end{lemma}

\begin{proof}
For a graph $\Gamma$ of first Betti number $g\geq 1$, we write $\Gamma^{\prime}$ for the graph obtained from $\Gamma$ by removing all vertices that are not in any cycle of $\Gamma$ and all edges incident to those vertices.
If the removed vertices are of valence at least $3$ (indeed, of valence at least $2$) in $\Gamma$, then they are cut vertices of $\Gamma$.
If a vertex of  $\Gamma^{\prime}$ is of valence at least $3$ in $\Gamma$ but of valence $2$ in $\Gamma^{\prime}$, then it is also a cut vertex of $\Gamma$.
On the other hand, if a vertex of $\Gamma^{\prime}$ is of valence at least $3$ in $\Gamma$ and also in $\Gamma^{\prime}$, it may or may not be a cut vertex of $\Gamma$.
We will show that $\Gamma^{\prime}$ has at most $2(g-1)$ vertices that are of valence at least $3$ in $\Gamma^{\prime}$, by induction on $g$. 
Then, we can show that $\Gamma$ has at most $2(g-1)$ vertices that are of valence at least $3$ but are not cut vertices.

If $g=1$, $\Gamma^{\prime}$ has the shape of a polygon and it has no vertices of valence at least $3$.
For a general $g$, we choose one of the shortest paths $p$ in $\Gamma^{\prime}$ among all paths between (not necessarily distinct) vertices of valence at least $3$.
We write $v_1$ and $v_2$ for the two  (not necessarily distinct) ends of the path $p$.
The vertices $v_1$ and $v_2$ are the only two vertices on $p$ that have valence at least $3$.
Consider the graphs $\gamma$ and $\gamma^{\prime}$ obtained from $\Gamma$ and $\Gamma^{\prime}$ by removing the path $p$.
By induction hypothesis, $\gamma^{\prime}$ has at most $2(g-2)$ vertices of valence at least $3$.
Adding the path $p$ increases the number of vertices of valence at least $3$ by at most two.
Then $\Gamma^{\prime}$ has at most $2(g-1)$ vertices of valence at least $3$.

\end{proof}

\begin{proof}[Proof of Proposition  \ref{numberofinternal}]
Let $L$ be the number of internal black vertices of $\Gamma$. Assume that $L-2(g-1)>0$ is satisfied.
By Remark \ref{allowedvertices}, internal black vertices of $\Gamma$ may be assumed to have at least three solid edges.
Let $\{v_1, \dots v_L\}$ be the set of internal black vertices of $\Gamma$.
By Lemma \ref{lemmanumberofinternal}, $\Gamma$ has at least $L-2(g-1)>0$ cut vertices.
Then we have $\Gamma = 0 \in A_{n,j}$.

\end{proof}

\begin{prop}
\label{Degree-zero elements prop}
Let $\Gamma$ be an admissible plain graph ($|V(\Gamma)| \geq 3$) such that the normalized degree $[\Gamma]$ is zero. If $j \geq 3$ or $j=1$, then $\Gamma = 0 \in A_{n,j}$. If $j=2$ and $g < k(n-4) + 4$, then $\Gamma = 0 \in A_{n,j}$ .
\end{prop}
\begin{proof}
Recall that we write $l$ for the defect of $\Gamma$, which is the sum of the defects of the vertices of $\Gamma$.
If $\Gamma$ is an admissible plain graph, negative contributions to the defect only come from internal black vertices and each contribution is $-1$. See Remark \ref{allowedvertices}.
We first show the proposition in the case $g=0$. If $\Gamma$ has an internal black vertex, it is a cut vertex, so $\Gamma = 0$. Otherwise, the normalized degree 
\[
[\Gamma] =   k(n-j-2) + (g-1)(j-1) + l +(j+1) = k(n-j-2) + l + 2
\]
is positive since $l \geq 0$. 
Let $g \geq 1$. 
From Proposition \ref{numberofinternal}, we can assume that the number of internal black vertices is at most $2(g-1)$. Hence, if $j \geq 3$, we have
\begin{align*}
[\Gamma] &=   k(n-j-2) + (g-1)(j-1) + l +(j+1) \\
&\geq k(n-j-2) + (j+1) > 0,
\end{align*}
since $l \geq -2(g-1)$. On the other hand, if $j=2$, 
\begin{align*}
[\Gamma] &=   k(n-j-2) + (g-1)(j-1) + l +(j+1) \\
&\geq k(n-j-2) + (j+1) -(g-1) \\
& = k(n-4)+3-(g-1),
\end{align*}
which is positive if $g < k(n-4) + 4$.
\end{proof}

\begin{cor}
\label{cordegree0}
If $j \geq 3$, elements of $A_{n,j}$ of degree $0$ are scalar multiples of 
\[
\emptyset = \dashededgea{}{}{} = \dashededgec{}{}{} = \solidedgee{}{}{}. 
\]
If $j =1$, elements of $A_{n,j}$ of degree $0$ are scalar multiples of 
\[
\emptyset = \dashededgea{}{}{} = \dashededgec{}{}{}.
\]
If $j = 2$, similar results hold for the restriction $A_{n,j}(g \leq 3)$.
\end{cor}

\subsection{The decorated graph complex $DGC_{n,j}$}

We define the new graph complex $DGC_{n,j}$. This new complex is generated by \textit{decorated graphs}.
\begin{definition}
A \textit{decorated graph} is a plain graph whose external black vertices are decorated by an element of $Z = A_{n,j} \otimes_{\tau} BA_{n,j}$. A decorated graph is labeled if its plain part is labeled. Let  $\Gamma$ be a labeled decorated graph. We write the plain part as $P(\Gamma)$ and write the decorated part as $D(\Gamma) = z_1 \otimes \cdots \otimes z_l$, if $\Gamma$ has $l$ external black vertices and the $i$-th external vertex is decorated by $z_i \in Z$. We write $\Gamma = P(\Gamma) \otimes D(\Gamma)$. See Figure \ref{figofdecoratedgraph}. 
\end{definition}

\begin{figure}[htpb]
   \centering
    \includegraphics [width =7cm] {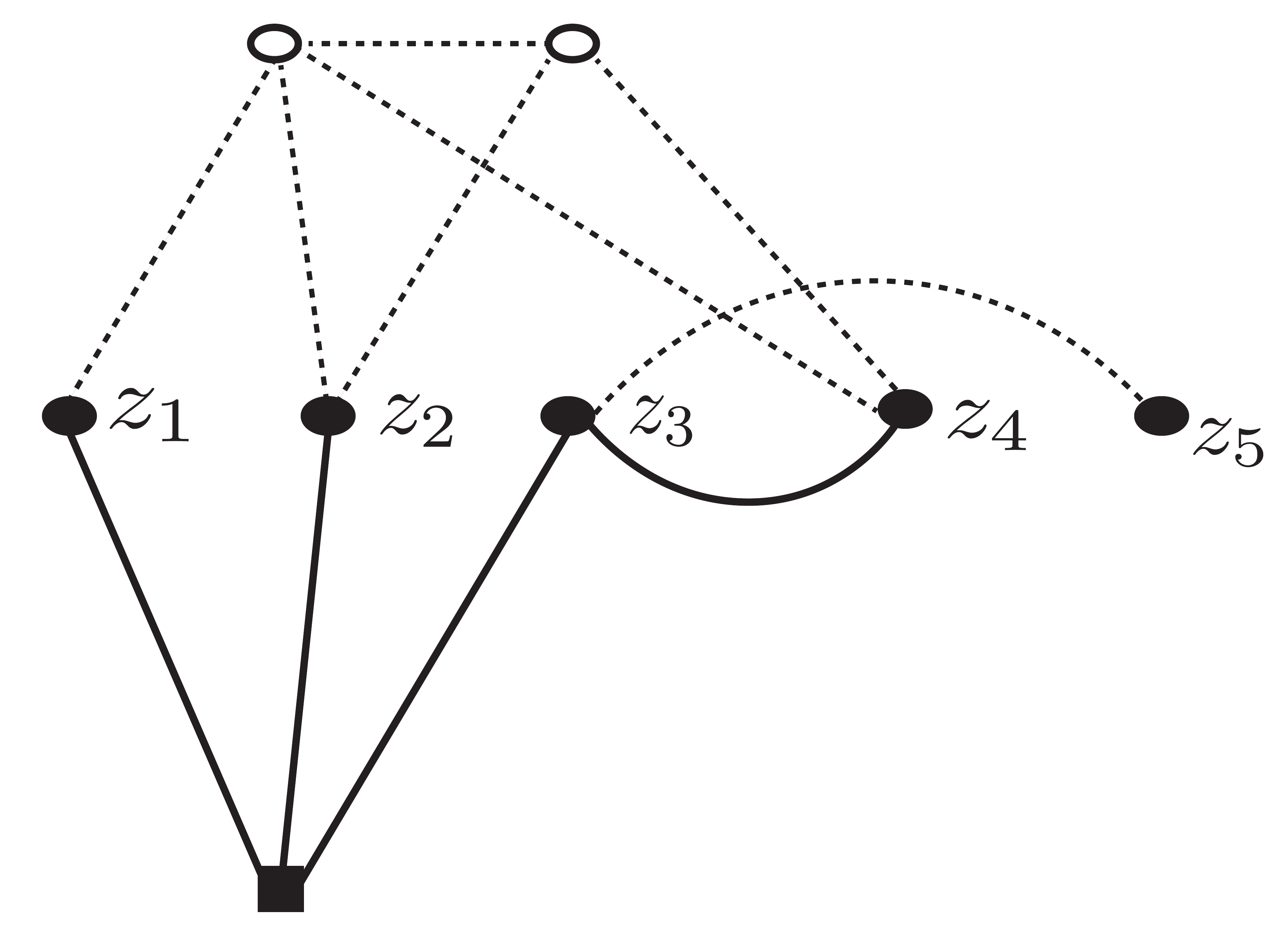}
    \caption{Example of an admissible decorated graph}
    \label{figofdecoratedgraph}
\end{figure}

\begin{definition}
A decoration $z$ of an external black vertex is called \textit{trivial} if $z$ is obtained by scalar multiplication of the unit of $A_{n,j}$. A decorated graph is admissible if any external black vertex of it has either a dashed edge or a non-trivial decoration. Note that the plain part of an admissible decorated graph is not necessarily admissible. That is, we allow external black vertices without dashed edges for the plain part. 
\end{definition}

\begin{definition}
As a graded vector space, 
\[
DGC_{n,j}= \frac{\mathbb{Q} \{\text{Connected labeled admissible decorated graphs} \}}{\text{Relations} },
\]
where relations consist of orientation relations, decoration relations, double edges, small loops. Orientation relations and decoration relations are defined below. 
\end{definition}

\begin{definition}
Orientation relations of $DGC_{n,j}$ are defined in the same way as the plain graph complex, except for relations on exchanging labels of two external black vertices. 
Assume two external black vertices  $v$ and $v^{\prime}$ have decorations $z$ and $z^{\prime}$ of homogeneous degree. Then, if we exchange labels of $v$ and $v^{\prime}$, we multiply by the sign $(-1)^j (-1)^{|z||z^{\prime}|}$. 
\end{definition}

\begin{definition}
We define decoration relations as follows.  For a decorated graph $\Gamma = P\otimes (z_1 \otimes \dots \otimes z_l)$ and $\Gamma^{\prime} =  P \otimes ({z_1}^{\prime} \otimes \dots \otimes {z_l}^{\prime})$ with the same plain part, 
\[
\Gamma + \Gamma^{\prime} \sim P \otimes ((z_1 + {z_1}^{\prime}) \otimes \dots \otimes (z_l + {z_l}^{\prime})), 
\]
and
\[
a \Gamma \sim P\otimes (z_1 \otimes \dots \otimes a z_i \otimes \dots \otimes z_l)
\]
for any $ a \in \mathbb{Q} $ and $i = 1, \dots,l$.
\end{definition}

Let us define the differential of decorated graphs. 

\begin{notation}
Let $S = \{s_1, \dots, s_{|S|}\}$ be an ordered subset of $ \{1, \dots, l\}$, $|S| \geq 2$. Define the map
\[
m_S: Z^{\otimes l} \rightarrow Z^{\otimes l-S+1}
\]
by multiplying the $s_1$-th, $s_2$-th, $\dots$ and $s_{|S|}$-th terms. For example,
\[
m_{1,3}(z_1 \otimes z_2 \otimes z_3) = (-1)^{|z_1||z_2|}z_1z_3 \otimes z_2,
\]
if $z_2$ and $z_3$ are homogeneous. 
\end{notation}

\begin{notation}
Let $\Gamma$ be a labeled decorated graph such that $D(\Gamma) = z_1 \otimes \cdots \otimes z_l$. Let $S$ be a subset of $V(\Gamma)$, $|S| \geq 2$. Define a new labeled decorated graph $\Gamma/S$ by 
\[
\Gamma/S = P(\Gamma)/S \otimes m_S(z_1 \otimes \cdots \otimes z_l).
\]
\end{notation}

\begin{notation}
Define the right $i$-th action of $A_{n,j}$ on $Z^{\otimes l}$ as follows. If $z_1, \dots, z_l \in Z$ and $a \in A_{n,j}$, and if $a$ and $z_i$ are homogeneous, 
\begin{align*}
(z_1 \otimes \cdots \otimes z_l) \cdot_i a & = (z_1 \otimes \cdots \otimes z_l) \cdot (1 \cdots \otimes 1 \otimes a  \otimes 1  \cdots \otimes 1) \\
&= (-1)^{ |a|(|z_{i}| + \cdots + |z_l|)}(z_1 \cdots \otimes z_{i-1} \otimes a z_i \otimes z_{i+1}  \cdots \otimes z_l).
\end{align*}
\end{notation}

\begin{notation}
Let $\Gamma$ be a labeled plain graph and let $S=\{s_1, \dots s_l\}$ be an ordered subset of $V(\Gamma)$. Let $\Gamma_S$ be a subgraph of $\Gamma$ generated by $S$. We consider $\Gamma_S$ as an element of $A_{n,j}$. Define the new decorated graph $\Gamma/S \bullet \Gamma_S$ by
\[
\Gamma/S \bullet \Gamma_S = P(\Gamma/S) \otimes \left(D(\Gamma/S) \cdot_{s_1} \Gamma_S\right)
\]
See Figure \ref{figofsubgraphcontraction}.
\end{notation} 

\begin{figure}[htpb]
   \centering
    \includegraphics [width =9cm] {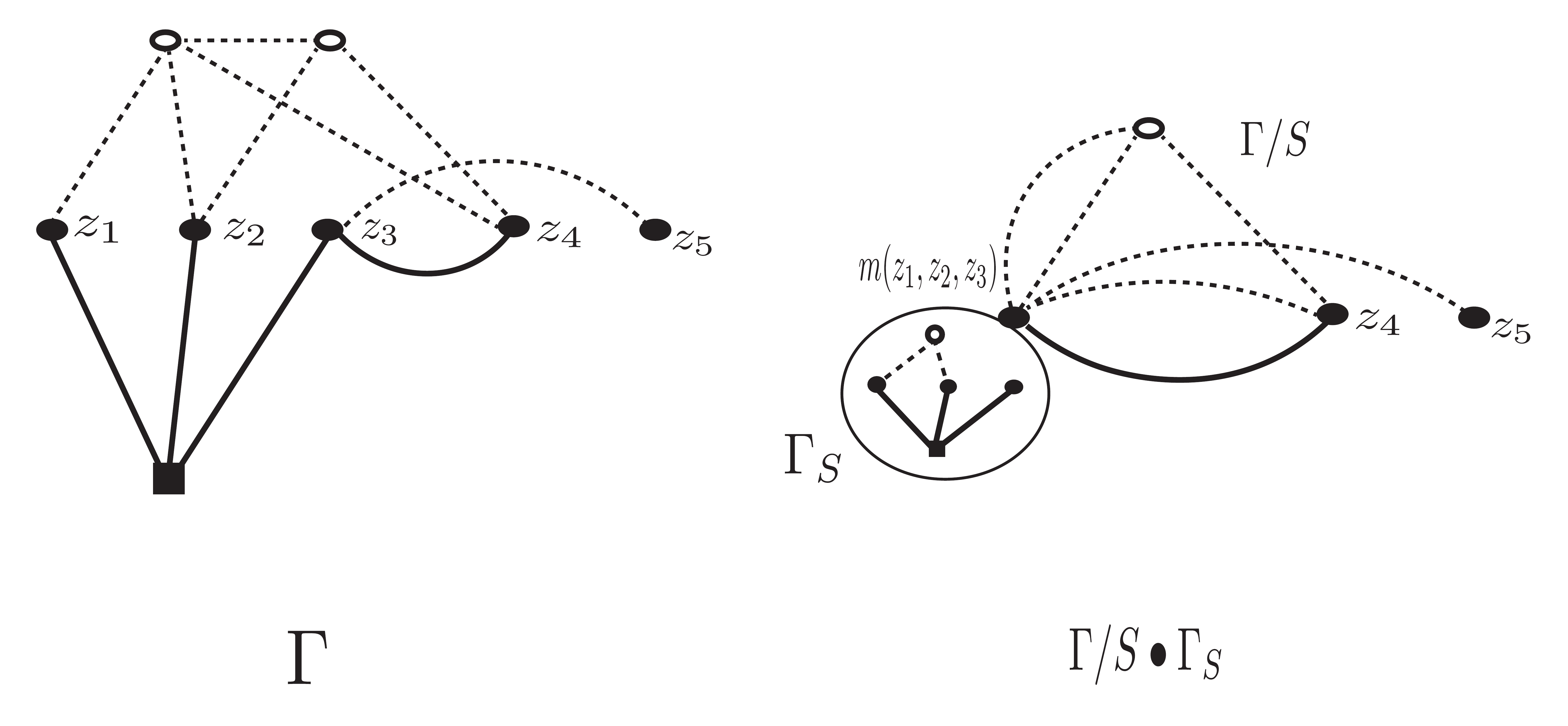}
    \caption{Example of a contraction $\Gamma/S \bullet \Gamma_S$}
    \label{figofsubgraphcontraction}
\end{figure}

Recall that $d_Z$ is the differential on the bar complex $Z =  B(A_{n,j} ,A_{n,j}, \mathbb{R})$.

\begin{definition}
The differential $d_{DGC}$ on $DGC_{n,j}$  is decomposed into two commutative differentials
\[
d_{DGC} = d_{H} + d_{V}.
\]
The first differential $d_{H}$, which we call the \textit{horizontal differential}, is defined by the sum of contractions of subgraphs. 
\[
d_{H} (\Gamma) = \sum_{\substack{S\subset V(P(\Gamma)) \\ |S| \geq 2}} \sigma (S)\ (\Gamma/S \bullet \Gamma_S).
\]
The sign $\sigma (S)$ is defined  by
\begin{align*}
\sigma(S) &= (-1)^{|\Gamma|+1} \tau(S) \Phi(S) (-1)^{deg(V(P(\Gamma/S)))} \\
& = (-1)^{|\Gamma|+1} \tau(S) \Phi(S) (-1)^{-js-nt+ deg(S)-j} \ (\text{if $S$ has a black vertex}).
\end{align*}
This sign $\sigma(S)$ is taken so that $d_{H}$ is compatible with configuration space integrals for decorated graphs that we later define:
\[
(-1)^{|\Gamma|+1} \int_{\tilde{C}_{S}} \omega(\Gamma) = \sigma (S)\ \overline{I}(\Gamma/S \bullet \Gamma_S).
\]
The second differential $d_{V}$,  which we call the \textit{vertical differential}, is the differential induced by $d_{Z}$:
\begin{align*}
d_{V} (\Gamma) &= (-1)^{deg(E(P(\Gamma)))} P(\Gamma) \otimes d_{Z}(D(\Gamma))\\
& = (-1)^{(j-1)|E_{\eta}(\Gamma)| + (n-1)|E_{\theta}(\Gamma)|} P(\Gamma) \otimes d_{Z}(D(\Gamma)) .
\end{align*}
 \end{definition}
 
 \begin{rem}
 $DGC_{n,j}$ has a decomposition 
 \[
DGC = \bigoplus_{k \geq 1} \bigoplus_{g\geq 0} DGC(k, g)
 \]
$DGC(k, g)$ is generated by decorated graphs such that 
\begin{align*}
k(P(\Gamma)) + k(D(\Gamma)) &= k, \\
g(P(\Gamma)) + g(D(\Gamma)) &= g.
\end{align*}
 \end{rem}
 
 \subsection{An operadic description of $DGC_{n,j}$}
 
\begin{prop}
Let $A_{n,j}$ be the hidden face dg algebra. Write $Z$ for the bar complex $A_{n,j} \otimes_{\tau} BA_{n,j} \simeq \mathbb{Q}$ of $A_{n,j}$. 
Then the decorated graph complex $\text{DGC}_{n,j}$ is expressed as the subcomplex of 
\[
DGC_{n,j} = \bigoplus_{l\geq1} \left((\text{sign}_l)^{\otimes j} \otimes \overline{\text{Graph}}_j(l) \otimes N(\overline{\text{Graph}}_n(l) \otimes (Z)^{\otimes l})\right) /\Sigma_l.
\]
generated by connected graphs.
\end{prop}

There are canonical projections of vector spaces 
\begin{align*}
&P: DGC(l) \rightarrow pPGC(l) / \Sigma_l  , \\
&D: DGC(l) \rightarrow Z^{{\otimes} l}/ \Sigma_l.
\end{align*}

Here, $pPGC_{n,j}$ is the graph complex of possibly non-admissible graphs:
\[
pPGC_{n,j}= \frac{\mathbb{Q} \{\text{connected labeled plain graphs} \}}{\text{orientation relations, double edges, loop edges} }.
\]

 \subsection{Relationship between $DGC_{n,j}$ and $PGC_{n,j}$}
 We show that the decorated graph complex $DGC_{n,j}$ and the plain graph complex $PGC_{n,j}$ are quasi-isomorphic (Theorem \ref{main theorem 2 on hidden faces}). We use the decomposition with respect to the order and the first Betti number of graphs. 

\begin{theorem}[Theorem \ref{main theorem 2 on hidden faces}, suggested by Victor Turchin]
\label{maintheoremrestate}
The projection  
\[
p: DGC(k,g) \rightarrow PGC(k,g)
\]
is a quasi-isomorphism.
\end{theorem}

We use a bounded spectral sequence to prove this theorem. 

\begin{definition}
\label{defoffiltration}
Define a decreasing filtration of $DGC(k,g)$ by 
\[
F^p DGC(k,g) = \bigoplus_{-p \geq l \geq 1} DGC(k,g, l) \quad p \in \mathbb{Z}
\]
Here $DGC(k,g, l)$ is the subspace of $DGC(k,g)$ generated by decorated graphs whose plain part has exactly $l$ vertices. 
\end{definition}

\begin{lemma}
Fixing $k$ and $g$, this filtration is bounded. That is, for any degree $n$ and for any $k,g$, there exists $p$ such that 
\[
DGC^n(k,g) = F^p DGC^n(k,g).
\]
\end{lemma}
\begin{proof}
We will find a lower bound which depends on $k$ and $g$ but does not depend on the degree $n$.
First, we have $|E_{\theta}(P(\Gamma))| \leq 3k$ and $|W(P(\Gamma))| \leq 2k$. In fact, since the order of each decoration is at least $1$, 
\begin{align*}
\frac{1}{2}|W(P(\Gamma))| & = \frac{3}{2}|W(P(\Gamma))| - |W(P(\Gamma))|  \\
&\leq |E_{\theta}(P(\Gamma))| - |W(P(\Gamma))| \\
&\leq |E_{\theta}(\Gamma)| - |W(\Gamma)| = k.
\end{align*}
Moreover, since the number of (non-trivial) decorations is at most $k$,
\[
|B_e(P(\Gamma))| \leq 2 |E_{\theta}(P(\Gamma))| + k \leq 7k.
\]
On the other hand,
\begin{align*}
\frac{1}{2}|B_i(P(\Gamma))|  &\leq  |E_{\eta}(P(\Gamma))| - |B_i(P(\Gamma))| \\
&=  |B_e(P(\Gamma))| + |W(P(\Gamma))|  - |E_{\theta}(P(\Gamma))| + (g-1) \\
& \leq 7k +(g-1).
\end{align*}
\end{proof}

Recall that the $E_1$-term of the spectral sequence associated with the filtration of  \ref{defoffiltration} is 
\[
E^{p,q}_1 = H^{p+q} ( F^p DGC^{\ast} / F^{p+1} DGC^{\ast}). 
\]
Recall that $d_{DGC}$ is decomposed as $d_{DGC} = d_H + d_V$. The horizontal differential $d_H$ decreases the number of vertices of plain parts, while the vertical differential $d_V$ preserves it. So we have 
\begin{lemma}
$E_1^{p,q}$ is isomorphic to the $(p+q)$th cohomology of the complex 
\[
d_V: DGC(k,g, -p) \rightarrow DGC(k, g, -p).
\]
\end{lemma}

Theorem \ref{maintheoremrestate} follows from the theorem below. 

\begin{theorem}
\label{keylemma}
The $E_1$-term of the spectral sequence $E$ is 
\[
E_1^{p,q} = PGC^{p+q}(-p).
\]
Here, $PGC^{p+q}(-p)$ is generated by plain graphs of degree  $(p+q)$ with $-p$ vertices.
The differential 
\[
d_1^{p,q}: E_1^{p,q} \rightarrow E_1^{p+1 ,q}
\]
is equal to  $d_{PGC}$.
\end{theorem}

\begin{lemma}
\label{collapseofPGC}
Fixing $k$ and $g$, this spectral sequence collapses at the $E_2$-page.
\end{lemma}
\begin{proof}
Recall that for a graph of $PGC^{p+q}(-p)$, we have
\begin{align*}
k &= |E_{\theta}|-|W|\\
g &= |E_{\theta}|+|E_{\eta}|-|B|-|W|+1\\
-p &= |B| + |W| \\
p+q &= (n-1)|E_{\theta}| +(j-1)|E_{\eta}|-n|W|-j|B|.
\end{align*}
Hence, we have 
\begin{align*}
p+q &= p+2k+k(n-j-2)+(g-1)(j-1).
\end{align*}
Therefore, fixing $k$ and $g$, the diagrams on the $E_1$-page are supported on a single horizontal line. Hence, the same is true for the $E_2$-page. 
Since every differential $d_r (r \geq 2)$ changes the horizontal coordinate, all higher differentials vanish.
\end{proof}

Theorem \ref{keylemma} will be a consequence of acyclicity of $Z$.

\begin{prop}
The complex $(Z, d_Z)$ is acyclic. That is,  $ (Z, d_Z)$ is quasi-isomorphic to the complex $\mathbb{Q}$ with the trivial differential. The cohomology $H^{\ast}(Z)$ is generated by the trivial decoration $1 \in A_{n,j}$. Similarly, the complexes $(Z^{\otimes k}, d_Z)$ and $(Z^{\otimes k}/\Sigma_k, d_Z)$ are acyclic. 
\end{prop}

\begin{proof}
Define $h: Z \rightarrow Z$ by 
\[
a_0[sa_1|sa_2| \dots |sa_k] \mapsto [s\overline{a_0}| sa_1|sa_2| \dots |sa_k],
\]
where $a_0 = \overline{a_0} + \varepsilon(a_0) \in \overline{A}_{n,j} \oplus \mathbb{R}$.
Then we have $hd-dh = id$.
\end{proof}

\begin{proof}[Proof of Theorem \ref{keylemma}]
Consider the map
\begin{align*}
P\otimes D: H^{\ast}(DGC(k,g, -p), d_V) &\rightarrow pPGC(k,g, -p) \otimes H^{\ast}(Z^{\otimes (-p)} / \Sigma_{(-p)}) \\
& \xrightarrow{\cong} pPGC(k,g, -p)
\end{align*}
This map is injective and its image lies in $PGC(k,g, -p)$. In fact, if the plain part $P(\Gamma)$ vanishes, the decorated graph also vanishes. On the other hand, any admissible plain graph is  always a cocycle of $(DGC(k,g, -p), d_V)$. 
\end{proof}

The following property is used in the study of pairing between cocycles and cycles of $\overline{\mathcal{K}}_{j+2 ,j}$ in \cite{Yos 3}.
\begin{prop}
\label{keypropforpairing}
Let $\Gamma$ be a decorated graph of order $\leq k$.
Assume that the plain part $P(\Gamma)$ has at least $2k$ external vertices with dashed edges. 
Then $\Gamma$ is a plain graph without white vertices. Moreover, $\Gamma$ has exactly $k$ dashed edges and $2k$ external black vertices. 
\end{prop}

\begin{proof}
First, we have
\[
\frac{3}{2}|W(P(\Gamma))| + k - |W(P(\Gamma))| \leq |E_{\theta}(P(\Gamma))| - |W(P(\Gamma))| \leq k.
\]
Thus $|W(P(\Gamma))| = 0$ and $|E_{\theta}(P(\Gamma))| = |E_{\theta}(P(\Gamma))| - |W(P(\Gamma))| = k$. By the second equality, the graph $\Gamma$ has no decoration. Every external black vertex has exactly one dashed edge.
\end{proof}

\begin{example}
Recall that when $n$ and $j$ are even, 
\begin{align*}
H^{top}(PGC_{n,j})(k=1,g=1) & = \mathbb{Q}
(
\begin{tikzpicture}[x=1pt,y=1pt,yscale=1,xscale=1, baseline=-3pt, line width =1pt]
\draw  [fill={rgb, 255:red, 0; green, 0; blue, 0 }  ,fill opacity=1 ] (0,0) circle (2);
\draw  [fill={rgb, 255:red, 0; green, 0; blue, 0 }  ,fill opacity=1 ] (50,0) circle (2);
\draw  [dash pattern = on 3pt off 3 pt] (0,0) .. controls (20,5) and (30,5)  .. (50,0);
\draw  (0,0) .. controls (20,-5) and (30,-5)  .. (50,0);
\end{tikzpicture}).
\end{align*}
We have 
\begin{align*}
H^{top}(DGC_{n,j})(k=1,g=1) & = \mathbb{Q}
(
\begin{tikzpicture}[x=1pt,y=1pt,yscale=1,xscale=1, baseline=-3pt, line width =1pt]
\draw  [fill={rgb, 255:red, 0; green, 0; blue, 0 }  ,fill opacity=1 ] (0,0) circle (2);
\draw  [fill={rgb, 255:red, 0; green, 0; blue, 0 }  ,fill opacity=1 ] (50,0) circle (2);
\draw  [dash pattern = on 3pt off 3 pt] (0,0) .. controls (20,5) and (30,5)  .. (50,0);
\draw  (0,0) .. controls (20,-5) and (30,-5)  .. (50,0);
\end{tikzpicture}
+
\begin{tikzpicture}[x=1pt,y=1pt,yscale=1,xscale=1, baseline=-3pt, line width =1pt]
\draw  [fill={rgb, 255:red, 0; green, 0; blue, 0 }  ,fill opacity=1 ] (0,0) circle (2);
\draw (0,0) node [anchor = south] {$z$};
\end{tikzpicture}
),
\end{align*}
where $z= [sa] \in Z_{n,j} $ and $a \in A_{n,j} $ is the element of degree $n-1$ represented by the graph
 \begin{tikzpicture}[x=1pt,y=1pt,yscale=1,xscale=1, baseline=-3pt, line width =1pt]
\draw  [fill={rgb, 255:red, 0; green, 0; blue, 0 }  ,fill opacity=1 ] (0,0) circle (2);
\draw  [fill={rgb, 255:red, 0; green, 0; blue, 0 }  ,fill opacity=1 ] (50,0) circle (2);
\draw  [dash pattern = on 3pt off 3 pt] (0,0) .. controls (20,5) and (30,5)  .. (50,0);
\draw  (0,0) .. controls (20,-5) and (30,-5)  .. (50,0);
\end{tikzpicture}.
\end{example}

\begin{example}
Recall that when $n$ and $j$ are odd, 
\begin{align*}
H^{top}(PGC_{n,j})(k=2,g=1) & = \mathbb{Q}
(
 \graphE{} + 2 \graphDb{} +\frac{1}{2}\graphDc{} 
). \\
\end{align*}
We have 
\begin{align*}
H^{top}(DGC_{n,j})(k=2,g=1)  = \mathbb{Q} 
( &\graphE{} + 2 \graphDb{} +\frac{1}{2}\graphDc{} \\
+ &  2
 \begin{tikzpicture}[x=1pt,y=1pt,yscale=-0.3,xscale=0.3,baseline=-85pt, line width=1pt]
\draw [dash pattern={on 3pt off 2pt}] [-stealth]  (90,270)..controls (130,220)..(165,265) ;
\draw [-stealth] (90, 270)--(165,270) ;
\draw [-stealth](170, 270)--(245,270) ;
\draw  [fill={rgb, 255:red, 0; green, 0; blue, 0 }  ,fill opacity=1 ] (90,270) circle (5);
\draw (90, 290) node {$1$};
\draw  [fill={rgb, 255:red, 0; green, 0; blue, 0 }  ,fill opacity=1 ] (170,270) circle (5);
\draw (170, 290) node {$2$};
\draw [fill={rgb, 255:red, 0; green, 0; blue, 0 }  ,fill opacity=1 ] (250,270) circle (5);
\draw (250, 290) node {$3$};
\draw (250, 270) node [anchor = south] {$z^{\prime}$};
\end{tikzpicture}
 +
\begin{tikzpicture}[x=1pt,y=1pt,yscale=-0.3,xscale=0.3,baseline=-85pt, line width=1pt]
\draw [dash pattern={on 3pt off 2pt}] [-stealth]  (90,270)..controls (170,220)..(245,265) ;
\draw [-stealth] (90, 270)..controls (185,350).. (330,275);
\draw [-stealth](250, 270)--(325,270) ;
\draw  [fill={rgb, 255:red, 0; green, 0; blue, 0 }  ,fill opacity=1 ] (90,270) circle (5);
\draw (90, 290) node {$1$};
\draw [fill={rgb, 255:red, 0; green, 0; blue, 0 }  ,fill opacity=1 ] (250,270) circle (5);
\draw (250, 290) node {$2$};
\draw [fill={rgb, 255:red, 0; green, 0; blue,0 } ,fill opacity=1] (330,270) circle (5);
\draw (330, 295) node {$3$};
\draw (330, 270) node [anchor = south] {$z^{\prime}$};
\end{tikzpicture}
), \\
\end{align*}
where $z^{\prime} = [s a^{\prime}] \in Z_{n,j} $ and $a^{\prime} \in A_{n,j} $ is the element of degree $n-j$ represented by the graph
 \begin{tikzpicture}[x=1pt,y=1pt,yscale=1,xscale=1, baseline=-3pt, line width =1pt]
\draw  [fill={rgb, 255:red, 0; green, 0; blue, 0 }  ,fill opacity=1 ] (0,0) circle (2);
\draw (0,0) node [anchor = south] {$1$};
\draw  [fill={rgb, 255:red, 0; green, 0; blue, 0 }  ,fill opacity=1 ] (50,0) circle (2);
\draw (50,0) node [anchor = south] {$2$};
\draw  [dash pattern = on 3pt off 3 pt, -Stealth] (0,0) -- (50,0);
\end{tikzpicture}.

\end{example}
 
\section{The modified configuration space integrals}
\label{The configuration space integrals with iterated integrals}

From now on, the coefficient ring of our graph complexes is considered to be $\mathbb{\mathbb{R}}$.
In this section,  we first give a map of graded vector spaces 
\[
I: PGC_{n,j} \rightarrow \Omega_{dR}(\overline{\mathcal{K}}_{n,j}). 
\]
The map $I$ is given by configuration space integrals in the same way as Bott \cite{Bot}, Cattaneo, Rossi \cite{CR}, Sakai and Watanabe \cite{Sak, SW, Wat 1} give. 
Unfortunately, it is unknown whether these configuration space integrals give a cochain map, due to potential obstructions called \textit{hidden faces}.  However, we can show that $I$ is a cochain map ``up to homotopy''. Namely, we give another graph complex $DGC_{n,j}$ quasi-isomorphic to $PGC_{n,j}$, from which we can construct a cochain map.  See Theorem \ref{cochainmapuptohomotopy}.

\subsection{Some remarks on fiber integrals}
We review some facts on fiber integrals. Only in this subsection, we follow the outward normal convention for the orientation of boundaries.
\begin{theorem}[Stokes' theorem]
\label{Stokes' theorem}
\
Let $\pi: E \rightarrow B$ be an oriented bundle with fiber $F$. Let $\omega$ be a form on $E$. Then we have
\[
d \pi_{\ast} \omega = \pi_{\ast} d \omega + (-1)^{|\pi_{\ast} \omega|} {\pi_{\partial}}_{\ast} \omega.
\]
Here, the boundary of the fiber $F$ of $E$ is oriented by the outward normal convention. 
\end{theorem}

\begin{lemma}[Compatibility with pullback  \cite{BoTu}]
Let $\pi: E \rightarrow B$ be an oriented bundle with fiber $F$. Let $\omega$ be a form on $E$ and let $\theta$ be a form on $B$. Then we have
\[
\pi_{\ast}(\pi^{\ast} \theta \wedge \omega) = \theta \wedge \pi_{\ast} \omega.
\]
Note that from this lemma, we have $\pi_{\ast}(\omega \wedge \pi^{\ast} \theta) = (-1)^{|\theta||F|}\pi_{\ast} \omega \wedge \theta $.
\end{lemma}

\begin{prop}[Compatibility in a pullback diagram]
\label{Compatibility in pullback diagram}
Consider the following pullback diagram:
\begin{center}
\begin{tikzpicture}[auto]
\node (c) at (-2, 1) {$\overline{E}$}; \node (d) at (2, 1) {$E_2$}; 
\node (a) at (-2, -1) {$E_1$}; \node (b) at (2, -1) {$B$};

\draw[->] (c)  to node {$f_2$}  (d);
\draw[->] (a) to node {$\pi_1$} (b);
\draw [->] (c) to node  {$f_1$} (a);
\draw [->] (d) to  node {$\pi_2$} (b);
\end{tikzpicture}
\end{center}
where $E_1$ and $E_2$ are oriented bundles over $B$ with fibers $F_1$ and $F_2$ respectively. Note that $\overline{\pi}: \overline{E} \rightarrow B$ is an oriented bundle with fiber $F_1 \times F_2$. Let $\omega$ be a form on $E_2$. Then we have
\[
(f_1)_{\ast} (f_2)^{\ast} \omega = (\pi_1)^{\ast} (\pi_2)_{\ast} \omega.
\]

\end{prop}

\begin{prop}[Double push forward formula]
\label{Double push forward}
In the pullback diagram in Proposition \ref{Compatibility in pullback diagram}, let  $\omega_1$ and $\omega_2$ be forms on $E_1$ and $E_2$ respectively.
Then we have 
\begin{align*}
& \int_{F_1} \int_{F_2} (f_1^{\ast} \omega_1 \wedge f_2^{\ast} \omega_2) = (\pi_1)_{\ast} (f_1)_{\ast} (f_1^{\ast} \omega_1 \wedge f_2^{\ast} \omega_2) \\ 
= & (-1)^{|F_1||\int_{F_2}\omega_2|} \int_{F_1} \omega_1 \wedge \int_{F_2} \omega_2 =  (-1)^{|F_1||{\pi_2}_{\ast} \omega_2|}({\pi_1}_{\ast} \omega_1 \wedge {\pi_2}_{\ast} \omega_2) 
\end{align*}
\end{prop}

\begin{proof}
By using Proposition \ref{Compatibility in pullback diagram} and \ref{Double push forward}, we have
\begin{align*}
&(\pi_1)_{\ast} (f_1)_{\ast} (f_1^{\ast} \omega_1 \wedge f_2^{\ast} \omega_2) \\
= & (\pi_1)_{\ast} (\omega_1 \wedge (f_1)_{\ast} (f_2)^{\ast} \omega_2) \\
= & (\pi_1)_{\ast} (\omega_1 \wedge (\pi_1)^{\ast} (\pi_2)_{\ast}  \omega_2) \\
= & (-1)^{|F_1||{\pi_2}_{\ast} \omega_2|}  ({\pi_1}_{\ast} \omega_1 \wedge {\pi_2}_{\ast} \omega_2).
\end{align*}
\end{proof}

\subsection{The configuration space integrals $I: PGC_{n,j} \rightarrow \Omega_{dR}^{\ast}(\overline{\mathcal{K}}_{n,j})$} 

We define a map of graded vector spaces $I: PGC_{n,j} \rightarrow \Omega_{dR}^{\ast}(\overline{\mathcal{K}}_{n,j})$ by configuration space integrals. Note that the data of paths of immersions in 
$\overline{\mathcal{K}}_{n,j}$, are not used in the definition of $I$ in this section. So we can construct $I: PGC_{n,j} \rightarrow \Omega_{dR}^{\ast}(\mathcal{K}_{n,j})$. 

\begin{notation}
We write $C_{k}(\mathbb{R}^n)$ for the configuration space of $k$ points in $\mathbb{R}^n$. 
\end{notation}

Although $C_{k}(\mathbb{R}^n)$ is an open manifold, there exists a canonical compactification $\overline{C}_{k}(\mathbb{R}^n)$ of $C_{k}(\mathbb{R}^n)$, called the \textit{Fulton--Macpherson compactification}.

Recall that $\overline{\mathcal{K}}_{n,j}$ consists of a family of immersions $\{\overline{\psi}_u\}_{u \in [0,1]}$, ${\overline{\psi}}_{u} \in \text{Imm}(\mathbb{R}^j, \mathbb{R}^n)$ such that ${\overline{\psi}}_{0} \in \mathcal{K}_{n,j}$ and ${\overline{\psi}}_{1}$ is the trivial immersion. We use the following bundle $E_{s,t}$ to define the configuration space integral associated to a plain graph with $s$ black vertices and $t$ white vertices. 

\begin{definition} [\textit{Configuration space bundles}]
\label{defofconfigurationspacebundle}
$E_{s,t} = E_{s,t}(\mathbb{R}^j, \mathbb{R}^n)$ is the bundle over $\overline{\mathcal{K}}_{n,j}$ defined by the pullback
\begin{center}
\begin{tikzpicture}[auto]
\node (c) at (3, 4) {$C_{s+t}(\mathbb{R}^n)$};
\node (b) at (-3, 2.5) {$\overline{\mathcal{K}}_{n,j} \times C_{s}(\mathbb{R}^j)$};
\node (a) at (3, 2.5) {$C_{s}(\mathbb{R}^n)$};
\node (f) at (-3,4) {$E_{s,t}$};

\draw [->] (f) to node {$$} (c);
\draw [->] (f) to node {$$}  (b);

\draw [->] (c) to node {restriction} (a);
\draw [->] (b) to node {evaluation at $u=0$} (a);
\end{tikzpicture}.
\end{center}
The typical fiber (at $\overline{\psi} \in \overline{K}_{n,j})$ of the projection $E_{s,t} \rightarrow \overline{\mathcal{K}}_{n,j}$ is written by $C_{s,t} = C_{s,t}(\overline{\psi})$. $C_{s,t}$ is the space of configurations of $(s+t)$ points in $\mathbb{R}^n$ such that the first $s$ vertices are images of $s$ points on $\mathbb{R}^j$.
\end{definition}

Let $\Gamma$ be a labeled plain graph with $s$ black vertices
and $t$ white vertices.
Then, each oriented dashed (resp. solid) edge $e$ gives a map
\[
P_e : E_{s,t} \rightarrow S^{n-1}\quad (\text{resp}.\ P_e : E_{s,t} \rightarrow S^{j-1}).
\]
by assigning the direction from the initial point to the end point. Thus, we have the direction map
\[
P_{\Gamma} = \prod_e  P_e  : E_{s,t} \rightarrow (S^{j-1})^{E_{\eta}(\Gamma)} \times (S^{n-1})^{E_{\theta}(\Gamma)}.
\]
Below, we consider $\Omega^{\ast}_{dR}(\overline{\mathcal{K}}_{n,j})$ as the mapping space of simplicial sets $\mathbf{Sset}(Sing^{\infty}_{\ast}(\overline{\mathcal{K}}_{n,j}), \Omega^{\ast}_{dR}(\Delta))$. Namely, a $k$-form of $\overline{\mathcal{K}}_{n,j}$ assigns a $k$-form on $\Delta^m$ for each singular smooth $m$-simplex $f: \Delta^m \rightarrow \overline{\mathcal{K}}_{n,j}$.

\begin{definition}[Configuration space integrals]
Define a form $I(\Gamma) \in \Omega_{dR} (\overline{\mathcal{K}}_{n,j})$ as follows. For a simplex $f: \Delta_m \rightarrow \overline{\mathcal{K}}_{n,j}$, $I(\Gamma)(f)$ is given by the fiber integral
\[
I(\Gamma)(f) =   \pi_{\ast} \omega_f(\Gamma) = \int_{\overline{C}_{s,t}} \omega_f(\Gamma) \in \Omega_{dR}(\Delta^m),
\]
where $\omega_f(\Gamma)$ is the pullback of the standard volume forms by the direction map $P_{\Gamma}$ as shown in the diagram below:
\[
\omega_f(\Gamma) = (P(\Gamma)\circ f) ^\ast (\bigwedge \omega_{S^{j-1}} \wedge \bigwedge \omega_{S^{n-1}}). 
\].

\begin{center}
\begin{tikzpicture}[auto]
\node (c) at (4, 8) {$E_{s,t}$}; \node (d) at (10, 8) {$(S^{j-1})^{E_{\eta}(\Gamma)} \times (S^{n-1})^{E_{\theta}(\Gamma)}$}; 
\node (b) at (-2, 6) {$\Delta^m$}; \node (a) at (4, 6) {$\overline{\mathcal{K}}_{n,j}$};
\node (f) at (-2,8) {$f^{\ast} E_{s,t}$};

\draw [->] (f) to node {$f$} (c);
\draw [->] (f) to node {$\pi$}  (b);

\draw[->] (c)  to  node {$P_\Gamma = \prod_e  P_e$}(d);
\draw [->] (c) to node {$\pi$} (a);
\draw [->] (b) to node {$f$: a smooth simplex } (a);
\end{tikzpicture}
\end{center}
\end{definition}

Unfortunately, as mentioned before, it is not known whether the map of graded vector spaces
\[
I: PGC \rightarrow \Omega_{dR}(\overline{\mathcal{K}}_{n,j})
\]
gives a cochain map. 
More precisely, by Stokes' theorem (Theorem \ref{Stokes' theorem}), we have
\[
(-1)^{|\Gamma|+1} d I (\Gamma)  = \int_{\partial \overline{C}_{s,t}} \omega(\Gamma) = \sum_{\substack{S\subseteq V(\Gamma)\cup\infty  \\  |S| \geq 2}}\ \int_{\widetilde{C}_S} \omega(\Gamma), 
\]
where $\widetilde{C}_S$ is the space of configurations such that the vertices of $S$ are infinitely close \footnote{We follow the inward normal convention for orientations of boundaries of configuration spaces. Then, the sign on the left-hand side of the formula is $(-1)^{|\Gamma|+1}$, not $(-1)^{|\Gamma|}$}. The explicit definition of $\widetilde{C}_S$ is given in subsection \ref{Hidden and infinite faces}.

\begin{definition}
Let $\Gamma$ be a (connected) plain graph. Let $S\subset \{1, \dots, s, s+1, \dots, s+t, \infty\} = V(\Gamma) \cup \infty $ and $ |S| \geq 2$. We classify the faces $\widetilde{C}_S$  into three types as follows.
\begin{itemize}
\item   \textit{Infinite faces}: $\infty \in S$.
\item   \textit{Principal faces}: $|S| = 2$ and $\infty \notin S$.
\item   \textit{Hidden faces}: $ |S| \geq 3$, $\infty \notin S$. 
\end{itemize}
(In this section, we regard the anomalous face $S = V(\Gamma)$ as a hidden face. )
\end{definition}

\begin{theorem}
Let $\Gamma$ be an admissible plain graph without double edges or loop edges. Let $S\subset V(\Gamma)$ and $|S|= 2$. Then the integral $\int_{\widetilde{C}_{S}} \omega(\Gamma)$ survives only when $S = e$ for some edge $e$ of $\Gamma$. Moreover, if $e$ is neither a dashed chord nor a multiple edge,  we have
\[
(-1)^{|\Gamma|+1} \int_{\widetilde{C}_{e}} \omega(\Gamma) =  \int_{C_{V(\Gamma/e)}} \omega(d_e\Gamma). \]
Here $C_{V(\Gamma/e)}$ is the configuration space of vertices of $\Gamma/e$. The graph $d_e\Gamma $ is defined by $d_e\Gamma = \sigma(e) \Gamma/e$, which appears in the differential of the plain graph complex. See Definition \ref{diffofPGC}.
\end{theorem}
\begin{proof}
Let $S = S^{j-1}$ or $S = S^{n-1}$, depending on whether $e$ is solid or dashed. Then we have
\begin{align*}
(-1)^{|\Gamma|+1} \int_{\widetilde{C}_e} \omega(\Gamma) &= (-1)^{|\Gamma|+1} (-1)^{deg(V(\Gamma/e))} \Phi(e) \tau(e) \eta(e) \int_{C_{V(\Gamma/e)} \times S} \omega(\Gamma/e) \wedge \omega_S\\
& = \int_{C_{V(\Gamma/e)}} \omega(d_e\Gamma).
\end{align*}
\end{proof}

As we see in Lemma \ref{infinitefacecontributions}, the contribution of infinite faces vanishes (under the assumption of Theorem \ref{cochainmapuptohomotopy} below). So we have
\[
dI(\Gamma) - I(d \Gamma) =  (-1)^{|\Gamma|+1} (\sum_{\substack{S\subseteq V(\Gamma)  \\  |S| \geq 3}} +  \sum_{\substack{S = V(e) \\ e:\ \text{dashed chord or multiple}}})  \int_{\widetilde{C}_S} \omega(\Gamma). 
\]
The obstructions on the right-hand side are called \textit{hidden face contributions}. Some hidden faces vanish by symmetries and rescalings of the faces. Other faces are canceled by introducing correction terms. 

We may interpret adding correction terms as replacing graph complexes. In the rest of the paper, we show the following. 

\begin{theorem}[Theorem \ref{main theorem 1 on hidden faces}]
\label{cochainmapuptohomotopy}
When  $j \geq 3$ or $j=1$, there exists a zigzag
\[
PGC_{n,j} \xleftarrow [p]{\simeq}  DGC_{n,j} \xrightarrow[\overline{I}]{}\Omega^{\ast}_{dR}(\overline{\mathcal{K}}_{n,j}), 
\]
of cochain maps, which is compatible with the map $I$. It holds when $j = 2$, if we restrict the graph complexes to the case $g \leq 3$. 
\end{theorem}

Note that we already showed that the projection $PGC_{n,j} \xleftarrow [p]{\simeq}  DGC_{n,j}$ is a quasi-isomorphism in Theorem \ref{maintheoremrestate}.

\subsection{Hidden and infinite faces}
\label{Hidden and infinite faces}
We recall the definition of principal and hidden faces $\widetilde{C}_S$ $(|S| \geq 2)$ and infinite faces $\widetilde{C}_{S\cup \infty}$ $(|S| \geq 1)$. 

We only handle the case where the subset $S\subset V(\Gamma)=\{1, \dots, s, s+1, \dots, s+t\}$ has at least one black vertex. The case where $S$ includes no black vertex is described similarly; replace $\text{Inj}(\mathbb{R}^j , \mathbb{R}^n)$ below with a single point.  However, we can observe that such faces without black vertices vanish as we see in Proposition \ref{welldefofJ}.  
 
We define the space $\widetilde{E}_S(\mathbb{R}^j, \mathbb{R}^n)$ as the pullback of the following diagram.  The configuration space $\widetilde{C}_S(\overline{\psi})$ is the fiber  of the projection $\widetilde{E}_S(\mathbb{R}^j, \mathbb{R}^n) \rightarrow \overline{\mathcal{K}}_{n,j}$ at an embedding $\overline{\psi}$ (modulo immersions).
\begin{center}
\begin{tikzpicture}[yscale=0.6,xscale=0.6]
\node (c) at (-2, 2) {$\widetilde{E}_S(\mathbb{R}^j, \mathbb{R}^n)$}; \node (d) at (5, 2) {$D_S(\mathbb{R}^j, \mathbb{ R}^n)$}; 
\node (a) at (-2, -1) {$E_{V/S}(\mathbb{R}^j, \mathbb{R}^n)$}; \node (b) at (5, -1) {$\text{Inj}(\mathbb{R}^j, \mathbb{R}^n)$};
\node (e) at (-2, -4) {$\overline{\mathcal{K}}_{n,j}$}; 

\draw[->] (c)  to (d);
\draw[->] (a) to node [anchor = north] {$D$} (b);
\draw [->] (c) to  (a);
\draw [->] (d) to  node [anchor = west] {$b_S$} (b);
\draw [->] (a) to  node [anchor = east] {$\pi$} (e); 
\end{tikzpicture}
\end{center}
Here, {$\text{Inj}(\mathbb{R}^j, \mathbb{R}^n)$} is the space of linear injective maps from $\mathbb{R}^j$ to $\mathbb{R}^n$. 
$E_{V/S}(\mathbb{R}^j, \mathbb{R}^n)$ is the bundle over $\mathcal{K}_{n,j}$ whose fiber is the configuration space $C_{V/S}$ of $V(\Gamma)/S$. $D_S(\mathbb{R}^j, \mathbb{R}^n)$  is the bundle over $\text{Inj}(\mathbb{R}^j, \mathbb{R}^n)$ with the ``normalized configuration space'' of $S$ as the fiber (Definition  \ref{normalizedconfigurationspacebundle}). 
$D$  is a map from $E_{V/S}(\mathbb{R}^j, \mathbb{R}^n)$  to $\text{Inj} (\mathbb{R}^j ,\mathbb{R}^n) $ which assigns the differential of embeddings at the collapsed point $S/S$. 

On the other hand, the space $\widetilde{E}_{S\cup \infty}(\mathbb{R}^j, \mathbb{R}^n)$ is identified with
\[
\widetilde{E}_{S\cup \infty}(\mathbb{R}^j, \mathbb{R}^n) = E_{V \setminus S}(\mathbb{R}^j, \mathbb{R}^n) \times B_{S\cup \infty}(\iota)
\]
Here, $B_{S\cup \infty}(\iota)$ is the fiber of $D_{S\cup \infty}(\mathbb{R}^j, \mathbb{R}^n)$ at the standard inclusion $\iota: \mathbb{R}^j \rightarrow \mathbb{R}^n$, where $\mathbb{R}^j$ and $\mathbb{R}^n$ are identified with $T_{\infty} S^j$ and $T_{\infty} S^n$ respectively via the map 
\[
\mathbb{R}^n \rightarrow S^n \setminus 0 = (\mathbb{R}^n \setminus 0) \cup \infty ,   \quad v \mapsto - \frac{v}{||v||^2}.
\]

\begin{definition}[\textit{Normalized configuration space bundle}]
\label{normalizedconfigurationspacebundle}
Let $E_{s,t}$ be the configuration space bundle over $\text{Inj}(\mathbb{R}^j, \mathbb{R}^n)$ defined as in Definition \ref{defofconfigurationspacebundle}, that is the bundle 
\begin{center}
\begin{tikzpicture}[auto]
\node (c) at (3, 4) {$C_{s+t}(\mathbb{R}^n)$};
\node (b) at (-3, 2.5) {$\text{Inj}(\mathbb{R}^j, \mathbb{R}^n) \times C_{s}(\mathbb{R}^j)$};
\node (a) at (3, 2.5) {$C_{s}(\mathbb{R}^n)$};
\node (f) at (-3,4) {$E_{s,t}$};

\draw [->] (f) to node {$$} (c);
\draw [->] (f) to node {$$}  (b);

\draw [->] (c) to node {restriction} (a);
\draw [->] (b) to node {evaluation} (a);
\end{tikzpicture}.
\end{center}
Assume $s \geq 1$. Then the group $G= \mathbb{R}^j  \rtimes \mathbb{R}$ of translations and rescalings of $\mathbb{R}^j$ acts on each fiber over $A \in \text{Inj}(\mathbb{R}^j, \mathbb{R}^n)$ in a way that depends on $A$. Define the \textit{normalized configuration space bundle} $D_{s,t}$ by 
\[
D_{s,t} \coloneqq E_{s,t} / G
\]
Write $B_{s,t}$ for the typical fiber $C_{s,t}/G$ of $D_{s,t}$. Note that the dimension of  $D_{s,t}$ is equal to  $\text{dim}(E_{s,t}) - (j+1)$.
\end{definition}

Let us describe the codimension-one stratum of $B_{V} = B_{s,t}$. It is decomposed as 
\[
\bigsqcup_{\substack{T \subsetneq V \\  |T| \geq 2}} \widetilde{B}_T, 
\]
where $\widetilde{B}_T$ is described below. 
Let $D_{V} = D_{s,t}$ be the normalized configuration space bundle.  The space $\widetilde{D}_T(\mathbb{R}^j, \mathbb{R}^n)$ is defined as the pullback of the following diagram. The configuration space $\widetilde{B}_T(I)$ is the fiber  of the projection $\widetilde{D}_T(\mathbb{R}^j, \mathbb{R}^n) \rightarrow \text{Inj}(\mathbb{R}^j, \mathbb{R}^n)$ at an injective map $I$.

\begin{center}
\begin{tikzpicture}[yscale=0.6,xscale=0.6]
\node (c) at (-2, 2) {$\widetilde{D}_T(\mathbb{R}^j, \mathbb{R}^n)$}; \node (d) at (5, 2) {$D_T$}; 
\node (a) at (-2, -1) {$D_{V/T}(\mathbb{R}^j, \mathbb{R}^n)$}; \node (b) at (5, -1) {$\text{Inj}(\mathbb{R}^j, \mathbb{R}^n)$};

\draw[->] (c)  to (d);
\draw[->] (a) to node [anchor = north] {$b_{V/T}$} (b);
\draw [->] (c) to  (a);
\draw [->] (d) to  node [anchor = west] {$b_T$} (b);
\end{tikzpicture}
\end{center}
We have $\widetilde{B}_T = B_{V/T} \times B_T$.

\begin{lemma}
\label{infinitefacecontributions}
Let $\Gamma$ be an admissible plain graph. Then, the integral over any infinite face
\[
\int_{\widetilde{C}_{S \cup \infty}} \omega(\Gamma) \quad (S\subset V(\Gamma), |S| \geq 1)
\]
vanishes when $ j\geq 3$ or $j = 1$. It holds even when $j = 2$, if $g(\Gamma) \leq 3$.
\end{lemma}
\begin{proof}
First, by Proposition \ref{Double push forward} and the description of $\widetilde{C}_{S \cup \infty}$, we have
\begin{align*}
&\int_{\widetilde{C}_{S \cup \infty}} \omega(\Gamma) \\
= & \pm   \int_{C_{V \setminus S}} \omega(\Gamma_{V\setminus S})\int_{B_{S \cup \infty}} \omega(\Gamma_{S\cup \ast}), 
\end{align*}
where $\ast$ is a new vertex not in $V(\Gamma)$ and $\Gamma_{S\cup \ast} =(\Gamma \cup \ast) /((V \setminus S) \cup \ast)$. See Figure \ref{infinitefacegraph}.  The right integral $\int_{B_{S \cup \infty}} \omega(\Gamma_{S\cup \ast})$ gives a form on a point. Its degree is
\begin{align*}
&(n-1)|E_{\theta}(\Gamma_{S\cup \ast})| + (j-1)|E_{\eta}(\Gamma_{S\cup \ast})| - n|W(\Gamma_{S\cup \ast})| -j |B(\Gamma_{S\cup \ast})| +1 \\
=& k(\Gamma_{S\cup \ast})(n-j-2) + (j-1)g(\Gamma_{S\cup \ast}) + l(\Gamma_{S\cup \ast}) +1.
\end{align*}
Here, 
\begin{align*}
k(\Gamma_{S\cup \ast}) &= |E_{\theta}(\Gamma_{S\cup \ast})| - |W(\Gamma_{S\cup \ast})|,\\
g(\Gamma_{S\cup \ast}) & = |E_{\theta}(\Gamma_{S\cup \ast})| + |E_{\eta}(\Gamma_{S\cup \ast})| - |W(\Gamma_{S\cup \ast})| - |B(\Gamma_{S\cup \ast})|, \\
l (\Gamma_{S\cup \ast}) & = 2 |E_{\theta}(\Gamma_{S\cup \ast})| - 3 |W(\Gamma_{S\cup \ast})| - |B(\Gamma_{S\cup \ast})|.
\end{align*}
As a graph, the point $\ast$ is considered to have degree $0$. The remaining argument is similar to the proof of Proposition \ref{Degree-zero elements prop}. We show $\int_{B_{S \cup \infty}} \omega(\Gamma_{S\cup \ast}) = 0$ by showing that the degree of $\int_{B_{S \cup \infty}} \omega(\Gamma_{S\cup \ast}) $ must be positive if it survives after symmetries and rescalings. 
If $|S| =  1$, since the original graph is admissible, we may assume that $\Gamma_{S\cup \ast}$ is a graph with one of the following vertices:
\begin{itemize}
\item [(1)] an external black vertex with at least one dashed edge
\item [(2)]  an internal black vertex with at least three solid edges
\item [(3)] a white vertex with at least three dashed edges.
\end{itemize}
See Remark \ref{allowedvertices}. 
Then, the degree is positive, since in each case we have:
\begin{itemize}
\item [(1)] $k\geq 1$, $g\geq 0$, $l\geq1$
\item [(2)] $k\geq 0$, $g\geq 2$, $l\geq-1$
\item [(3)] $k\geq 2$, $g\geq 2$, $l\geq3$.
\end{itemize}
If $|S| \geq 2$, we can assume that the vertices of $\Gamma_{S\cup \ast}$ of negative defect are only internal black vertices of valence $\geq 3$, which are possible when $j\neq 1$. 
(Recall that each internal black vertex contributes $-1$ to the defect, and it may be assumed to have at least three solid edges.)
Let $L$ be the number of these internal black vertices. Assume they are not cut vertices. If $g(\Gamma_{S\cup \ast}) = 0$, no internal black vertices are allowed  and the degree of $\int_{B_{S \cup \infty}} \omega(\Gamma_{S\cup \ast})$ is positive since $l \geq 0$.  If $g(\Gamma_{S\cup \ast}) \geq 1$, we have 
\[
L \leq 2(g(\Gamma_{S\cup \ast})-1),
\]
as we have seen in subsection \ref{Degree-zero elements}. Hence, if $j\geq 3$, 
\[
-(j-1)g(\Gamma_{S\cup \ast}) \leq -(j-1)(g(\Gamma_{S\cup \ast})-1) \leq -L \leq l.
\]
Consequently, the degree of $\int_{\widetilde{C}_{S \cup \infty}} \omega(\Gamma)$ is positive when $j\geq 3$. 
Consider the case $j=2$. First, assume $L < 2(g(\Gamma_{S\cup \ast})-1)$. Since $-L \leq l$, the degree of $\int_{\widetilde{C}_{S \cup \infty}} \omega(\Gamma)$ is positive if 
\[
2(g(\Gamma_{S\cup \ast})-1) \leq  g(\Gamma_{S\cup \ast}) + 1,
\]
that is, when $g(\Gamma_{S\cup \ast}) \leq  3$. 
Next, assume $L =2(g(\Gamma_{S\cup \ast})-1)$. Then, the vertex $\ast$ must have at least one dashed edge and hence $-(L-1) \leq l$.  (Otherwise, $\ast$ must have at least three solid edges because of rescaling or symmetry. Then at least one of the vertices $\ast$ and $L$ internal black vertices must be a cut vertex. )
Then, the degree $\int_{\widetilde{C}_{S \cup \infty}} \omega(\Gamma)$ is positive if 
\[
2(g(\Gamma_{S\cup \ast})-1) \leq g(\Gamma_{S\cup \ast}) + 1,
\]
that is, when $g(\Gamma_{S\cup \ast}) \leq  3$.

\begin{figure}[htpb]
   \centering
    \includegraphics [width =7cm] {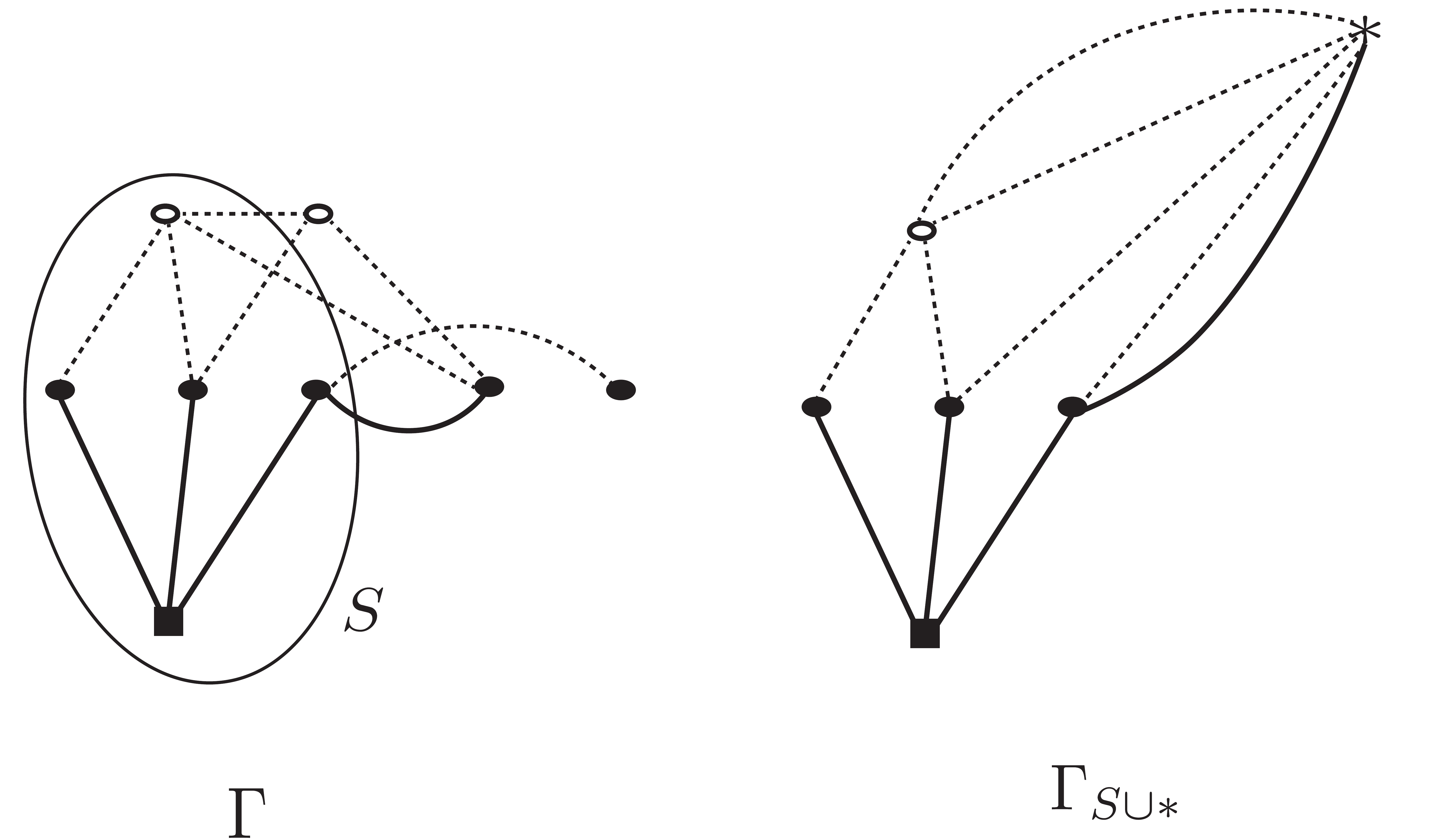}
    \caption{Example of $\Gamma_{S\cup \ast}$. $g(\Gamma_{S\cup \ast})=5$. $k(\Gamma_{S\cup \ast})=5$. $l(\Gamma_{S\cup \ast})=5$.}
    \label{infinitefacegraph}
\end{figure}
\end{proof}

 \subsection{The infinitesimal  integrals $J: A_{n,j} \rightarrow \Omega_{dR}^{\ast}\text{Inj}(\mathbb{R}^j, \mathbb{R}^n)$}
Recall that $\text{Inj}(\mathbb{R}^j, \mathbb{R}^n)$ is the space of  linear injective maps from $\mathbb{R}^j$ to $\mathbb{R}^n$.
We define the map of dg algebras $J: A_{n,j}  \rightarrow \Omega_{dR}^{\ast} \text{Inj}(\mathbb{R}^j, \mathbb{R}^n)$ by configuration space integrals. 

Let $\Gamma$ be a generalized plain graph with $s$ black vertices and $t$ white vertices. Assume $\Gamma$ has no loop edges. Obviously, the direction map $P_{\Gamma}$ from the bundle $E_{s,t}$ over $\text{Inj}(\mathbb{R}^j, \mathbb{R}^n)$, factors through the normalized bundle $D_{s,t}$. So we have 
\[
P_{\Gamma}: D_{s,t} \longrightarrow (S^{j-1})^{\times |E_{\eta}(\Gamma)|} \times (S^{n-1})^{\times |E_{\theta}(\Gamma)|}.
\]
By pulling back standard volume forms of spheres, we have the differential form $\omega(\Gamma)$ on $D_{s,t}$.

\begin{definition}
The integral $J: A_{n,j} \rightarrow \Omega_{dR}^{\ast} \text{Inj}(\mathbb{R}^j, \mathbb{R}^n)$ is defined by the fiber integral
\[
J(\Gamma) = \pi_{\ast} \omega(\Gamma) = \int_{B_{s,t}} \omega(\Gamma).
\]
for a connected graph $\Gamma$. We set $J(\Gamma) = 0$, if $\Gamma$ has a loop edge. The correspondence $J$ is extended to $A$ by setting
\[
J(\Gamma \cdot \Gamma^{\prime}) = J(\Gamma) \wedge J(\Gamma^{\prime})
\]
\end{definition}

Recall that the algebra $A_{n,j}$ has several relations. See Definition \ref{defofA}.

\begin{prop}
\label{welldefofJ}
The integral $J$ is well-defined. 
\end{prop}

\begin{proof}
First, observe that the relations on rescalings and symmetries are defined so that they reflect rescalings and symmetries of configuration space integrals. 
For the rescaling relations, we consider the rescaling of configuration spaces that acts on one of the two components obtained by removing the cut vertex. 
For the symmetry relations, we consider the central symmetry of configuration spaces with respect to the middle point between $v_1$ and $v_2$.
Thus, the integrals $J$ vanish on these relations. 

Let $\Gamma$ include no black vertex. Suppose the valency of each white vertex is greater than or equal to three. Then, the form $J(\Gamma) = \int_{B_{s,t}} \omega(\Gamma) $ is given by a pullback of a form of positive degree on a single point. In fact, if the order of $\Gamma$ is $k$, the degree must be greater than or equal to $ k(n-3)+(n+1)$.
Since $n \geq 3$, we have $J(\Gamma) = 0$  for a dimensional reason. 

Similarly, let $\Gamma$ include no external black vertices and no white vertices. Suppose the valency of each internal black vertex is greater than or equal to three. Then, if $j \geq 3$, $J(\Gamma) =0$ for the same reason. Moreover, if $j = 2$, the integral $J(\Gamma)$ of this graph vanishes by Kontsevich's argument in \cite[Theorem 6.5] {Kon 3}.
\end{proof}

\begin{prop}
The integral $J: A_{n,j} \rightarrow \Omega_{dR}^{\ast} \text{Inj}(\mathbb{R}^j, \mathbb{R}^n)$ is a map of dg algebras. 
\end{prop}

\begin{proof}
By definition, $J$ preserves products. We show that $J$ is a cochain map. Let $\Gamma$ be a connected plain graph. Then we have 
\begin{align*}
d J(\Gamma) &= (-1)^{|\Gamma|+1}\sum_{\substack{S \subsetneq V(\Gamma) \\ |S| \geq 2}} \int_{\widetilde{B}_S} \omega(\Gamma) \\
& = \sigma(S) \sum_{\substack{S \subsetneq V(\Gamma) \\ |S| \geq 2}} \int_{B_{V/S}} \omega(\Gamma/S) \int_{B_S} \omega(\Gamma_S) \\
& = \sigma(S) \sum_{\substack{S \subsetneq V(\Gamma) \\ |S| \geq 2}} J(\Gamma/S) J(\Gamma_S) \\
& = \sum_{\substack{S \subsetneq V(\Gamma) \\ |S| \geq 2}} J(d_S \Gamma) = J(d \Gamma).
\end{align*}
\end{proof}

\subsection{The iterated integral $\overline{J}: A_{n,j} \otimes_{\tau} BA_{n,j}   \rightarrow \Omega_{dR}^{\ast} P(\text{Inj}(\mathbb{R}^j, \mathbb{R}^n), \bullet, \iota)$}
\label{subsection iterated integrals}

Write $P(\text{Inj}(\mathbb{R}^j, \mathbb{R}^n), \bullet, \iota)$ for the space of paths of $\text{Inj}(\mathbb{R}^j, \mathbb{R}^n)$  whose terminal point is fixed at the standard linear injective map $\iota: \mathbb{R}^j \subset \mathbb{R}^n$. 

\begin{definition}
The integral $\overline{J}: A_{n,j} \otimes_{\tau} BA_{n,j} \rightarrow \Omega_{dR}^{\ast} P(\text{Inj}(\mathbb{R}^j, \mathbb{R}^n), \bullet, \iota)$ is defined by the iterated integral
\begin{align*}
a_0 [s a_1 | s a_2 | \dots | s a_k] &\mapsto J(a_0)  \int  J(a_1)J(a_2) \dots J(a_k) \\
& =  \pm J(a_0) \int _{\Delta_k} ev^{\ast}(J(a_1)\wedge J(a_2)\wedge\dots\wedge J(a_k)) \\
\end{align*}
Here, $\Delta^k$ is the $k$-simplex
\[
\Delta^k = \{0 \leq t_1 \leq t_2 \leq  \dots \leq t_k \leq 1\}
\]
and $ev$ is the evaluation map
\[
ev: P(\text{Inj}(\mathbb{R}^j, \mathbb{R}^n), \bullet, \iota) \times \Delta^k \rightarrow (\text{Inj}(\mathbb{R}^j, \mathbb{R}^n))^{\times k}.
\]
The sign $ \pm $ is taken so that
\[
\pm J(a_1)\wedge J(a_2)\wedge\dots\wedge J(a_k) = \iota_{t_1} J(a_1)\wedge \iota_{t_2} J(a_2)\wedge\dots\wedge \iota_{t_k} J(a_k)\ (dt_1 \wedge \dots \wedge dt_k), 
\]
where $\iota_{t_i}$ is the interior product by the vector field $\frac{\partial}{\partial t_i}$. 
\end{definition}

\begin{rem}
Let $\mathcal{A}_{n,j}$ be the subalgebra of $\Omega_{dR}^{\ast}(\text{Inj}(\mathbb{R}^j, \mathbb{R}^n))$ defined as the image of $J: A_{n,j} \rightarrow \Omega_{dR}^{\ast}(\text{Inj}(\mathbb{R}^j, \mathbb{R}^n))$. Then, the above map $\overline{J}$ is the composition of 
\[
J \otimes BJ : A_{n,j} \otimes_{\tau} B A_{n,j} \rightarrow \mathcal{A}_{n,j} \otimes_{\tau} B \mathcal{A}_{n,j}
\]
with the usual iterated integral
\[
\mathcal{A}_{n,j} \otimes_{\tau} B \mathcal{A}_{n,j} \rightarrow P(\text{Inj}(\mathbb{R}^j, \mathbb{R}^n), \bullet, \iota).
\]
\end{rem}

The following proposition is a fundamental result of Chen's iterated integrals (see Proposition \ref{factiteratedintegral1}). Recall that when $j \geq 3$ or $j=1$, we have $(A _{n,j})^0 = \mathbb{R}$ and the augmentation ideal $\overline{A}_{n,j}$ of $A _{n,j}$ is equal to the positive degree truncation $(A _{n,j})^{>0 }$. See Corollary \ref{cordegree0}.
\begin{prop}
\label{cochainmapfromZ}
If $j \geq 3$  or $j=1$, 
\[
\overline{J}: Z _{n,j} =  A_{n,j} \otimes_{\tau} BA _{n,j} \rightarrow \Omega_{dR}^{\ast} P(\text{Inj}(\mathbb{R}^j, \mathbb{R}^n), \bullet, \iota)
\]
 is a map of dg algebras. If $j = 2$, the restriction of $\overline{J} $ to $Z_{n,j}(g \leq 3)$ gives a cochain map. 
\end{prop}

\subsection{The modified configuration space integrals $\overline{I}: \text{DGC}_{n,j} \rightarrow \Omega_{dR}^{\ast}(\overline{\mathcal{K}}_{n,j})$}

We define the integrals $\overline{I}: \text{DGC}_{n,j} \rightarrow \Omega_{dR}^{\ast}(\overline{\mathcal{K}}_{n,j})$ by combining the configuration space integral $I$ of the plain parts and the iterated integral 
$\overline{J}$ of the decoration parts. We show that $\overline{I}$ is a cochain map (Theorem \ref{main theorem 1 on hidden faces}). 

By definition of the space $\overline{\mathcal{K}}_{n,j}$ of long embeddings modulo immersions, we have a map
\[
\overline{\mathcal{K}}_{n,j} \rightarrow P(\text{Imm}(\mathbb{R}^j, \mathbb{R}^n), \bullet, \iota).
\]
Let $\Gamma$ be a decorated graph with $s$ black vertices and $t$ white vertices. Then, by assigning the derivative at the $i$th $(i = 1, \dots, |B_{e}(P(\Gamma))|)$ external black vertex, we have a map
\[
Q_i : E_{s,t}  \rightarrow P(I_{n,j}) = P(\text{Inj}(\mathbb{R}^j, \mathbb{R}^n), \bullet, \iota).
\]
for $(i = 1, \dots, |B_{e}(P(\Gamma))|)$. We combine the initial direction map associated with the plain part $P(\Gamma)$ with the maps $Q_i \ (i = 1, \dots |B_{e}(P(\Gamma))|)$, and we have a map
\[
P_\Gamma: E_{s,t} \longrightarrow  (S^{j-1})^{\times |E_{\eta}(P(\Gamma))|} \times  (S^{n-1})^{\times |E_{\theta}(P(\Gamma))|} \times  (P(I_{n,j}))^{\times |B_{e}(P(\Gamma))|}.
\]
 So we have the following diagram.

\centering
\begin{tikzpicture}[yscale=0.6,xscale=0.6]
\centering
\node (c) at (2, 4) {$E_{s,t}$}; \node (d) at (14, 4) {$(S^{j-1})^{\times |E_{\eta}(P(\Gamma))|} \times  (S^{n-1})^{\times |E_{\theta}(P(\Gamma))|} \times  (P(I_{n,j}))^{\times |B_{e}(P(\Gamma))|}$}; 
\node (a) at (2, 0) {$\overline{\mathcal{K}}_{n,j}$};
\node (f) at (-2,4) {$C_{s,t}$};

\draw[->] (c)  to  node [anchor = south] {$P_{\Gamma}$}(d);
\draw [->] (c) to node [anchor = east] {$\pi$} (a);
\draw [->] (f) to (c);
\end{tikzpicture}

\begin{definition}
Let $\Gamma$ be a decorated graph with $s$ black vertices and $t$ white vertices. We define the configuration space integral associated with $\Gamma$ by 
\[
\overline{I}(\Gamma) = \pi_{ \ast} \omega(\Gamma) = \int_{C_{s,t}} \omega(\Gamma).
\]
Here the form  $\omega(\Gamma)$ is defined by 
\[
\omega(\Gamma) = P_{\Gamma}^{\ast} (\bigwedge \omega_{j-1} \wedge \bigwedge \omega_{n-1} \wedge \bigwedge \overline{J}(D_i(\Gamma))),
\]
where $D_i(\Gamma)$ is the $i$-th decoration of the decorated graph $\Gamma$.
\end{definition}

\begin{theorem}[Theorem \ref{main theorem 1 on hidden faces}]
\label{main theorem 1 on hidden faces2}
If $j \geq 3$, $\overline{I}$ is a cochain map. If $j = 2$, the restriction of $\overline{I} $ to $DGC_{n,j}(g \leq 3)$ gives a cochain map.
\end{theorem}

\begin{proof}
Recall that 
\begin{align*}
d \overline{I} (\Gamma) &= \int_{\overline{C}_{s,t}} d \omega (\Gamma) + (-1)^{|\Gamma|+1} \int_{\partial \overline{C}_{s,t}} \omega (\Gamma)
\end{align*}
Here, 
\begin{align*}
(-1)^{|\Gamma|+1} \int_{\partial \overline{C}_{s,t}} \omega (\Gamma) 
&= (-1)^{|\Gamma|+1} \sum_S \int_{\tilde{C}_S} \omega(\Gamma) \\
&=  (-1)^{|\Gamma|+1 }\sum_S  (-1)^{deg(P(\Gamma/S))} \tau(S) \Phi(S) \int_{C_{V(P(\Gamma/S))}} \omega(\Gamma/S) \wedge (\int_{B_S} \omega(\Gamma_S)) \\ 
& = \sigma(S) \int_{C_{V(P(\Gamma/S))}} \omega(\Gamma/S \bullet \Gamma_S)\\
& = \overline{I}(d_H(\Gamma)).
\end{align*}
On the other hand, by Proposition \ref{cochainmapfromZ}, $d \omega (\Gamma) = \omega(d_V \Gamma)$. 
So we have
\begin{align*}
d \overline{I} (\Gamma)
& = \overline{I}(d_V \Gamma) + \overline{I}(d_H \Gamma) \\
& = \overline{I}(d \Gamma).
\end{align*}
In the above, we use the fact that contributions of infinite faces vanish. This fact can be shown similarly to Lemma \ref{infinitefacecontributions}, considering that all non-trivial decorations have positive degree.
\end{proof}

\begin{lemma}
\label{infinitefacecontributionsofdecoratedgraphs}
Let $\Gamma$ be an admissible decorated graph. Then, the integral over any infinite face
\[
\int_{\widetilde{C}_{S \cup \infty}} \omega(\Gamma) \quad (S\subset V(\Gamma), |S| \geq 1)
\]
vanishes when $ j\geq 3$ or $j = 1$. It holds even when $j = 2$, if $g(\Gamma) \leq 3$.
\end{lemma}
\begin{proof}
The proof is similar to Lemma \ref{infinitefacecontributions}. By Proposition \ref{Double push forward} and the description of $\widetilde{C}_{S \cup \infty}$, we have
\begin{align*}
&\int_{\widetilde{C}_{S \cup \infty}} \omega(\Gamma) \\
= & \pm   \int_{C_{V \setminus S}} \omega(\Gamma_{V\setminus S})\int_{B_{S \cup \infty}} \omega(\Gamma_{S\cup \ast}), 
\end{align*}
where $\Gamma_{S\cup \ast} =(\Gamma \cup \ast) /((V \setminus S) \cup \ast)$.  The point $\ast$ is considered to have the trivial decoration.  
The right integral $\int_{B_{S \cup \infty}} \omega(\Gamma_{S\cup \ast})$ gives a form on a point. 
Its degree is
\begin{align*}
&(n-1)|E_{\theta}(\Gamma_{S\cup \ast})| + (j-1)|E_{\eta}(\Gamma_{S\cup \ast})| - n|W(\Gamma_{S\cup \ast})| -j |B(\Gamma_{S\cup \ast})| + \sum_{i \in S} |z_i|+1 \\
=& k(P(\Gamma_{S\cup \ast}))(n-j-2) + (j-1)g(P(\Gamma_{S\cup \ast})) + l(P(\Gamma_{S\cup \ast})) + \sum_{i \in S} |z_i| + 1.
\end{align*} 
Here, $P(\Gamma_{S\cup \ast})$ is the plain part of $\Gamma_{S\cup \ast}$ and 
\begin{align*}
k(P(\Gamma_{S\cup \ast})) &= |E_{\theta}(P(\Gamma_{S\cup \ast}))| - |W(P(\Gamma_{S\cup \ast}))|,\\
g(P(\Gamma_{S\cup \ast})) & = |E_{\theta}(P(\Gamma_{S\cup \ast}))| + |E_{\eta}(P(\Gamma_{S\cup \ast}))| - |W(P(\Gamma_{S\cup \ast}))| - |B(P(\Gamma_{S\cup \ast}))|, \\
l (P(\Gamma_{S\cup \ast})) & = 2 |E_{\theta}(P(\Gamma_{S\cup \ast}))| - 3 |W(P(\Gamma_{S\cup \ast}))| - |B(P(\Gamma_{S\cup \ast}))|.
\end{align*}
We show $\int_{B_{S \cup \infty}} \omega(\Gamma_{S\cup \ast}) = 0$ by showing that the degree of $\int_{B_{S \cup \infty}} \omega(\Gamma_{S\cup \ast}) $ must be positive if it survives after symmetries and rescalings. 
If $|S| =  1$, since the original graph is admissible, we can assume that $\Gamma_{S\cup \ast}$ is a graph with one of the following vertices.
\begin{itemize}
\item [(1)] an external black vertex with at least one dashed edge
\item [(2)] an internal black vertex with at least three solid edges
\item [(3)] a white vertex with at least three dashed edges
\item [(4)] an (external) black vertex with a non-trivial decoration.
\end{itemize}
Then, the degree is positive, since in each case we have:
\begin{itemize}
\item [(1)] $k\geq 1$, $g\geq 0$, $l\geq1$
\item [(2)] $k\geq 0$, $g\geq 2$, $l\geq-1$
\item [(3)] $k\geq 2$, $g\geq 2$, $l\geq3$
\item [(4)] $k\geq0$, $g\geq0$, $l\geq -1$, $|z_i| \geq 1$ ($z_i$: the decoration of the external black vertex).
\end{itemize}
If $|S| \geq 2$, we can assume that the vertices of $P(\Gamma_{S\cup \ast})$ which contribute negatively to  $l(P(\Gamma_{S\cup \ast})) + \sum_{i \in S} |z_i|$ are only internal black vertices of valence $\geq 3$, which are possible when $j\neq 1$. 
Let $L$ be the number of these internal black vertices. Assume they are not cut vertices. If $g(P(\Gamma_{S\cup \ast})) = 0$, no internal black vertices are allowed and the degree of $\int_{B_{S \cup \infty}} \omega(\Gamma_{S\cup \ast})$ is positive since $l \geq 0$.  If $g(P(\Gamma_{S\cup \ast})) \geq 1$, we have 
\[
L \leq 2(g(P(\Gamma_{S\cup \ast}))-1),
\]
as we have seen in subsection \ref{Degree-zero elements}. Hence, if $j\geq 3$, 
\[
-(j-1)(g(P(\Gamma_{S\cup \ast}))) \leq -(j-1)(g(P(\Gamma_{S\cup \ast}))-1) \leq -L \leq l.
\]
Consequently,  the degree of $\int_{\widetilde{C}_{S \cup \infty}} \omega(\Gamma)$ is positive when $j\geq 3$. 
Consider the case $j=2$. First, assume $L < 2(g(P(\Gamma_{S\cup \ast}))-1)$. Since $-L \leq l$, the degree of $\int_{\widetilde{C}_{S \cup \infty}} \omega(\Gamma)$ is positive if 
\[
 2(g(P(\Gamma_{S\cup \ast}))-1) \leq g(P(\Gamma_{S\cup \ast})) + 1,
\]
that is, if $g(P(\Gamma_{S\cup \ast})) \leq  3$. 
Next, assume $L =2(g(P(\Gamma_{S\cup \ast}))-1)$. Then, the vertex $\ast$ must have a dashed edge and $-(L-1)\leq l$. Then, the degree $\int_{\widetilde{C}_{S \cup \infty}} \omega(\Gamma)$ is positive if 
\[
2(g(P(\Gamma_{S\cup \ast}))-1) \leq g(P(\Gamma_{S\cup \ast})) + 1,
\]
that is, when $g(P(\Gamma_{S\cup \ast})) \leq  3$. 
\end{proof}

\begin{rem}
\label{remark to definition of the correction term}
Let $\Gamma$ be a BCR graph of defect $0$. 
Then, the correction term $\overline{c}(\Gamma)$ in \cite{Yos 1} coincides with
\[
(-1)^{|\Gamma|+1} \ \overline{I}
(
\begin{tikzpicture}[x=1pt,y=1pt,yscale=1,xscale=1, baseline=-3pt, line width =1pt]
\draw  [fill={rgb, 255:red, 0; green, 0; blue, 0 }  ,fill opacity=1 ] (0,0) circle (2);
\draw (0, -2) node [anchor = north] {$[s\Gamma]$};
\end{tikzpicture}
),
\] the configuration space integral associated with a decorated graph whose plain part is one external black vertex. Note that $r = |\Gamma| +1$ and $d = [\Gamma] (= \overline{deg}(\Gamma))$.
We have
\[
d\overline{I}
(
\begin{tikzpicture}[x=1pt,y=1pt,yscale=1,xscale=1, baseline=-3pt, line width =1pt]
\draw  [fill={rgb, 255:red, 0; green, 0; blue, 0 }  ,fill opacity=1 ] (0,0) circle (2);
\draw (0, -2) node [anchor = north] {$[s\Gamma]$};
\end{tikzpicture}
)
=
\overline{I}
(d_V
\begin{tikzpicture}[x=1pt,y=1pt,yscale=1,xscale=1, baseline=-3pt, line width =1pt]
\draw  [fill={rgb, 255:red, 0; green, 0; blue, 0 }  ,fill opacity=1 ] (0,0) circle (2);
\draw (0, -2) node [anchor = north] {$[s\Gamma]$};
\end{tikzpicture}
)
=
-\overline{I}
(
\begin{tikzpicture}[x=1pt,y=1pt,yscale=1,xscale=1, baseline=-3pt, line width =1pt]
\draw  [fill={rgb, 255:red, 0; green, 0; blue, 0 }  ,fill opacity=1 ] (0,0) circle (2);
\draw (0, -2) node [anchor = north] {$[s(d\Gamma)]$};
\end{tikzpicture}
)
-
\overline{I}
(
\begin{tikzpicture}[x=1pt,y=1pt,yscale=1,xscale=1, baseline=-3pt, line width =1pt]
\draw  [fill={rgb, 255:red, 0; green, 0; blue, 0 }  ,fill opacity=1 ] (0,0) circle (2);
\draw (0, -2) node [anchor = north] {$\Gamma$};
\end{tikzpicture}
), 
\]
when all the hidden face contributions vanish. On the other hand, 
\[
(-1)^{|\Gamma|} \ \overline{I}
(
\begin{tikzpicture}[x=1pt,y=1pt,yscale=1,xscale=1, baseline=-3pt, line width =1pt]
\draw  [fill={rgb, 255:red, 0; green, 0; blue, 0 }  ,fill opacity=1 ] (0,0) circle (2);
\draw (0, -2) node [anchor = north] {$\Gamma$};
\end{tikzpicture}
)
\]
is canceled by the contribution of the anomalous face in $d (r^{\ast}I(\Gamma))$. 
\end{rem}

\begin{example}
Let $n$ and $j$ be odd, and let $j\geq 3$. 
Let $H = \sum_i \Gamma_i$ be the $2$-loop BCR graph cocycle of order 3 in \cite {Yos 1}. 
Then, the linear combination
\[
H^{\prime} = \sum_i \Gamma_i +  \sum_i (-1)^{|\Gamma_i|+1}
\begin{tikzpicture}[x=1pt,y=1pt,yscale=1,xscale=1, baseline=-3pt, line width =1pt]
\draw  [fill={rgb, 255:red, 0; green, 0; blue, 0 }  ,fill opacity=1 ] (0,0) circle (2);
\draw (0, -2) node [anchor = north] {[$s \Gamma_i]$};
\end{tikzpicture}
\]
is a cocycle of the decorated graph complex. Hence $\overline{I}(H^{\prime})$ gives a cocycle of $\overline{\mathcal{K}}_{n,j}$.  Note that $\overline{I}(H^{\prime})$ coincides with $r^{\ast} I(H) + \overline{c}(H)$. 
\end{example}

\end{document}